\documentclass[a4paper,12pt]{article}
\usepackage{mathrsfs,amsthm,graphicx,color,verbatim,bbm,amsmath,amsfonts,amssymb,newclude,nicefrac,amsfonts,graphicx,enumerate,url,hyperref}
\usepackage[latin1]{inputenc}

\newtheorem{lemma}{Lemma}[section]
\newtheorem{theorem}[lemma]{Theorem}
\newtheorem{prop}[lemma]{Proposition}
\newtheorem{corollary}[lemma]{Corollary}

\theoremstyle{lemma}
\newtheorem*{remark}{Remark}

\providecommand{\N}{{\ensuremath{\mathbbm{N}}}}
\providecommand{\R}{{\ensuremath{\mathbbm{R}}}}
\providecommand{\E}{{\ensuremath{\mathbb{E}}}}
\renewcommand{\P}{{\ensuremath{\mathbb{P}}}}
\providecommand{\1}{{\ensuremath{\mathbbm{1}}}}

\newcommand{\diffns}[1]{\mathrm{d}#1}

\newcommand{\lpn}[3]{L^{#1}(#2;#3)}
\newcommand{\smallsum}{\textstyle\sum}

\newcommand{\perturb}{\Xi}

\title{Existence, uniqueness, and regularity for stochastic evolution equations with irregular initial values}

\author{Adam Andersson, Arnulf Jentzen, and Ryan Kurniawan}

\begin{document}

\maketitle

\begin{abstract}
In this article we develop a framework for studying parabolic semilinear stochastic evolution equations (SEEs) 
with singularities in the initial condition and singularities at the initial time of the time-dependent coefficients 
of the considered SEE. We use this framework to establish existence, uniqueness, and regularity results for mild 
solutions of parabolic semilinear SEEs with singularities at the initial time. 
We also provide several counterexample 
SEEs that illustrate the optimality of our results.
\end{abstract}

\tableofcontents

\section{Introduction}
\label{sec:intro}

There are a number of existence, uniqueness, and regularity results 
for mild solutions of semilinear stochastic evolution equations (SEEs) 
in the literature;
see, e.g., 
\cite{dz92,
DaPratoZabczyk1996,
b97b,
VanNeervenVeraarWeis2008,
JentzenKloeden2011,
jr12,
KruseLarsson2012,
VanNeervenVeraarWeis2012}
and the references mentioned therein.
In this work we extend the above cited results 
by adding
singularities in the initial condition and
by introducing singularities
at the initial time of the time-dependent coefficients of the considered SEE;
cf., e.g., Chen \& Dalang~\cite{Chen2014,chen2015}. 
To be more specific, in the first main result of this work (see Proposition~\ref{prop:perturbation_estimate} below)
we establish a general perturbation estimate (see~\eqref{eq:intro.perturb} in this introductory section below) for a general class of 
stochastic processes which allows us to derive a priori bounds (see, e.g., \eqref{eq:intro.SEE.regularity} in this introductory section below) 
for solutions and numerical approximations of SEEs with singularities at the initial time. 
This perturbation estimate, in turn, is used to prove the second main result of this article (see Theorem~\ref{thm:existence_uniqueness} below) 
which establishes existence, uniqueness, and regularity properties for solutions of SEEs with singularities at the initial time 
(see~\eqref{eq:intro.SEE} and~\eqref{eq:intro.SEE.regularity} in this introductory section below). 
As an application of our perturbation estimate and this second main result of our article, we reveal a regularity barrier (see~\eqref{eq:regularity_barrier} in this introductory section below) for the initial condition of SEEs under which the considered SEE has unique solutions which are Lipschitz continuous with respect to initial values (see Corollary~\ref{cor:rough_initial} below). 
By means of several counterexamples (see Proposition~\ref{prop:counterexampleI}, Proposition~\ref{prop:counterexampleII}, and Proposition~\ref{prop:counterexample2} below) we also demonstrate that this regularity barrier can in general not essentially be improved (cf. \eqref{eq:intro.Lip} and~\eqref{eq:initial.instability} in this introductory section below). 
We illustrate the above findings in the case of possibly nonlinear stochastic heat equations on an interval such as the continuous version of the parabolic Anderson model on an interval (cf. Corollary~\ref{cor:periodic_boundary}, Proposition~\ref{prop:counterexampleI}, and Proposition~\ref{prop:positive.example} below).
Existence, uniqueness, and regularity results for possibly nonlinear stochastic heat equations on the whole real line with rough initial values, 
that is, signed Borel measures with exponentially growing tails over $\R$ as initial values can be found in Chen \& Dalang~\cite{Chen2014,chen2015} 
(see Theorem~2.4 in Chen \& Dalang~\cite{chen2015} for an existence and uniqueness result and a priori estimates 
and see Theorem~3.1 in Chen \& Dalang~\cite{Chen2014} for a H\"{o}lder regularity result). 
Moreover, Proposition~2.11 in Chen \& Dalang~\cite{chen2015} disproves the existence of a solution of the considered stochastic heat equation 
in the case of a specific rough initial value, that is, the derivative of the Dirac delta measure at zero as the initial value.

To illustrate the results of this article in more details,
we assume the following setting throughout 
this introductory section. 
Let
$
  \left(
    H,
    \left\| \cdot \right\|_H,
    \left< \cdot, \cdot \right>_H
  \right)
$ 
and 
$
  \left(
    U,
    \left\| \cdot \right\|_U ,
    \left< \cdot, \cdot \right>_U
  \right)
$
be nontrivial separable $ \R $-Hilbert spaces,
let
$ T \in (0,\infty) $,
$ \eta \in \R $,
$
  p \in [ 2 , \infty )
$,
$ \alpha \in [ 0, 1 ) $,
$ \hat{ \alpha } \in ( -\infty, 1 ) $, 
$
  \beta \in [ 0, \nicefrac{ 1 }{ 2 } )
$,
$
  \hat{ \beta } \in ( -\infty, \nicefrac{ 1 }{ 2 } )
$,
$ L_0, \hat{L}_0, L_1, \hat{L}_1 \in [0,\infty) $, 
$ \kappa = \1_{(0,\infty)}(L_1) $
satisfy
$
  \kappa \,
  (\alpha + \hat{ \alpha })
  < \nicefrac{ 3 }{ 2 }
$, 
let
$
  ( \Omega , \mathcal{F}, \P, ( \mathcal{F}_t )_{ t \in [0,T] } )
$
be a stochastic basis,
let
$
  ( W_t )_{ t \in [0,T] }
$
be an $\operatorname{Id}_U$-cylindrical $ ( \mathcal{F}_t )_{ t \in [0,T] } $-Wiener process,
let
$
  A \colon D(A)
  \subseteq
  H \rightarrow H
$
be a generator of a strongly continuous analytic semigroup
with 
$
  \operatorname{spectrum}( A )
  \subseteq
  \{
    z \in \mathbb{C}
    \colon
    \operatorname{Re}( z ) < \eta
  \}
$,
let
$
  (
    H_r
    ,
    \left\| \cdot \right\|_{ H_r }
    ,
    \left< \cdot , \cdot \right>_{ H_r }
  )
$,
$ r \in \R $,
be a family of interpolation spaces associated to
$
  \eta - A
$
(cf., e.g., \cite[Section~3.7]{sy02}),
let
$ \mathbf{F} \colon [0,T] \times \Omega \times H \to H_{-\alpha} $ 
be a 
$
(\mathrm{Pred}( ( \mathcal{F}_t )_{ t \in [0,T] } ) \otimes \mathcal{B}( H ))
$/$\mathcal{B}( H_{ - \alpha } )$-measurable mapping, 
let
$ \mathbf{B} \colon [0,T] \times \Omega \times H \to HS( U, H_{ - \beta } ) $ 
be a 
$
(\mathrm{Pred}( ( \mathcal{F}_t )_{ t \in [0,T] } ) 
$
$
\otimes \mathcal{B}( H ))
$/$\mathcal{B}( HS( U, H_{ - \beta } ) )$-measurable mapping, 
assume for all $ t \in (0,T] $, $ X, Y \in \mathcal{L}^p( \P; H ) $
that
\begin{align}
\label{eq:intro.singular.F}
  \|
    \mathbf{F}(t,X) - \mathbf{F}(t,Y)
  \|_{
    L^p( \P; H_{ - \alpha } )
  }
  \leq
  L_0 \,
  \| X - Y \|_{
    L^p( \P; H )
  }
  ,\,\,
&
  \|
    \mathbf{F}( t , 0 )
  \|_{
    L^p(\P;H_{ - \alpha } )
  }
  \leq
    \hat{L}_0
    \,
    t^{ - \hat{\alpha} }
  ,
\\
\label{eq:intro.singular.B}
  \|
    \mathbf{B}(t,X) - \mathbf{B}(t,Y)
  \|_{
    L^p( \P; HS( U, H_{ - \beta } ) )
  }
  \leq
  L_1
  \| X - Y \|_{
    L^p( \P; H )
  } ,\,\,
&
  \| \mathbf{B}(t,0) \|_{
    L^p( \P; HS( U, H_{ - \beta } ) )
  }
\leq
    \hat{L}_1
    \,
    t^{ - \hat{ \beta } }
  ,
\end{align}
for every 
$ a, b \in (-\infty,1) $
let 
$
  \mathrm{E}_{ a, b } \colon [0,\infty) \to [0,\infty)
$
be the function which satisfies for all
$
  x \in [0,\infty)
$
that
$
  \mathrm{E}_{ a, b }[ x ]
  =
  1
  +
  \sum_{ n = 1 }^{ \infty }
  x^n
  \prod_{ k = 0 }^{ n - 1 }
  \int_0^1
  t^{
    -b
  }
  \,
  ( 1 - t )^{
    k(1-b) - a
  }
  \, dt
$
(cf., e.g., \cite[Chapter~7]{h81}), 
for every $ r \in [0,1] $
let 
$ \chi_r \in ( 0, \infty ) $
be the real number given by
$
  \chi_r
  =
  \sup_{ t \in (0,T] }
    t^r
    \,
    \|
      ( \eta - A )^r
      e^{ t A }
    \|_{ L( H ) }
$
(see, e.g.,  \cite[Lemma~11.36]{rr93}),
and 
for every
$ 
  \lambda \in 
  ( - \infty, 
    \frac{
      1
    }{ 2 }
    [
      1 +
      \1_{ \{ 0 \} }( L_1 )
    ] 
  )
$
let 
$
  \Theta_{ \lambda } \in [0,\infty)
$
be the real number given by
\begin{equation}
\label{eq:BigTheta.intro}
\begin{split}
&
  \Theta_{ \lambda }
  =
 2^{ \kappa / 2 }
    \left|
    \mathrm{E}_{
      (1 + \kappa) \lambda
      ,
      \max\{ \alpha, 2 \beta \kappa \}
    }
    \!\left[
      \left|
        \tfrac{
          \chi_\alpha 
          \, 
          L_0 
          \, 
          2^{ \kappa / 2 } 
          \, 
          T^{ ( 1 - \alpha ) }
        }{ 
          ( 1 - \alpha )^{ \kappa / 2 } 
        }
        +
        \chi_\beta \, 
        L_1 
        \sqrt{ 
          p \, (p-1) \, T^{ ( 1 - 2 \beta ) } 
        }
      \right|^{ ( 1 + \kappa ) }
    \right]
    \right|^{ ( 2 - \kappa ) / 2 }
  .
\end{split}
\end{equation}
In displays~\eqref{eq:intro.perturb}--\eqref{eq:initial.instability} 
below we illustrate the above framework through several examples
and applications.

Our first result is a suitable \emph{perturbation estimate} 
for predictable stochastic processes. 
More formally, in Proposition~\ref{prop:perturbation_estimate} 
below we prove for all 
$ \delta \in \R $, 
$
  \lambda \in 
  ( - \infty, 
    \frac{
      1
    }{ 2 }
    [
      1 +
      \1_{ \{ 0 \} }( L_1 )
    ] 
  )
$ 
and all $ ( \mathcal{F}_t )_{ t \in [0,T] } $-predictable
stochastic processes 
$ Y^1 , Y^2 \colon [0,T] \times \Omega \to H_{\delta} $ 
with 
\begin{equation}
\label{eq:perturb.condition}
   \cup_{ k \in \{1,2\} }
   Y^k( (0,T] \times \Omega )
   \subseteq H
   \quad
   \text{and}
   \quad
   \limsup_{ 
     r 
     \nearrow  
     \frac{
       1
     }{ 2 }
     [
       1 +
       \1_{ \{ 0 \} }( L_1 )
     ]  
   }
   \max_{ k \in \{1,2\} }
   \sup_{ t \in ( 0 , T ] }
   t^{ r }
   \,
   \|
     Y_t^k
   \|_{
     L^p ( \P ; H )
   }
   < \infty
\end{equation}
that 
$
  \forall \, t \in [0,T] \colon
  \P
  \big(
  \sum^2_{ k=1 }
$
$
  \int_0^t
  \|
    e^{  (t - s) A}
    \mathbf{F}( s , Y^k_s )
  \|_H
  +
  \|
    e^{  (t - s) A}
    \mathbf{B}( s , Y^k_s )
  \|_{ HS( U, H ) }^2
  \,
  \diffns s
  < \infty
  \big)
  =1
$ 
and 
\begin{multline}
\label{eq:intro.perturb}
  \sup_{
    t \in (0,T]
  }
  \left[
  t^{ \lambda }
  \,
  \|
    Y_t^1 - Y_t^2
  \|_{
    L^p ( \P ; H )
  }
  \right]
\leq
  \sup_{ t \in ( 0 , T ] }
  \bigg[
  t^{ \lambda }
  \,
  \bigg\|
    Y_t^1
    -
    \smallint_0^t
      e^{(t-s)A}
      \mathbf{F}(s,Y_s^1)
    \, \diffns s
    -
    \smallint_0^t
      e^{(t-s)A}
      \mathbf{B}(s,Y_s^1)
    \, \diffns W_s
\\
    +
    \smallint_0^t
      e^{(t-s)A}
      \mathbf{F}(s,Y_s^2)
    \, \diffns s
    +
    \smallint_0^t
      e^{(t-s)A}
      \mathbf{B}(s,Y_s^2)
    \, \diffns W_s
    -
    Y_t^2
  \bigg\|_{
    L^p ( \P ; H )
  }
  \bigg]
  \,
  \Theta_{ \lambda }
  .
\end{multline}
We note that the right hand side of \eqref{eq:intro.perturb}
might be infinite.
Moreover,
we would like to emphasize that $ Y^1 $ and $ Y^2 $ 
in \eqref{eq:intro.perturb}
are arbitrary $ ( \mathcal{F}_t )_{ t \in [0,T] } $-predictable
stochastic processes which satisfy~\eqref{eq:perturb.condition}
and, in particular, we emphasize that $ Y^1 $ and $ Y^2 $ 
do not need to be solution processes of some SEEs.
Estimate~\eqref{eq:intro.perturb} follows from an appropriate application 
of a generalized Gronwall-type inequality 
(see the proof of Proposition~\ref{prop:perturbation_estimate} below for details).

We use inequality~\eqref{eq:intro.perturb} to establish an existence, 
uniqueness, and regularity result for SEEs 
with singularities at the initial time. 
More precisely, in Theorem~\ref{thm:existence_uniqueness} below 
we prove that for all 
$
  \delta
  \in
  \big(
    -\infty
    ,
    \frac{
      1
    }{ 2 }
    [
      1 +
      \1_{ \{ 0 \} }( L_1 )
    ]
  \big)
$,
$
  \lambda \in
      \big[\!
      \max\{
        \delta ,
        \alpha + \hat{ \alpha } - 1 ,
        \beta + \hat{ \beta } - \nicefrac{ 1 }{ 2 }
      \}
    ,
    \frac{
      1
    }{ 2 }
    [
      1 +
      \1_{ \{ 0 \} }( L_1 )
    ]
  \big)
$, 
$ \xi \in \mathcal{L}^p( \P|_{\mathcal{F}_0}; H_{-\max\{\delta,0\}} ) $
with  
$
  \sup_{ t\in (0,T] }
  t^{ \delta}\,
  \| e^{tA} \xi \|_{ \lpn{p}{\P}{H} }
$
$
  < \infty
$ 
it holds 
(i) that there exists an up-to-modifications unique 
$ (\mathcal{F}_t)_{ t \in [0,T] } $-predictable stochastic process 
$ X \colon [0,T] \times \Omega \to H_{-\max\{\delta,0\}} $
which satisfies for all $ t \in [0,T] $
that 
$ X( (0,T] \times \Omega ) \subseteq H $, 
that 
$
  \sup_{ s \in (0,T] }
  s^\lambda\,
  \|X_s\|_{ \lpn{p}{\P}{H} }
  < \infty
$, 
that 
$
  \P\big(
  \int_0^t
  \|
    e^{ (t - s) A}
    \mathbf{F}( s , X_s )
  \|_H
  +
  \|
    e^{ (t - s) A}
    \mathbf{B}( s , X_s )
  \|_{ HS( U, H ) }^2
  \,
  \diffns s
  < \infty
  \big)
  =1
$, 
and $ \P $-a.s.\ that
\begin{equation}
\label{eq:intro.SEE}
  X_t
  =
  e^{t A } \xi
  +
  \int_0^t
    e^{  (t - s) A}
    \mathbf{F}( s , X_s )
    \,
  \diffns s
  +
  \int_0^t
    e^{  (t - s) A}
    \mathbf{B}( s , X_s )
    \,
  \diffns W_s
\end{equation}
and (ii) that
\begin{equation}
\label{eq:intro.SEE.regularity}
\begin{split}
&
  \sup_{ t \in (0,T] }
  \left[
    t^{ \lambda }
    \left\|
      X_t
    \right\|_{
      L^p( \P; H )
    }
  \right]
  \leq
  T^\lambda\,
  \Theta_\lambda
\\&\quad\cdot
  \bigg[
  \tfrac{
      \sup_{ t \in (0,T] }
      (
      t^\delta
      \|
        e^{tA}
        \xi
      \|_{
        L^p( \P; H )
      }
      )
  }{
    T^\delta
  }
+
    \tfrac{
      \chi_{ \alpha }
      \,
      \hat{L}_0
      \,
      \mathbbm{B}(
        1 - \alpha
        ,
        1 - \hat{\alpha}
      )
    }{
      T^{
        \left(
          \alpha + \hat{\alpha} - 1
        \right)
      }
    }
    +
    \tfrac{
      \chi_{ \beta }
      \,
      \hat{L}_1
      \left|
        p \, ( p - 1 )
        \,
        \mathbbm{B}(
          1 - 2 \beta
          ,
          1 - 2 \hat{\beta}
        )
      \right|^{ \nicefrac{ 1 }{ 2 } }
    }{
      \sqrt{2}\,
      T^{
        (
          \beta + \hat{ \beta } - 1/2
        )
      }
    }
  \bigg]
  < \infty .
\end{split}
\end{equation}
We would like to point out that inequality~\eqref{eq:intro.SEE.regularity} under the generality of~\eqref{eq:intro.singular.F} and~\eqref{eq:intro.singular.B} 
is a crucial ingredient to establish essentially sharp weak convergence rates for numerical approximations of SEEs with possibly smooth initial values (see the last paragraph in this introductory section for more details).

Inequality~\eqref{eq:intro.SEE.regularity} follows
from the perturbation estimate~\eqref{eq:intro.perturb}
(with $ Y^1 = X $ and $ Y^2 = 0 $ in the notation
of \eqref{eq:intro.perturb}).
We now illustrate Theorem~\ref{thm:existence_uniqueness} 
and \eqref{eq:intro.SEE}--\eqref{eq:intro.SEE.regularity}, 
respectively, by some examples.
In particular, in Corollary~\ref{cor:rough_initial} below we prove 
by an 
application of Theorem~\ref{thm:existence_uniqueness} that for all 
$ F \in \operatorname{Lip}( H , H_{ - \alpha } ) $,
$ B \in \operatorname{Lip}( H , HS( U , H_{ - \beta } ) ) $, 
$
  \hat{ \delta }
  =
  \frac{
    1
  }{ 2 }
  \big[
    1 +
    \mathbbm{1}_{ 
      \{ 0 \} 
    }(
      | B |_{
        \operatorname{Lip}( H, HS( U, H_{ - \beta } ) )
      }
    )
  \big]
$ 
it holds 
(i) 
that there exist up-to-modifications unique 
$ ( \mathcal{F}_t )_{ t \in [0,T] } $-predictable stochastic processes
$
  X^x
  \colon
  [ 0 , T ] \times \Omega
  \to
  H_{ - \delta }
$,
$
  x
  \in
  H_{ - \delta }
$,
$
  \delta \in [ 0 , \hat{ \delta } )
$,
which fulfill
for all
$ q \in [2,\infty) $,
$
  \delta
  \in [ 0 , \hat{\delta} )
$,
$
  x \in H_{ - \delta } 
$,
$ t \in [0,T] $
that
$
  X^x
  ( ( 0 , T ] \times \Omega ) \subseteq H
$,
that
$
  \sup_{ s \in (0,T] }
  s^{
    \delta
  }
  \,
  \|
    X_s^x
  \|_{ L^q( \P ; H ) }
  < \infty
$,
and $ \P $-a.s.\ that
\begin{equation}
\label{eq:intro.SEE.initial}
\begin{split}
&
  X_t^x
  =
  e^{ t A } x
  +
  \int_0^t
    e^{ ( t - s ) A }
      F( X_s^x )
    \, \diffns s
  +
  \int_0^t
    e^{ ( t - s ) A }
      B( X_s^x )
  \, \diffns W_s
\end{split}
\end{equation}
%
%
%
%
%
and (ii) that 
\begin{equation}
\label{eq:intro.SEE.initial.regularity}
\begin{split}
&
  \forall 
  \, 
  \delta \in [ 0 , \hat{ \delta } ) ,
  \, 
  q \in [ 2 , \infty )
  \colon
  \sup_{
    \substack{
      x, y \in H_{ - \delta } ,
    \\
      x \neq y
    }
  }
  \sup_{ t \in (0,T] }
  \max\!
  \left\{
  \frac{
    t^{ \delta }
    \left\|
      X_t^x
    \right\|_{
      L^q( \P; H )
    }
  }{
    \max\{ 1, \| x \|_{ H_{ - \delta } } \}
  }
  ,
    \frac{
      t^{ \delta }
      \,
      \|
        X_t^x - X_t^y
      \|_{
        L^q( \P ; H )
      }
    }{
      \|
      x - y
      \|_{
        H_{ - \delta } 
      }
    }
  \right\}  
  < \infty
  .
\end{split}
\end{equation}
Here and below 
we denote for 
$ \R $-Banach spaces
$
  ( V, \left\| \cdot \right\|_V ) 
$
and 
$
  ( W, \left\| \cdot \right\|_W ) 
$
by
$ \operatorname{Lip}( V, W ) $
the set of all Lipschitz continuous functions from $ V $ to $ W $
and 
we denote for 
$ \R $-Banach spaces
$
  ( V, \left\| \cdot \right\|_V ) 
$
and 
$
  ( W, \left\| \cdot \right\|_W ) 
$
and a function 
$ f \in \operatorname{Lip}( V, W ) $
by
$
  \left| f \right|_{
    \operatorname{Lip}( V, W )
  } 
  \in [0,\infty)
$
the Lipschitz semi-norm associated to $f$
(see~\eqref{eq:Lip.def} in Subsection~\ref{sec:notation} below for details).
The finiteness of the second element in the set in the maximum 
in \eqref{eq:intro.SEE.initial.regularity}
follows from the perturbation estimate~\eqref{eq:intro.perturb}
(with $ Y^1 = X^x $ and $ Y^2 = X^y $
for $ x, y \in H_{ - \delta } $, $ \delta \in [0,\hat{\delta}) $
in the notation of \eqref{eq:intro.perturb})
and the finiteness of the first element in the set in the maximum
in \eqref{eq:intro.SEE.initial.regularity}
is a consequence from \eqref{eq:intro.SEE.regularity}, which, in turn,
also follows from the perturbation estimate~\eqref{eq:intro.perturb}
(see above and the proof 
of Corollary~\ref{cor:rough_initial}
for details).
Roughly speaking, 
Corollary~\ref{cor:rough_initial} 
establishes the existence of 
mild solutions of the SEE~\eqref{eq:intro.SEE.initial}
and also establishes 
the Lipschitz continuity of the solutions with respect to the initial conditions 
for any initial condition in $ H_{ - \delta } $
and any 
$
  \delta < 
  \hat{ \delta } = 
  \frac{
    1
  }{ 2 }
  \big[
    1 +
    \mathbbm{1}_{ 
      \{ 0 \} 
    }(
      | B |_{
        \operatorname{Lip}( H, HS( U, H_{ - \beta } ) )
      }
    )
  \big]
$
(see \eqref{eq:intro.SEE.initial.regularity}).
In Corollary~\ref{cor:periodic_boundary},
Proposition~\ref{prop:counterexampleI},
Proposition~\ref{prop:counterexampleII},
and Proposition~\ref{prop:counterexample2}
below we demonstrate that
the \emph{regularity barrier}
\begin{equation}
\label{eq:regularity_barrier}
  \hat{ \delta } = 
  \tfrac{
    1
  }{ 2 }
  \big[
    1 +
    \mathbbm{1}_{ 
      \{ 0 \} 
    }(
      | B |_{
        \operatorname{Lip}( H, HS( U, H_{ - \beta } ) )
      }
    )
  \big]
  =
  \begin{cases}
    \nicefrac{ 1 }{ 2 }
  &
    \colon
    B \text{ is not a constant function}
  \\
    1
  &
    \colon
    B \text{ is a constant function}
  \end{cases}
\end{equation}
for the regularity of the initial conditions 
revealed in Corollary~\ref{cor:rough_initial}
(and 
Proposition~\ref{prop:perturbation_estimate}
and 
Theorem~\ref{thm:existence_uniqueness},
respectively)
can, in general, not essentially be improved.
In particular,
Corollary~\ref{cor:periodic_boundary} and~Proposition~\ref{prop:counterexampleI} below prove
in the case where 
$ H = U = L^2( (0,1) ; \R ) $,
where 
$ \beta \in ( \nicefrac{1}{4} , \nicefrac{ 1 }{ 2 } ) $,
where 
$ A \colon D(A) \subseteq H \rightarrow H $ 
is the Laplacian with periodic boundary conditions on $ H $,
and where
$
  B \in L( H , HS( H, H_{ - \beta } ) )
$
satisfies 
$ 
  \forall \, 
  u, v \in H  
  \colon
  B( v ) u 
  = v \cdot u
$
($ B $ is not a constant function)
that it holds 
(i) that there exist up-to-modifications 
unique 
$ ( \mathcal{F}_t )_{ t \in [0,T] } $-predictable
stochastic processes
$
  X^x \colon [0,T] \times \Omega \to H_{ - \delta }
$,
$ x \in H_{ - \delta } $,
$ \delta \in [0, \nicefrac{ 1 }{ 2 } ) $,
which fulfill for all 
$ q \in [2,\infty) $,
$
  \delta
  \in [ 0 , \nicefrac{1}{2} )
$,
$
  x \in H_{ - \delta } 
$, 
$ t \in [0,T] $
that
$
  X^x
  ( ( 0 , T ] \times \Omega ) \subseteq H
$,
that
$
  \sup_{ s \in (0,T] }
  s^\delta
  \,
  \|
    X_s^x
  \|_{ L^q( \P ; H ) }
  < \infty
$, 
and $ \P $-a.s.\ that 
\begin{equation}
\label{eq:SPDE_counterexample1_intro}
\begin{split}
  X^x_t
  =
  e^{ tA } x
  +
  \int^t_0 
  e^{ (t-s)A }
  B(X^x_s)
  \,\diffns{W_s}
  ,
\end{split}
\end{equation}
(ii) that 
\begin{equation}
\label{eq:intro.Lip}
\begin{split}
&
  \forall 
  \, 
  \delta \in [0,\nicefrac{1}{2}) 
  ,
  \, 
  q \in [2,\infty),
  \,
  t \in (0,T]
  \colon
  \sup_{
    \substack{
      x, y \in H ,
    \\
      x \neq y
    }
  }
  \left[
    \frac{
      \|
        X_t^x - X_t^y
      \|_{
        L^q( \P ; H )
      }
    }{
      \|
      x - y
      \|_{
        H_{ - \delta } 
      }
    }
  \right]
  < \infty
  ,
\end{split}
\end{equation}
and (iii) that 
\begin{equation}
\label{eq:initial.instability}
  \forall 
  \, \delta \in (\nicefrac{1}{2},\infty) ,
  \, q \in [2,\infty) ,
  \,
  t \in (0,T]
  \colon
  \sup_{
    \substack{
      x, y \in H ,
    \\
      x \neq y
    }
  }
  \left[
    \frac{
      \|
        X_t^x - X_t^y
      \|_{
        L^q( \P ; H )
      }
    }{
      \|
      x - y
      \|_{
        H_{ - \delta } 
      }
    }
  \right]
  = \infty
  .
\end{equation}
The SEE~\eqref{eq:SPDE_counterexample1_intro} is sometimes referred to as
a continuous version of the \emph{parabolic Anderson model}
in the literature (see, e.g., Carmona \& Molchanov~\cite{CarmonaMolchanov1994}).
In addition, 
Proposition~\ref{prop:counterexampleI} below
\emph{disproves} 
the existence of square integrable solutions 
of the SEE~\eqref{eq:SPDE_counterexample1_intro}
with initial conditions in 
$ ( \cup_{ \delta \in \R } H_{ \delta } ) \backslash H_{ - 1 / 2 } $.
The noise in the counterexample SEE~\eqref{eq:SPDE_counterexample1_intro}
is spatially very rough and one might question whether the regularity barrier~\eqref{eq:regularity_barrier}
can be overcome in the case of more regular spatially smooth noise.
In Proposition~\ref{prop:counterexampleII} below
we answer this question to the negative by presenting another counterexample
SEE with a non-constant diffusion coefficient 
but a spatially smooth noise for which we 
disprove the existence of square integrable solutions 
with initial conditions in 
$ ( \cup_{ \delta \in \R } H_{ \delta } ) \backslash H_{ - 1 / 2 } $
(cf., however, also Proposition~\ref{prop:positive.example} below).
Proposition~\ref{prop:counterexample2} 
below also provides a further counterexample SEE
which illustrates the sharpness of the 
regularity barrier~\eqref{eq:regularity_barrier}
in the case where $ B $ is a constant function.

Proposition~\ref{prop:perturbation_estimate}, 
Theorem~\ref{thm:existence_uniqueness}, 
and
Corollary~\ref{cor:rough_initial} outlined above
(see \eqref{eq:intro.perturb}--\eqref{eq:intro.SEE.initial.regularity})
are of particular importance for establishing regularity properties
for Kolmogorov backward equations associated to parabolic semilinear SEEs
and, thereby, for establishing essentially sharp probabilistically \emph{weak convergence rates} for 
numerical approximations of parabolic semilinear SEEs
(cf., e.g., 
Lemmas~4.4--4.6 in Debussche~\cite{Debussche2011}, 
Lemma~3.3 in Wang \& Gan~\cite{WangGan2013_Weak_convergence}, 
(4.2)--(4.3) in Andersson \& Larsson~\cite{AnderssonLarsson2015},
Propositions~5.1--5.2 and Lemma~5.4 in Br\'{e}hier~\cite{Brehier2014},
Lemma~3.3 in Wang~\cite{Wang2016481},
(79) in Conus et al.~\cite{ConusJentzenKurniawan2014arXiv}, 
Proposition~7.1, Lemma~10.5, and Lemma~10.10 in Kopec~\cite{Kopec2014_PhD_Thesis},
and
(183)--(184) in Jentzen \& Kurniawan~\cite{JentzenKurniawan2015arXiv}).
The analytically weak norm for the initial condition in \eqref{eq:intro.SEE.initial.regularity} 
as well as the singularities in the nonlinear coefficients of the SEE in~\eqref{eq:intro.singular.F} and~\eqref{eq:intro.singular.B} 
above translate in an analytically weak norm for the approximation errors in the probabilistically weak error analysis
which, in turn, results in essentially sharp probabilistically weak convergence rates
(cf., e.g., 
Theorem~2.2 in Debussche~\cite{Debussche2011},
Theorem~2.1 in Wang \& Gan~\cite{WangGan2013_Weak_convergence}, 
Theorem~1.1 in Andersson \& Larsson~\cite{AnderssonLarsson2015},
Theorem~1.1 in Br\'{e}hier~\cite{Brehier2014},
Theorem~5.1 in Br\'{e}hier \& Kopec~\cite{BrehierKopec2016},
Corollary~1 in Wang~\cite{Wang2016481},
Corollary~5.2 in Conus et al.~\cite{ConusJentzenKurniawan2014arXiv}, 
Theorem~6.1 in Kopec~\cite{Kopec2014_PhD_Thesis},
and
Corollary~8.2 in~\cite{JentzenKurniawan2015arXiv}).
The perturbation inequality in Proposition~\ref{prop:perturbation_estimate}
(see \eqref{eq:intro.perturb} above) 
is also useful 
to establish essentially sharp probabilistically \emph{strong convergence rates} for numerical approximations
and perturbations of SEEs
(cf., e.g., 
Proposition~4.1 in Conus et al.~\cite{ConusJentzenKurniawan2014arXiv}
and 
Proposition~4.3 in \cite{JentzenKurniawan2015arXiv}).

\subsection{Notation}
\label{sec:notation}

Throughout this article the following notation is used.
For 
two measurable spaces
$
  ( A, \mathcal{A} )
$
and
$
  ( B, \mathcal{B} )
$
we denote by
$
  \mathcal{M}( \mathcal{A}, \mathcal{B} )
$
the set of all 
$ \mathcal{A} $/$ \mathcal{B} $-measurable
functions.
For a set $ A $ 
we denote by 
$ \mathcal{P}(A) $ the power set of $ A $
and we denote by 
$ \#_A \colon \mathcal{P}(A) \to [0,\infty] $ 
the counting measure on $ A $.
For a Borel measurable set $ A \in \mathcal{B}( \R ) $
we denote by $ \mu_A \colon \mathcal{B}( A ) \to [0,\infty] $
the Lebesgue-Borel measure on $ A $.
For a real number $ T \in (0,\infty) $ 
and a probability space 
$ (\Omega,\mathcal{F},\P) $ 
with a normal filtration 
$ (\mathcal{F}_t)_{ t \in [0,T] } $
(see, e.g., Definition~2.1.11 in~\cite{PrevotRoeckner2007}) 
we call the quadruple 
$
  ( \Omega , \mathcal{F}, \P, ( \mathcal{F}_t )_{ t \in [0,T] } )
$ 
a stochastic basis. 
For a real number $ T \in (0,\infty) $
and a filtered probability space
$
  ( \Omega , \mathcal{F}, \P, ( \mathcal{F}_t )_{ t \in [0,T] } )
$
we denote by
$
  \mathrm{Pred}( ( \mathcal{F}_t )_{ t \in [0,T] } )
$
the sigma-algebra given by 
\begin{equation}
  \mathrm{Pred}( ( \mathcal{F}_t )_{ t \in [0,T] } )
  =
  \sigma_{ [0,T] \times \Omega }\big(
    \big\{
      ( s, t ] \times A
      \colon
      s, t \in [0,T] , s < t,
      A \in \mathcal{F}_s
    \big\}
    \cup
    \big\{
      \{ 0 \} \times A
      \colon
      A \in \mathcal{F}_0
    \big\}
  \big)
\end{equation}
(the predictable sigma-algebra associated
to
$
  ( \mathcal{F}_t )_{ t \in [0,T] }
$).
We denote by 
$ \lceil \cdot \rceil_h \colon \R \to \R $,
$ h \in (0,\infty) $,
the functions which satisfy for all
$ h \in (0,\infty) $, $ t \in \R $
that 
$
  \lceil t \rceil_h =
  \min(
    [ t, \infty )
    \cap
    \{ 0, h , - h , 
$
$
    2 h , - 2 h , \dots \}
  )
$. 
For $ \R $-Banach spaces
$ ( V , \left\| \cdot \right\|_V ) $
and
$ ( W , \left\| \cdot \right\|_W ) $
we denote by
$
  \left| \cdot \right|_{
    \operatorname{Lip}( V, W )
  }
  \colon 
  \mathcal{C}( V, W )
  \to [0,\infty]
$
and 
$
  \left\| \cdot \right\|_{
    \operatorname{Lip}( V, W )
  }
  \colon 
  \mathcal{C}( V, W )
  \to [0,\infty]
$
the functions which satisfy
for all $ f \in \mathcal{C}( V, W ) $
that
\begin{equation}
\label{eq:Lip.def}
\begin{split}
  \left| f \right|_{ 
    \operatorname{Lip}( V, W )
  }
  &=
  \sup\!\left(\left\{
    \frac{
      \left\| f( x ) - f( y ) \right\|_W
    }{
      \left\| x - y \right\|_V
    }
    \colon
    x, y \in V, \, x \neq y
  \right\}
  \cup \{0\}
  \right)
  ,
\\
  \left\| f \right\|_{
    \operatorname{Lip}( V, W )
  }
  &
  =
  \|f(0)\|_W
  +
  \left| f \right|_{ \operatorname{Lip}(V,W) }
\end{split}
\end{equation}
and we denote by 
$ \operatorname{Lip}(V,W) $ 
the set given by 
$
  \operatorname{Lip}(V,W)
  =
  \{ f \in \mathcal{C}(V,W) \colon |f|_{ \operatorname{Lip}(V,W) } < \infty \}
$.
For a separable $ \R $-Hilbert space
$ ( H , \left\| \cdot \right\|_H , \left< \cdot , \cdot \right>_H ) $,
real numbers 
$ T \in (0,\infty) $, 
$ \eta \in \R $, 
$ r \in [0,\infty) $, 
$ s \in [0,1] $, 
and a generator of a strongly continuous analytic semigroup
$
  A \colon D(A)
  \subseteq
  H \rightarrow H
$
with 
$
  \operatorname{spectrum}( A )
  \subseteq
  \{
    z \in \mathbb{C}
    \colon
    \text{Re}( z ) < \eta
  \}
$
we denote by
$
  \chi^{ r, T }_{
    A, \eta
  }, 
  \kappa^{ s, T }_{
    A, \eta
  }  
  \in [0,\infty)
$
the real numbers given by
$
  \chi^{ r, T }_{ A, \eta }
  =
  \sup_{ t \in (0,T] }
    t^r
    \,
    \|
      ( \eta - A )^r
      e^{ t A }
    \|_{ L( H ) }
$
(cf., e.g., \eqref{eq:BigTheta.intro} in Section~\ref{sec:intro} above)
and 
$
  \kappa^{ s, T }_{ A, \eta }
  =
  \sup_{ t \in (0,T] }
      t^{-s}
      \,
      \|
        ( \eta - A )^{ - s }
        ( e^{ t A } - \operatorname{Id}_H )
      \|_{ L( H ) }
$
(cf., e.g., \cite[Lemma~11.36]{rr93})
.
We denote by 
$
  \mathbbm{B} \colon (0,\infty)^2 \to (0,\infty)
$
the function with the property
that for all $ x, y \in (0,\infty) $ it holds that
$
  \mathbbm{B}( x, y ) = \int_0^1 t^{ (x - 1) } \left( 1 - t \right)^{ (y - 1) }
  \diffns t
$ 
(Beta function).
We denote by 
$
  \mathrm{E}_{ \alpha, \beta } \colon [0,\infty) \to [0,\infty)
$,
$ \alpha, \beta \in (-\infty,1) $,
the functions which satisfy for all
$
  \alpha, \beta \in (-\infty,1)
$,
$
  x \in [0,\infty)
$
that
\begin{equation}
  \mathrm{E}_{ \alpha, \beta }[ x ]
  =
  1
  +
  \smallsum_{ n = 1 }^{ \infty }
  x^n
  \prod_{ k = 0 }^{ n - 1 }
  \mathbb{B}\big(
    1-\beta
    ,
    k(1-\beta) + 1-\alpha
  \big)
\end{equation}
(generalized exponential function; cf.\ Lemma~7.1.1 in Chapter~7 in Henry~\cite{h81}, 
(1.0.3) in Chapter~1 in Gorenflo et al.~\cite{Gorenfloetal2014}, 
and Lemma~\ref{lem:gronwall}
below). 
For real numbers
$ T \in (0,\infty) $, 
$ \eta \in \R $, 
$ p \in [1,\infty) $, 
$ a, \lambda \in (-\infty,1) $,
$ b \in (-\infty, \frac{ 1 }{ 2 } ) $, 
a separable $ \R $-Hilbert space
$ 
  ( H , \left\| \cdot \right\|_H , \left< \cdot , \cdot \right>_H ) 
$, 
and a generator
$ A \colon D(A) \subseteq H \to H $
of a strongly continuous analytic semigroup with
$
  \operatorname{spectrum}( A )
  \subseteq
  \{
    z \in \mathbb{C}
    \colon
    \text{Re}( z ) < \eta
  \}
$
we denote by 
$
  \Theta_{ A, \eta, p, T }^{ a, b, \lambda }
  \colon
  [0,\infty)^2
  \to
  [0,\infty]
$
the function which satisfies
for all
$ L, \hat{L} \in [0,\infty) $
that
{\small
\begin{equation}
\label{eq:BigTheta}
\begin{split}
&
  \Theta_{ A , \eta, p, T }^{ a, b, \lambda }( L, \hat{L} )
  =
\\ &
\begin{cases}
  \sqrt{2}
  \,
  \bigg|
    E_{
      2 \lambda, \max\{ a, 2 b \}
    }\bigg[
      \Big|
        \tfrac{
        \chi_{ A , \eta }^{ a , T }\,
        L\,
          \sqrt{2}\,
          T^{ (1 - a) }
        }
        {
          \sqrt{1-a}
        }
        +
        \chi_{A,\eta}^{b,T}\,
        \hat{L}\,
        \sqrt{ p \, ( p - 1 ) \, T^{ (1 - 2b) } }
      \Big|^2
    \bigg]
  \bigg|^{1/2}
&
  \colon
  ( \lambda, \hat{L} )
  \in ( -\infty, \frac{ 1 }{ 2 } ) \times (0,\infty)
\\[1ex]
  E_{\lambda,a}\!\left[
    \chi_{ A , \eta }^{ a , T }\,
    L\,
    T^{ (1 - a) }
  \right]
&
  \colon
  \hat{L} = 0
\\[1ex]
  \infty
&
  \colon
  \text{otherwise}
\end{cases}
  .
\end{split}
\end{equation}
}For a measure space $ ( \Omega , \mathcal{F}, \mu ) $,
a measurable space $ ( S , \mathcal{S} ) $,
and an $ \mathcal{F} $/$ \mathcal{S} $-measurable function
$ f \colon \Omega \to S $
we denote by
$
   \left[ f \right]_{
     \mu, \mathcal{S}
   }
$
the set given by
\begin{equation}
   \left[ f \right]_{
     \mu, \mathcal{S}
   }
   =
   \left\{
     g \in \mathcal{M}( \mathcal{F}, \mathcal{S} )
     \colon
     (
     \exists \, A \in \mathcal{F} \colon
     \mu(A) = 0 
     \text{ and }
     \{ \omega \in \Omega \colon f(\omega) \neq g(\omega) \}
     \subseteq A
     )
   \right\}
   .
\end{equation}
For a measure space $ ( \Omega , \mathcal{F}, \mu ) $ 
and a measurable space $ ( S , \mathcal{S} ) $
we do as usual often not distinguish between
an $ \mathcal{F} $/$ \mathcal{S} $-measurable function 
$ f \colon \Omega \to S $ 
and its equivalence class $ \left[ f \right]_{ \mu, \mathcal{S} } $.

\section{Stochastic evolution equations (SEEs) with singularities at the initial time}
\label{sec:SEE_rough_initial_values}

\subsection{Setting}
\label{sec:SPDE_setting}

Throughout this section the following setting is frequently used.
Let
$
  \left(
    H,
    \left\| \cdot \right\|_H,
    \left< \cdot, \cdot \right>_H
  \right)
$ 
and 
$
  \left(
    U,
    \left\| \cdot \right\|_U ,
    \left< \cdot, \cdot \right>_U
  \right)
$
be separable $ \R $-Hilbert spaces
with $ \#_H(H) > 1 $,
let
$ T \in (0,\infty) $,
$ \eta \in \R $,
$
  p \in [ 2 , \infty )
$,
$ \alpha \in [ 0, 1 ) $,
$ \hat{ \alpha } \in ( -\infty, 1 ) $, 
$
  \beta \in [ 0, \nicefrac{ 1 }{ 2 } )
$,
$
  \hat{ \beta } \in ( -\infty, \nicefrac{ 1 }{ 2 } )
$,
$ L_0, \hat{L}_0, L_1, \hat{L}_1 \in [0,\infty) $
satisfy
$
  \mathbbm{1}_{ (0,\infty) }( L_1 )
  \cdot
  \left[ \alpha + \hat{ \alpha }
  \right]
  < \nicefrac{ 3 }{ 2 }
$, 
let
$
  ( \Omega , \mathcal{F}, \P, ( \mathcal{F}_t )_{ t \in [0,T] } )
$
be a stochastic basis,
let
$
  ( W_t )_{ t \in [0,T] }
$
be an $\operatorname{Id}_U$-cylindrical $ ( \mathcal{F}_t )_{ t \in [0,T] } $-Wiener process,
let
$
  A \colon D(A)
  \subseteq
  H \rightarrow H
$
be a generator of a strongly continuous analytic semigroup
with 
$
  \operatorname{spectrum}( A )
  \subseteq
  \{
    z \in \mathbb{C}
    \colon
    \operatorname{Re}( z ) < \eta
  \}
$,
let
$
  (
    H_r
    ,
    \left\| \cdot \right\|_{ H_r }
    ,
    \left< \cdot , \cdot \right>_{ H_r }
  )
$,
$ r \in \R $,
be a family of interpolation spaces associated to
$
  \eta - A
$,
and let
$
  \mathbf{F} \in
  \mathcal{M}\big(
    \mathrm{Pred}( ( \mathcal{F}_t )_{ t \in [0,T] } ) \otimes \mathcal{B}( H )
    ,
    \mathcal{B}( H_{ - \alpha } )
  \big)
$ 
and 
$
  \mathbf{B} \in
  \mathcal{M}\big(
    \mathrm{Pred}(
      ( \mathcal{F}_t )_{ t \in [0,T] }
    )
    \otimes
    \mathcal{B}( H )
    ,
    \mathcal{B}(
      HS( U, H_{ - \beta } )
    )
  \big)
$
satisfy 
for all $ t \in (0,T] $, $ X, Y \in \mathcal{L}^p( \P; H ) $
that
\begin{align}
\label{eq:F_Lipschitz}
  \|
    \mathbf{F}(t,X) - \mathbf{F}(t,Y)
  \|_{
    L^p( \P; H_{ - \alpha } )
  }
  \leq
  L_0 \,
  \| X - Y \|_{
    L^p( \P; H )
  }
  ,\,\,
&
  \|
    \mathbf{F}( t , 0 )
  \|_{
    L^p(\P;H_{ - \alpha } )
  }
  \leq
    \hat{L}_0
    \,
    t^{ - \hat{\alpha} }
  ,
\\
\label{eq:G_Lipschitz}
  \|
    \mathbf{B}(t,X) - \mathbf{B}(t,Y)
  \|_{
    L^p( \P; HS( U, H_{ - \beta } ) )
  }
  \leq
  L_1
  \| X - Y \|_{
    L^p( \P; H )
  },\,\,
&
  \| \mathbf{B}(t,0) \|_{
    L^p( \P; HS( U, H_{ - \beta } ) )
  }
\leq
    \hat{L}_1
    \,
    t^{ - \hat{ \beta } }
  .
\end{align}

\subsection{Predictable stochastic processes with singularities at the initial time}
\label{sec:initial_singularity}
The next result, Lemma~\ref{lem:predictability}, is an elementary lemma that slightly
generalizes Proposition~3.6 (ii) in Da Prato \& Zabczyk~\cite{dz92}.
\begin{lemma}[Existence of predictable modifications]
\label{lem:predictability}
Let
$ T \in [0,\infty) $,
let
$ ( \Omega, \mathcal{F}, \P, ( \mathcal{F}_t )_{ t \in [0,T] } ) $
be a stochastic basis,
let $ ( E, d_E ) $ be a complete and separable metric space,
and let
$ Y \colon [0,T] \times \Omega \to E $
be an $ ( \mathcal{F}_t )_{ t \in [0,T] } $-adapted stochastic process
which satisfies for all 
$
  t \in (0,\infty) \cap (-\infty,T]
$
that 
$
  \limsup_{ [0,T] \ni s \to t } \E\big[ \min\{ 1, d_E( Y_s , Y_t ) \} \big] = 0
$.
Then there exists an $ ( \mathcal{F}_t )_{ t \in [0,T] } $-predictable 
stochastic process
$ X \colon [0,T] \times \Omega \to E $
which satisfies for all 
$ t \in [0,T] $ 
that 
$ \P( X_t = Y_t ) = 1  $.
\end{lemma}

\begin{proof}
First, we observe that the assumption that
$ ( \Omega, \mathcal{F}, \P ) $
is a probability space ensures that
$ \Omega \neq \emptyset $
and this implies that
$
  [0,T] \times \Omega \neq \emptyset
$.
The assumption that $ Y \colon [0,T] \times \Omega \to E $
is a mapping from $ [0,T] \times \Omega $ to $ E $
therefore ensures that $ E \neq \emptyset $.
Hence, there exists an element $ e_0 \in E $.
In the next step 
assume without loss of generality that $ T > 0 $, 
let
$ Z^N \colon [0,T] \times \Omega \to E $,
$ N \in \N $,
be the 
functions with the property that for all
$ N \in \N $, $ t \in [0,T] $
it holds that
$
  Z^N_t
  =
  Y_{
    \max\{
      \lceil t \rceil_{ T / N } - T/N
      ,
      0
    \}
  }
$, 
and let
$ w \colon (0,T] \times \N \to [0,\infty) $
be the function with the property that
for all $ \varepsilon \in (0,T] $, $ N \in \N $
it holds that
\begin{equation}
  w( \varepsilon, N )
  =
  \sup_{
    \substack{
      t_1, t_2 \in [ \varepsilon , T ] ,
    \\
      | t_1 - t_2 | \leq \nicefrac{ T }{ N }
    }
  }
  \E\big[
    \min\!\big\{
      1 ,
      d_E( Y_{ t_1 }, Y_{ t_2 } )
    \big\}
  \big]
  .
\end{equation}
The assumption that
$
  \forall \, t \in (0,T] \colon
  \lim_{ s \to t }
  \E\!\left[
    \min\!\left\{ 1,
      d_E( Y_s, Y_t )
    \right\}
  \right]
  = 0
$
ensures that 
for all $ \varepsilon \in (0,T] $
it holds that
$
  \lim_{ N \to \infty }
  w( \varepsilon, N ) = 0
$.
This implies that there exists a strictly
increasing sequence 
$ N_k \in \N $, $ k \in \N $, with the property that
for all $ k \in \N $ 
it holds that
\begin{equation}
\label{eq:sequence_Nk}
  w( \tfrac{1}{k} , N_k ) < \tfrac{ 1 }{ 2^k }
  .
\end{equation}
Next let $ X \colon [0,T] \times \Omega \to E $
be the mapping with the property that for all
$ ( t, \omega ) \in [0,T] \times \Omega $
it holds that
\begin{equation}
\label{eq:definition_of_X}
  X_t( \omega )
  =
  \begin{cases}
    \lim_{ k \to \infty }
    Z^{ N_k }_t( \omega )
  &
    \colon
    ( Z^{ N_k }_t( \omega ) )_{ k \in \N }
    \text{ is convergent}
  \\
    e_0
  &
    \colon
    \text{else}
  \end{cases}
  .
\end{equation}
The fact that for all $ N \in \N $ it holds that
$ Z^N $ is $ \mathrm{Pred}( ( \mathcal{F}_t )_{ t \in [0,T] } ) $/$ \mathcal{B}( E ) $-measurable,
the assumption that $ ( E, d_E ) $ is complete and separable,
and, e.g., Exercise~1.74 in Chapter~1 in Hoffmann-J\o{}rgensen~\cite{Hoffmann1994}
imply that
\begin{equation}
  \big\{
    ( t, \omega ) \in [0,T] \times \Omega
    \colon
    ( Z^{ N_k }_t( \omega ) )_{ k \in \N }
    \text{ is convergent}
  \big\}
  \in
  \mathrm{Pred}( ( \mathcal{F}_t )_{ t \in [0,T] } )
  .
\end{equation}
This together with 
the fact that for all $ N \in \N $ it holds that
$ Z^N $ is
$ \mathrm{Pred}( ( \mathcal{F}_t )_{ t \in [0,T] } ) $/$ \mathcal{B}( E ) $-measurable,
and, e.g., Exercise~1.74 in Chapter~1 in Hoffmann-J\o{}rgensen~\cite{Hoffmann1994}
ensure that
$ X $ is $ \mathrm{Pred}( ( \mathcal{F}_t )_{ t \in [0,T] } ) $/$ \mathcal{B}( E ) $-measurable.
It thus remains to prove that $ X $ is a modification of $ Y $.
For this we note that for all
$ N \in \N $,
$ t \in ( \frac{ T }{ N } , T ] $
it holds that
\begin{equation}
  \E\!\left[
    \min\!\left\{
      1 ,
      d_E( Y_t, Z^N_t )
    \right\}
  \right]
=
  \E\big[
    \min\!\big\{
      1 ,
      d_E( Y_t, Y_{ \lceil t \rceil_{ T / N } - T/N } )
    \big\}
  \big]
\leq
  w( t - \tfrac{ T }{ N }, N )
  .
\end{equation}
This together with \eqref{eq:sequence_Nk}, 
the fact that 
$ 
  \forall \, 
  \varepsilon_1, \varepsilon_2 \in (0,T],\, 
  N \in \N 
  \text{ with }
  \varepsilon_1 \leq \varepsilon_2 
  \colon
  w( \varepsilon_1, N ) \geq w( \varepsilon_2, N )
$, 
and the fact that 
$ 
  \forall \,
  t \in (0,T],\,
  k \in \N \cap ( \frac{ T + 1 }{ t } , \infty ) 
  \colon
  \frac{ 1 }{ k } < t - \frac{ T }{ N_k }
$
ensure that for all $ t \in (0,T] $
it holds that
\begin{equation}
\begin{split}
&
  \sum_{ k = 1 }^{ \infty }
  \E\!\left[
    \min\!\left\{
      1 ,
      d_E( Y_t, Z^{ N_k }_t )
    \right\}
  \right]
  =
  \sum_{
    k \in \N
  }
  \E\!\left[
    \min\!\left\{
      1 ,
      d_E(
        Y_t,
        Y_{
          \lceil t \rceil_{ T / N_k } - T/N_k
        }
      )
    \right\}
  \right]
\\ & =
  \sum_{
    k \in \N
    \cap ( 0, \nicefrac{ ( T + 1 ) }{ t } ]
  }
  \E\!\left[
    \min\!\left\{
      1 ,
      d_E(
        Y_t,
        Y_{
          \lceil t \rceil_{ T / N_k } - T/N_k
        }
      )
    \right\}
  \right]
\\&\quad+
  \sum_{
    k \in \N
    \cap ( \nicefrac{ ( T + 1 ) }{ t } , \infty )
  }
  \E\!\left[
    \min\!\left\{
      1 ,
      d_E(
        Y_t,
        Y_{
          \lceil t \rceil_{ T / N_k } - T/N_k
        }
      )
    \right\}
  \right]
\\ & \leq
  \frac{T+1}{t}
  +
  \sum_{
    k \in \N
    \cap ( \nicefrac{ ( T + 1 ) }{ t } , \infty )
  }
  w( t - \tfrac{ T }{ N_k } , N_k )
\leq
  \frac{T+1}{t}
  +
  \sum_{
    k \in \N
    \cap ( \nicefrac{ ( T + 1 ) }{ t } , \infty )
  }
  w( \tfrac{ 1 }{ k } , N_k )
\\ & \leq
  \frac{T+1}{t}
  +
  \sum_{
    k \in \N
    \cap ( \nicefrac{ ( T + 1 ) }{ t } , \infty )
  }
  \frac{ 1 }{ 2^k }
  < \infty
  .
\end{split}
\end{equation}
This implies that 
for all $ t \in (0,T] $
it holds $ \P $-a.s.\ that
$
  \limsup_{ k \to \infty }
  d_E( Z^{ N_k }_t, Y_t )
  = 0
$
(see, e.g., item~(ii) of Theorem~6.12 in Klenke~\cite{Klenke2008}).
This and \eqref{eq:definition_of_X}
ensure for all $ t \in (0,T] $
that
$ \P( X_t = Y_t ) = 1  $.
This and the fact that 
$ 
  \forall\, N \in \N 
  \colon
  X_0 = Z^N_0 = Y_0 
$
complete the proof
of Lemma~\ref{lem:predictability}.
\end{proof}

\begin{lemma}
\label{lem:Kuratowski}
Let 
$ (V_k, \left\|\cdot\right\|_{V_k}) $, $ k \in \{0,1\} $, 
be separable $\R$-Banach spaces with 
$ V_1 \subseteq V_0 $ continuously. 
Then 
\begin{equation}
\label{eq:Kuratowski}
  \mathcal{B}(V_1) =
  \{
    B \in \mathcal{P}(V_1) \colon
    (
      \exists \, A \in \mathcal{B}(V_0) 
      \colon
      B = A \cap V_1
    )
  \} 
  \subseteq \mathcal{B}(V_0)
  .
\end{equation}
\end{lemma}
\begin{proof}
Throughout this proof let 
$ \varphi \colon V_1 \to V_0 $ 
and 
$ \phi \colon V_1 \to V_1 $ 
be the mappings with the property that for all 
$ v \in V_1 $ it holds that 
$ \varphi(v) = \phi(v) = v $. 
Next observe that 
$ \varphi \in \mathcal{C}(V_1,V_0) $. 
This implies that 
$ \varphi \in \mathcal{M}( \mathcal{B}(V_1), \mathcal{B}(V_0) ) $. 
Hence, we obtain that 
\begin{equation}
\label{eq:Kuratowski.subset.I}
  \{
  B \in \mathcal{P}(V_1) \colon
  (
  \exists \, A \in \mathcal{B}(V_0) 
  \colon
  B = A \cap V_1
  )
  \} 
  \subseteq
  \mathcal{B}(V_1)
.
\end{equation}
Moreover, note that the fact that 
$ \varphi \in \mathcal{M}( \mathcal{B}(V_1), \mathcal{B}(V_0) ) $ 
allows us to apply, e.g., Theorem~2.4 in Chapter~V in Parthasarathy~\cite{Parthasarathy67} 
(with 
$  
(X,\mathscr{B}) = (V_1,\mathcal{B}(V_1))
$, 
$
(Y,\mathscr{C}) = (V_0,\mathcal{B}(V_0))
$, 
and 
$ \varphi = \varphi $ 
in the notation of Theorem~2.4 in Chapter~V in Parthasarathy~\cite{Parthasarathy67}) 
to obtain that for all 
$ C \in \mathcal{B}(V_1) $ 
it holds that  
$ V_1 = \varphi(V_1) \in \mathcal{B}(V_0) $ 
and 
$
  C
  =
  \phi(C)
  =
  (\phi^{-1})^{-1}(C)  
  \in
  \{
  B \in \mathcal{P}(V_1) \colon
  (
  \exists \, A \in \mathcal{B}(V_0) 
  \colon
  B = A \cap V_1
  )
  \}   
$. 
This implies that 
\begin{equation}
\label{eq:Kuratowski.subset.II}
  \mathcal{B}(V_1)
  \subseteq
  \{
  B \in \mathcal{P}(V_1) \colon
  (
  \exists \, A \in \mathcal{B}(V_0) 
  \colon
  B = A \cap V_1
  )
  \} 
  .
\end{equation}
Combining~\eqref{eq:Kuratowski.subset.I}, \eqref{eq:Kuratowski.subset.II}, 
and the fact that $ V_1 \in \mathcal{B}(V_0) $ 
completes the proof of Lemma~\ref{lem:Kuratowski}.
\end{proof}

\begin{lemma}[Non-stochastic integral]
\label{lem:integrals}
Assume the setting in Section~\ref{sec:SPDE_setting}, 
let 
$ \delta \in \R $, 
$
  \lambda \in ( -\infty, 1 )
$, 
and let 
$ Y \colon [0,T] \times \Omega \to H_{\delta} $ 
be an $ (\mathcal{F}_t)_{ t \in [0,T] } $-predictable stochastic process which satisfies 
$ Y( (0,T] \times \Omega ) \subseteq H $ 
and 
$
  \sup_{ t \in (0,T] }
  t^\lambda \, 
  \| Y_t \|_{ \lpn{p}{\P}{H} }
  < \infty
$.
Then
\begin{enumerate}[(i)]
\item 
\label{item:integral.defined}
for all $ t \in [0,T] $ it holds $ \P $-a.s.\ that
$
  \int_0^t \| e^{ ( t - s ) A } \mathbf{F}( s, Y_s ) \|_H \, ds < \infty
$,
\item 
there exists an up-to-modifications unique
$ (\mathcal{F}_t)_{ t \in [0,T] } $-predictable stochastic process 
$
  \bar{Y} \colon [0,T] \times \Omega \to H
$ 
such that for all $ t \in [0,T] $ it holds $ \P $-a.s.\ that
$
  \bar{Y}_t = \int_0^t e^{ ( t - s ) A } \mathbf{F}( s, Y_s ) \, ds
$, 
\item
it holds that 
\begin{equation}
\label{eq:bound_Lebesgue1}
\begin{split}
&
  \sup_{ t \in (0,T] }
  t^{ (\max\{ \lambda, \hat{\alpha} \} + \alpha - 1) }\,
  \| \bar{Y}_t \|_{ \lpn{p}{\P}{H} }
\leq
  \left(
  \hat{L}_0
  +
  L_0
  \sup\nolimits_{ t \in (0,T] }
    t^\lambda\,
    \|
      Y_t
    \|_{
      \lpn{p}{\P}{H}
    }
  \right)
\\&\cdot
  | T \vee 1 |^{ | \lambda - \hat{\alpha} | }
  \,
  \mathbbm{B}\big(
    1 - \alpha,
    1 - \max\{ \lambda, \hat{ \alpha } \}
  \big)
  \,
  \chi_{ A , \eta }^{ \alpha , T }
  < \infty
  ,
\end{split}
\end{equation}
\item
\label{item:drift.temporal}
and for all 
$\varrho\in[0,1-\alpha)$, 
$s,t\in(0,T]$ 
with $s<t$ 
it holds that 
\begin{equation}
\begin{split}
&
  \|\bar{Y}_t-\bar{Y}_s\|_{\lpn{p}{\P}{H}}
\leq
  | T \vee 1 |^{ |\lambda-\hat{\alpha}| }
  \left(
  \hat{L}_0
  +
  L_0
  \sup\nolimits_{ u \in (0,T] }
    u^\lambda\,
    \|
      Y_u
    \|_{
      \lpn{p}{\P}{H}
    }
  \right)
  \left|
    t - s
  \right|^{
    \varrho
  }
\\&\cdot
  \bigg[
  \tfrac{
  \chi^{ \alpha, T }_{A,\eta} \,
    \left| t - s \right|^{ (1 - \alpha-\varrho) }
  }{
    \left(
      1 - \alpha
    \right)
    \min\{
    s^{
      \max\{
        \lambda,
        \hat{ \alpha }
      \}
    }
    ,
   t^{
      \max\{
        \lambda,
        \hat{ \alpha }
      \}
    }
    \}
  }
+
  \tfrac{
  \kappa^{\varrho,T}_{A,\eta}
  \,
  \chi^{ \varrho + \alpha, T }_{A,\eta} \,
    \mathbbm{B}(
      1 -
      \alpha
      - \varrho
      ,
      1 -
      \max\{ \lambda, \hat{ \alpha } \}
    )
  }{
    s^{(
      \varrho + \alpha
      +
      \max\{ \lambda , \hat{ \alpha } \}
      - 1
    )}
  }
  \bigg]
  .
\end{split}
\end{equation}
\end{enumerate}
\end{lemma}

\begin{proof}
Throughout this proof let 
$ K \in [0,\infty) $ 
be the real number given by 
$
  K
  =
  \sup_{ t \in (0,T] }
  t^\lambda \,
  \| Y_t \|_{ \lpn{p}{\P}{H} }
$.
We observe that~\eqref{eq:F_Lipschitz} implies that for all $ t \in (0,T] $
it holds that
\begin{equation}\label{eq:F_finite}
\begin{split}
& \int_0^t
    \big\|
      e^{(t-s)A}
      \mathbf{F}(s,Y_s)
    \big\|_{
      L^p( \P ; H )
    }
    \,
    \diffns s
\\ &
  \leq
  \chi_{ A , \eta }^{ \alpha , T }
  \int_0^t
    \left( t - s \right)^{ - \alpha }
    \left(
      \left\|
        \mathbf{F}( s, Y_s ) - \mathbf{F}( s, 0 )
      \right\|_{ L^p( \P; H_{ - \alpha } ) }
      +
      \left\|
        \mathbf{F}( s, 0 )
      \right\|_{ L^p( \P; H_{ - \alpha } ) }
    \right)
  \diffns s
\\ &
  \leq
  \chi_{ A , \eta }^{ \alpha , T }
  \int_0^t
    \left( t - s \right)^{ - \alpha }
    \left(
      L_0
      \left\|
        Y_s
      \right\|_{
        L^p( \P; H )
      }
      +
      \hat{L}_0
      \,
      s^{ - \hat{\alpha} }
    \right)
  \diffns s
\\ &
  \leq
  (
    K L_0
    +
    \hat{L}_0
  )
  \,
  \chi_{ A , \eta }^{ \alpha , T }
  \int_0^t
    \left( t - s \right)^{ - \alpha }
    \max\!\big\{
      s^{ - \lambda } ,
      s^{ - \hat{ \alpha } }
    \big\}
  \,
  \diffns s
\\ &
  \leq
  (
    K L_0
    +
    \hat{L}_0
  )
  \,
  \chi_{ A , \eta }^{ \alpha , T }\,
  | T \vee 1 |^{ | \lambda - \hat{\alpha} | }
  \int_0^t
    \left( t - s \right)^{ - \alpha }
    s^{
      -\max\{ \lambda, \hat{ \alpha } \}
    }
  \,
  \diffns s
\\ & \leq
  (
    K L_0
    +
    \hat{L}_0
  )
  \,
  \chi_{ A , \eta }^{ \alpha , T }\,
  | T \vee 1 |^{ | \lambda - \hat{\alpha} | }
  \,
  \mathbbm{B}\big(
    1 - \alpha,
    1 - \max\{ \lambda, \hat{ \alpha } \}
  \big)
  \,
  t^{
    (1 - \alpha - \max\{ \lambda, \hat{ \alpha } \})
  }
  .
\end{split}
\end{equation}
In particular, this ensures that
for all $ t \in [0,T] $ it holds $ \P $-a.s.\ that
$
  \int_0^t \| e^{ ( t - s ) A } \mathbf{F}( s, Y_s ) \|_H \, ds < \infty
$.
Moreover, we note that for all
$
  \varrho \in [ 0, 1 - \alpha )
$,	
$ t_1, t_2 \in (0,T] $ with $ t_1 < t_2 $
it holds that
\begin{equation}
\begin{split}
&
  \left\|
    \int_0^{t_2}
        e^{(t_2-s)A}
        \mathbf{F}(s, Y_s)
        \,
    \diffns s
    -
    \int_0^{t_1}
        e^{(t_1-s)A}
        \mathbf{F}(s, Y_s)
        \,
    \diffns s
  \right\|_{
    L^p(\P; H)
  }
\\& \leq
  \int_0^{t_1}
    \left\|
      \left(
        e^{ ( t_2 - s ) A } - e^{ ( t_1 - s ) A }
      \right)
      \mathbf{F}( s, Y_s )
    \right\|_{
      L^p(\P;H)
    }
  \diffns s
    +
    \int_{t_1}^{t_2}
    \left\|
      e^{ (t_2-s) A }
      \mathbf{F}( s, Y_s )
    \right\|_{
      L^p( \P ; H )
    }
  \diffns s
\\ & \leq
    \left\|
      \left(
        \operatorname{Id}_H - e^{ ( t_2 - t_1 ) A }
      \right)
    \right\|_{
        L(H_{\varrho},H)
    }
    \int_0^{t_1}
      \left\|
        e^{ ( t_1 - s ) A }
      \right\|_{
        L( H_{ - \alpha }, H_{ \varrho } )
      }
      \|
        \mathbf{F}( s, Y_s )
      \|_{
        L^p( \P ; H_{ - \alpha } )
      }
      \,
    \diffns s
\\& \quad
    +
    \int_{ t_1 }^{ t_2 }
      \left\|
        e^{ ( t_2 - s ) A }
      \right\|_{
        L( H_{ - \alpha }, H )
      }
      \|
        \mathbf{F}( s, Y_s )
      \|_{
        L^p( \P ; H_{ - \alpha } )
      }
      \,
    \diffns s
  .
\end{split}
\end{equation}
Assumption~\eqref{eq:F_Lipschitz} hence implies that for all
$
  \varrho \in [ 0, 1 - \alpha )
$,	
$ t_1, t_2 \in (0,T] $ with $ t_1 < t_2 $
it holds that
\begin{equation}
\label{eq:Holder_F}
\begin{split}
&
  \left\|
    \int_0^{t_2}
        e^{(t_2-s)A}
        \mathbf{F}(s, Y_s)
        \,
    \diffns s
    -
    \int_0^{t_1}
        e^{(t_1-s)A}
        \mathbf{F}(s, Y_s)
        \,
    \diffns s
  \right\|_{
    L^p(\P; H)
  }
\\ & \leq
  \kappa^{\varrho,T}_{A,\eta}
  \,
  \chi^{ \varrho + \alpha, T }_{A,\eta}
  \left| t_2 - t_1 \right|^{ \varrho }
  \int_0^{ t_1 }
    \left(
      t_1 - s
    \right)^{
      - (\alpha + \varrho)
    }
    \left(
      L_0
      \left\| Y_s \right\|_{
        \lpn{p}{\P}{H}
      }
      +
      \hat{L}_0
      \,
      s^{ - \hat{\alpha} }
    \right)
  \diffns s
\\& \quad
  +
  \chi^{ \alpha, T }_{A,\eta}
  \int_{ t_1 }^{ t_2 }
    \left(
      t_2 - s
    \right)^{
      - \alpha
    }
    \left(
      L_0
      \left\| Y_s \right\|_{
        \lpn{p}{\P}{H}
      }
      +
      \hat{L}_0
      \,
      s^{ - \hat{ \alpha } }
    \right)
  \diffns s
\\ & \leq
  (
    K L_0
    +
    \hat{L}_0
  )
  \,
  \kappa^{\varrho,T}_{A,\eta}
  \,
  \chi^{ \varrho + \alpha, T }_{A,\eta}
  \left| t_2 - t_1 \right|^{ \varrho }
  \int_0^{ t_1 }
    \left(
      t_1 - s
    \right)^{
      - (\alpha + \varrho)
    }
    \max\!\big\{
      s^{ - \lambda }
      ,
      s^{ - \hat{ \alpha } }
    \big\}
    \,
  \diffns s
\\ & \quad
  +
  (
    K L_0
    +
    \hat{L}_0
  )
  \,
  \chi^{ \alpha, T }_{A,\eta}
  \int_{ t_1 }^{ t_2 }
    \left(
      t_2 - s
    \right)^{
      - \alpha
    }
    \max\!\big\{
      s^{ - \lambda }
      ,
      s^{ - \hat{ \alpha } }
    \big\}
    \,
  \diffns s
\\ & \leq
  (
    K L_0
    +
    \hat{L}_0
  ) \,
  | T \vee 1 |^{ |\lambda-\hat{\alpha}| }
  \bigg[
  \chi^{ \alpha, T }_{A,\eta}
  \int_{ t_1 }^{ t_2 }
    \left(
      t_2 - s
    \right)^{
      - \alpha
    }
    s^{
      -
      \max\{
        \lambda,
        \hat{ \alpha }
      \}
    }
    \,
  \diffns s
\\&\quad+
  \kappa^{\varrho,T}_{A,\eta}
  \,
  \chi^{ \varrho + \alpha, T }_{A,\eta}
  \left| t_2 - t_1 \right|^{ \varrho }
  \int_0^{ t_1 }
    \left(
      t_1 - s
    \right)^{
      - (\alpha + \varrho)
      }
    s^{
      -\max\{
        \lambda,
        \hat{ \alpha }
      \}
    }
    \,
  \diffns s
  \bigg]
\\ & \leq
  \bigg[
  \tfrac{
  \chi^{ \alpha, T }_{A,\eta} \,
    \left| t_2 - t_1 \right|^{ (1 - \alpha) }
  }{
    \left(
      1 - \alpha
    \right)
    \min\{
    |
      t_1
    |^{
      \max\{
        \lambda,
        \hat{ \alpha }
      \}
    }
    ,
    |
      t_2
    |^{
      \max\{
        \lambda,
        \hat{ \alpha }
      \}
    }
    \}
  }
+
  \tfrac{
  \kappa^{\varrho,T}_{A,\eta}
  \,
  \chi^{ \varrho + \alpha, T }_{A,\eta} \,
  \left| t_2 - t_1 \right|^{ \varrho } \,
    \mathbbm{B}(
      1 -
      \alpha
      - \varrho
      ,
      1 -
      \max\{ \lambda, \hat{ \alpha } \}
    )
  }{
    \left| t_1 \right|^{(
      \varrho + \alpha
      +
      \max\{ \lambda , \hat{ \alpha } \}
      - 1
    )}
  }
  \bigg]
\\&\quad\cdot
  (
    K L_0
    +
    \hat{L}_0
  ) \,
  | T \vee 1 |^{ |\lambda-\hat{\alpha}| }
  .
\end{split}
\end{equation}
Combining
\eqref{eq:F_finite},
\eqref{eq:Holder_F},
and
Lemma~\ref{lem:predictability}
completes the proof
of Lemma~\ref{lem:integrals}.
\end{proof}

\begin{lemma}[Stochastic integral]
\label{lem:integrals2}
Assume the setting in Section \ref{sec:SPDE_setting}, 
let 
$ \delta, \lambda \in \R $, 
$
  \rho
  =
  \max\{
    \lambda + ( \hat{\beta} - \lambda ) 
    \1_{\{0\}}(L_1)
    ,
    \hat{\beta}
  \}
$ 
satisfy 
$
  \lambda \, \mathbbm{1}_{ (0,\infty) }(L_1)
  < \nicefrac{1}{2}
$, 
and let 
$ Y \colon [0,T] \times \Omega \to H_{\delta} $ 
be an $ (\mathcal{F}_t)_{ t \in [0,T] } $-predictable stochastic process which satisfies 
$ Y( (0,T] \times \Omega ) \subseteq H $ 
and 
$
  \sup_{ t \in (0,T] }
  t^\lambda \, 
  \| Y_t \|_{ \lpn{p}{\P}{H} }
  < \infty
$.
Then 
\begin{enumerate}[(i)]
\item 
\label{item:stoch.integral.defined}
for all $ t \in [0,T] $ it holds $ \P $-a.s.\ that
$
  \int_0^t \| e^{ (t - s) A } \mathbf{B}( s, Y_s ) \|^2_{ HS( U, H ) } \, ds < \infty
$, 
\item 
there exists an up-to-modifications unique
$ (\mathcal{F}_t)_{ t \in [0,T] } $-predictable stochastic process 
$
  \bar{Y} \colon [0,T] \times \Omega \to H
$ 
such that for all $ t \in [0,T] $
it holds $ \P $-a.s.\ that
$
  \bar{Y}_t =
  \int_0^t
    e^{ ( t - s ) A }
    \mathbf{B}( s , Y_s )
  \, \diffns W_s
$, 
\item
it holds that 
\begin{equation}\label{eq:bound_Lebesgue1_B}
\begin{split}
&
    \sup_{ t \in (0,T] }
    t^{ ( \rho + \beta - 1/2 ) }\,
    \|
      \bar{Y}_t
    \|_{\lpn{p}{\P}{H}}
\leq
  \sqrt{
    \tfrac{p \, (p-1)}{2}\,
    \mathbb{B}
    \big(
      1-2\beta,1-2\rho
    \big)
  }
\\&\cdot
  | T \vee 1 |^{ | \lambda - \hat{\beta} | \mathbbm{1}_{(0,\infty)}(L_1) }\,
  \chi_{ A, \eta }^{ \beta, T }
  \left(
  \hat{L}_1
  +
  L_1
  \sup\nolimits_{ t \in (0,T] }
    t^\lambda\,
    \|
      Y_t
    \|_{
      \lpn{p}{\P}{H}
    }
  \right)
  < \infty,
\end{split}
\end{equation}
\item
\label{item:diff.temporal}
and for all 
$\varrho\in[0,\nicefrac{1}{2}-\beta)$, 
$s,t\in(0,T]$ 
with $s<t$ 
it holds that 
\begin{equation}
\begin{split}
&
  \|\bar{Y}_t-\bar{Y}_s\|_{\lpn{p}{\P}{H}}
\leq
  | T \vee 1 |^{ |\lambda-\hat{\beta}| \1_{(0,\infty)}(L_1) }
  \left(
  \hat{L}_1
  +
  L_1
  \sup\nolimits_{ u \in (0,T] }
    u^\lambda\,
    \|
      Y_u
    \|_{
      \lpn{p}{\P}{H}
    }
  \right)
  \left|
    t - s
  \right|^{
    \varrho
  }
\\&\cdot
  \sqrt{\tfrac{p\,(p-1)}{2}} \,
  \bigg[
  \tfrac{
  \chi^{
    \beta, T
  }_{A,\eta}\,
    \left|
      t - s
    \right|^{
      ( 1/2 - \beta - \varrho )
    }
  }{
    \min\{
    s^\rho
    ,
    t^\rho
    \} \,
    \sqrt{1-2\beta}
  }
 +
    \tfrac{
  \kappa^{
    \varrho, T
  }_{A,\eta}
  \,
  \chi^{
    \varrho + \beta, T
  }_{A,\eta} \,
      |
      \mathbb{B}( 1 - 2 \beta - 2 \varrho, 1 - 2 \rho )
      |^{1/2}
    }{
      s^{ ( \rho + \varrho + \beta - 1/2 ) }
    }
  \bigg]
  .
\end{split}
\end{equation}
\end{enumerate}
\end{lemma}
\begin{proof}
Throughout this proof let 
$ K \in [0,\infty) $ 
be the real number given by 
$
  K
  =
  \sup_{ t \in (0,T] }
  t^\lambda \,
  \| Y_t \|_{ \lpn{p}{\P}{H} }
$.
We observe that~\eqref{eq:G_Lipschitz} implies 
for all $ t \in (0,T] $
that
\begin{equation}
\label{eq:B_finite}
\begin{split}
&
  \int_0^t
    \left\|
      e^{ ( t - s ) A }
      \mathbf{B}( s, Y_s )
    \right\|_{
      L^p(\P;HS(U,H))
    }^2
  \diffns s
\\ & \leq
  |\chi_{ A, \eta }^{ \beta, T }|^2
  \int_0^t
    \left(
      t - s
    \right)^{
      - 2 \beta
    }
    \left(
      L_1
      \left\|
        Y_s
      \right\|_{
        L^p( \P; H )
      }
      +
      \hat{L}_1
      \,
      s^{ - \hat{ \beta } }
    \right)^{ \! 2 }
  \diffns s
\\& \leq
  |\chi_{ A, \eta }^{ \beta, T }|^2
  \,
  ( K L_1 + \hat{L}_1 )^2
  \int_0^t
    \left( t - s \right)^{
      - 2 \beta
    }
      \max\!\big\{
        s^{ - 2( \lambda + (\hat{\beta}-\lambda)\mathbbm{1}_{\{0\}}(L_1) ) },
        s^{ - 2 \hat{ \beta } }
      \big\}
    \,
  \diffns s
\\& \leq
  |\chi_{ A, \eta }^{ \beta, T }|^2
  \,
  ( K L_1 + \hat{L}_1 )^2\,
  | T \vee 1 |^{ 2| \lambda - \hat{\beta} | \mathbbm{1}_{(0,\infty)}(L_1) }
  \int_0^t
      \left( t - s \right)^{
        - 2 \beta
      }
      s^{
        -2\rho
      }
    \,
  \diffns s
\\& \leq
  |\chi_{ A, \eta }^{ \beta, T }|^2
  \,
  ( K L_1 + \hat{L}_1 )^2\,
  | T \vee 1 |^{ 2| \lambda - \hat{\beta} | \mathbbm{1}_{(0,\infty)}(L_1) } \,
  \mathbb{B}
  \big(
    1-2\beta,1-2\rho
  \big)
  \,
  t^{ (1 - 2\beta - 2\rho) }
  .
\end{split}
\end{equation}
This implies, in particular, that
for all $ t \in [0,T] $ it holds $ \P $-a.s.\ that
$
  \int_0^t \| e^{ (t - s) A } \mathbf{B}( s, Y_s ) \|^2_{ HS( U, H ) } \, ds 
$
$
  < \infty
$.
In addition, \eqref{eq:B_finite}
and the Burkholder-Davis-Gundy type inequality
in Lemma~7.7 in Da Prato \& Zabczyk~\cite{dz92}
ensure that for all $ t \in (0,T] $
it holds that
\begin{equation}
\label{eq:B_finite2}
\begin{split}
&
  \left\|
    \int_0^t
    e^{ (t - s ) A } \mathbf{B}( s, Y_s )
    \, dW_s
  \right\|_{
    L^p( \P; H )
  }
\\&\leq
  \left[
    \frac{
      p \left( p - 1 \right)
    }{ 2 }
    \int_0^t
      \left\|
        e^{ ( t - s ) A }
        \mathbf{B}( s, Y_s )
      \right\|_{
        L^p(\P;HS(U,H))
      }^2
    \diffns s
  \right]^{ 1/2 }
\\ & \leq
  \chi_{ A, \eta }^{ \beta, T }
  \,
  ( K L_1 + \hat{L}_1 )\,
  | T \vee 1 |^{ | \lambda - \hat{\beta} | \mathbbm{1}_{(0,\infty)}(L_1) }
  \sqrt{
    \tfrac{ p \, (p-1) }{2} \,
    \mathbb{B}
    \big(
      1-2\beta,1-2\rho
    \big)
  }
  \,
  t^{ (\nicefrac{1}{2} - \beta - \rho) }
  .
\end{split}
\end{equation}
Furthermore, we observe that
the Burkholder-Davis-Gundy type inequality in
Lemma~7.7 in Da Prato \& Zabczyk~\cite{dz92}
proves that for all
$
  \varrho \in
  [ 0 , \nicefrac{1}{2} - \beta )
$,
$ t_1, t_2 \in (0,T] $
with $ t_1 < t_2 $
it holds that
\begin{equation}
\begin{split}
&
  \left\|
    \int_0^{ t_2 }
      e^{(t_2-s)A}
      \mathbf{B}( s, Y_s )
      \,
    \diffns W_s
    -
    \int_0^{ t_1 }
      e^{ ( t_1 - s ) A }
      \mathbf{B}( s, Y_s )
      \,
    \diffns W_s
  \right\|_{
    L^p(\P;H)
  }
\\ & \leq
  \left[
    \tfrac{
      p \, \left( p - 1 \right)
    }{
      2
    }
    \int_0^{ t_1 }
    \left\|
      \left(
        \operatorname{Id}_H
        - e^{ ( t_2 - t_1 ) A }
      \right)
      e^{ (t_1 - s ) A }
      \mathbf{B}( s , Y_s )
    \right\|_{
      L^p( \P; HS( U, H ) )
    }^2
    \diffns s
  \right]^{ 1/2 }
\\ &
  \quad+
  \left[
    \tfrac{
      p \, \left( p - 1 \right)
    }{
      2
    }
    \int_{ t_1 }^{ t_2 }
    \left\|
      e^{ ( t_2 - s ) A }
      \mathbf{B}( s , Y_s )
    \right\|_{
      L^p( \P; HS( U, H ) )
    }^2
    \diffns s
  \right]^{ 1/2 }
\\& \leq
  \left[
    \tfrac{
      p \, \left( p - 1 \right)
    }{
      2
    }
    \int_{ t_1 }^{ t_2 }
      \left\|
        e^{ ( t_2 - s ) A }
      \right\|_{
        L( H_{ - \beta }, H )
      }^2
      \left\|
        \mathbf{B}( s , Y_s )
      \right\|_{
        L^p( \P; HS( U, H_{ - \beta } ) )
      }^2
    \diffns s
  \right]^{ 1/2 }
\\ &
    \quad+
  \left[
    \tfrac{
      p \, \left( p - 1 \right)
    }{
      2
    }  
    \int_0^{ t_1 }
      \left\|
        e^{ ( t_1 - s ) A }
      \right\|_{
        L( H_{ - \beta }, H_{ \varrho } )
      }^2
      \left\|
        \mathbf{B}( s , Y_s )
      \right\|_{
        L^p( \P; HS( U, H_{ - \beta } ) )
      }^2
    \diffns s
  \right]^{ 1/2 }
\\&\quad\quad\cdot
    \left\|
      \left(
        \operatorname{Id}_H - e^{ ( t_2 - t_1 ) A }
      \right)
    \right\|_{
      L( H_{ \varrho } , H )
    }
  .
\end{split}
\end{equation}
Assumption~\eqref{eq:G_Lipschitz} hence ensures that for all
$
  \varrho \in
  [ 0 , \nicefrac12 - \beta )
$,
$ t_1, t_2 \in (0,T] $
with $ t_1 < t_2 $
it holds that 
\begin{equation}
\begin{split}
\label{eq:Holder_B}
&
  \left\|
    \int_0^{ t_2 }
      e^{(t_2-s)A}
      \mathbf{B}( s, Y_s )
      \,
    \diffns W_s
    -
    \int_0^{ t_1 }
      e^{ ( t_1 - s ) A }
      \mathbf{B}( s, Y_s )
      \,
    \diffns W_s
  \right\|_{
    L^p(\P;H)
  }
\\ &\leq
  \kappa^{
    \varrho, T
  }_{A,\eta}
  \,
  \chi^{
    \varrho + \beta, T
  }_{A,\eta}
  \left|
    t_2 - t_1
  \right|^{
    \varrho
  }
  \bigg[
    \tfrac{
      p \, \left( p - 1 \right)
    }{
      2
    }
    \int_0^{ t_1 }
      \left( t_1 - s \right)^{
        - ( 2 \beta + 2 \varrho )
      }
      \left(
        L_1
        \left\|
          Y_s
        \right\|_{
          \lpn{p}{\P}{H}
        }
        +
        \hat{L}_1
        \,
        s^{
          - \hat{\beta}
        }
      \right)^{ \! 2 }
    \diffns s
  \bigg]^{ 1/2 }
\\ &
  \quad+
  \chi^{
    \beta, T
  }_{A,\eta}
  \bigg[
    \tfrac{
      p \, \left( p - 1 \right)
    }{
      2
    }
    \int_{ t_1 }^{ t_2 }
      \left(
        t_2 - s
      \right)^{
        - 2 \beta
      }
      \left(
        L_1
        \left\|
          Y_s
        \right\|_{
          \lpn{p}{\P}{H}
        }
        +
        \hat{L}_1
        \,
        s^{
          - \hat{\beta}
        }
      \right)^{ \! 2 }
    \diffns s
  \bigg]^{ 1/2 }
\\ &\leq
  \sqrt{
    \tfrac{
      p \, \left( p - 1 \right)
    }{
      2
    }
  }
  \kappa^{
    \varrho, T
  }_{A,\eta}
  \,
  \chi^{
    \varrho + \beta, T
  }_{A,\eta} \,
  ( K L_1 + \hat{L}_1 )\,
  | T \vee 1 |^{ | \lambda-\hat{\beta} | \mathbbm{1}_{(0,\infty)}(L_1) }
  \left|
    t_2 - t_1
  \right|^{
    \varrho
  }
\\&\quad\cdot
  \left[
  \int^{ t_1 }_0
  ( t_1 - s )^{ -( 2\beta + 2\varrho ) }
  s^{ -2 \rho }
  \,\diffns{s}
  \right]^{ 1/2 }
+
  \tfrac{
  \chi^{
    \beta, T
  }_{A,\eta} \,
    | T \vee 1 |^{ | \lambda - \hat{\beta} | \mathbbm{1}_{(0,\infty)}(L_1) }
    \,
    (
      K L_1
      +
      \hat{L}_1
    )
    \sqrt{
    \frac{ p\,(p-1) }{ 2 }
    \left|
      t_2 - t_1
    \right|^{
      ( 1 - 2\beta )
    }
    }
  }{
    \min\{
    |
      t_1
    |^\rho
    ,
    |
      t_2
    |^\rho
    \}
    \sqrt{
      1 -
      2 \beta
    }
  }
\\ &
  \leq
  \bigg[
    \tfrac{
  \kappa^{
    \varrho, T
  }_{A,\eta}
  \,
  \chi^{
    \varrho + \beta, T
  }_{A,\eta} \,
  \left|
    t_2 - t_1
  \right|^{
    \varrho
  }
      \sqrt{
      \mathbb{B}( 1 - 2 \beta - 2 \varrho, 1 - 2 \rho )
      }
    }{
      |t_1|^{ ( \rho + \varrho + \beta - 1/2 ) }
    }
  +
  \tfrac{
  \chi^{
    \beta, T
  }_{A,\eta}\,
    \left|
      t_2 - t_1
    \right|^{
      ( 1/2 - \beta )
    }
  }{
    \min\{
    |
      t_1
    |^\rho
    ,
    |
      t_2
    |^\rho
    \} \,
    \sqrt{1-2\beta}
  }
  \bigg]
\\&\quad\cdot
    \sqrt{\tfrac{ p \, (p-1) }{2}} \,
    | T \vee 1 |^{ | \lambda - \hat{\beta} | \mathbbm{1}_{(0,\infty)}(L_1) }
    \,
    (
      K L_1
      +
      \hat{L}_1
    )
  .
\end{split}
\end{equation}
Combining
\eqref{eq:B_finite2},
\eqref{eq:Holder_B},
and
Lemma~\ref{lem:predictability}
completes the proof
of Lemma~\ref{lem:integrals2}.
\end{proof}

\begin{lemma}
\label{lem:singular.space.completeness}
Let
$
  \left(
    V_k,
    \left\| \cdot \right\|_{V_k}
  \right)
$, 
$ k \in \{0,1\} $, 
be separable $ \R $-Banach spaces 
with $ V_1 \subseteq V_0 $ continuously and densely,  
let 
$ T \in (0,\infty) $, 
$ \lambda \in \R $, 
$ p \in [1,\infty) $, 
let
$
  ( \Omega , \mathcal{F}, \P, ( \mathcal{F}_t )_{ t \in [0,T] } )
$
be a filtered probability space, 
let 
$
  \mathcal{L} \subseteq  
    \mathcal{M}(
      \mathrm{Pred}( ( \mathcal{F}_t )_{ t \in [0,T] } )
      ,
      \mathcal{B}( V_0 )
    )
$ 
be the set given by 
\begin{equation}
\begin{split}
&
  \mathcal{L}
  =
  \left\{
  \substack{
    X \in
    \mathcal{M}(
      \mathrm{Pred}( ( \mathcal{F}_t )_{ t \in [0,T] } )
      ,
      \mathcal{B}( V_0 )
    )
    \colon
    X( (0,T] \times \Omega ) \subseteq V_1,
\\ 
    \| X_0 \|_{
      \mathcal{L}^p( \P; V_0 )
    }
    +
    \sup_{ t \in (0,T] }
    t^{ \lambda }
    \,
    \| X_t \|_{
      \mathcal{L}^p( \P ; V_1 )
    }
    < \infty
  }
  \right\}
  ,
\end{split}
\end{equation}
let 
$
  \left| \cdot \right|_\mathcal{L}
  \colon
  \mathcal{L}
  \to [0,\infty)
$ 
be the mapping which satisfies for all 
$
  X \in
  \mathcal{L}
$
that
\begin{equation}
  \left|
    X
  \right|_\mathcal{L}
  =
  \|X_0\|_{ \mathcal{L}^p( \P ; V_0 ) }
  +
  \sup\nolimits_{ t \in (0,T] }
  \left[
    t^{ \lambda }
    \,
    \|
      X_t
    \|_{ \mathcal{L}^p( \P ; V_1 ) }
  \right]
  ,
\end{equation}
and let 
$ X^N \in \mathcal{L} $, $ N \in \N $, 
satisfy 
$
  \limsup_{ N \to \infty }
  \sup_{ n,m \in \N \cap [N,\infty) }
  | X^n - X^m |_\mathcal{L}
  =0
$. 
Then there exists a $ Y \in \mathcal{L} $ such that 
$
  \limsup_{ N \to \infty }
  | X^N - Y |_\mathcal{L}
  =
  0
$.
\end{lemma}
\begin{proof}
Throughout this proof 
let $ N_k \in \N $, $ k \in \N $, 
be a strictly increasing sequence such that 
for all $ k \in \N $ it holds that 
$
  | X^{N_{k+1}} - X^{N_k} |_\mathcal{L}
  < \frac{1}{2^k}
$, 
let $ \mathcal{Y} \colon [0,T] \times \Omega \to V_0 $ 
be the mapping with the property for all 
$ (t,\omega) \in [0,T] \times \Omega $ 
it holds that 
\begin{equation}
  \mathcal{Y}_t( \omega )
  =
  \begin{cases}
    \lim_{ k \to \infty }
    X^{ N_k }_t( \omega )
  &
    \colon
    ( X^{ N_k }_t( \omega ) )_{ k \in \N }
    \text{ is convergent in } V_0
  \\
    0
  &
    \colon
    \text{else}
  \end{cases}
  ,
\end{equation}
let $ \phi \colon V_0 \to V_0 $ 
be the mapping with the property that for all 
$ x \in V_0 $ 
it holds that 
$
  \phi(x)
  =
  \1_{V_1}(x) \cdot x
$, 
and let 
$
  Y \colon [0,T] \times \Omega \to V_0
$ 
be the mapping with the property for all 
$ (t,\omega) \in [0,T] \times \Omega $ 
it holds that 
$
  Y_t(\omega)
  =
  \phi\big(
    \1_{ (0,T] \times \Omega }(t,\omega) \cdot \mathcal{Y}_t(\omega)
  \big)
  +
  \1_{ \{0\} \times \Omega }(t,\omega) \cdot \mathcal{Y}_0(\omega)
$. 
The assumption that 
$ 
  \forall \, N \in \N \colon 
  X^N \in 
    \mathcal{M}(
      \mathrm{Pred}( ( \mathcal{F}_t )_{ t \in [0,T] } )
      ,
      \mathcal{B}( V_0 )
    )
$ 
and, e.g., Exercise~1.74 in Chapter~1 in Hoffmann-J\o{}rgensen~\cite{Hoffmann1994}
imply that 
$
  \big\{
    ( t, \omega ) \in [0,T] \times \Omega
    \colon
    ( X^{ N_k }_t( \omega ) )_{ k \in \N }
    \text{ is convergent in } V_0
  \big\}
  \in
  \mathrm{Pred}( ( \mathcal{F}_t )_{ t \in [0,T] } )
$. 
This together with the assumption that 
$ 
  \forall \, N \in \N \colon 
  X^N \in 
    \mathcal{M}(
      \mathrm{Pred}( ( \mathcal{F}_t )_{ t \in [0,T] } )
      ,
      \mathcal{B}( V_0 )
    )
$ 
and, e.g., Exercise~1.74 in Chapter~1 in Hoffmann-J\o{}rgensen~\cite{Hoffmann1994}
ensure that 
$
  \mathcal{Y} \in
    \mathcal{M}(
      \mathrm{Pred}( ( \mathcal{F}_t )_{ t \in [0,T] } )
      ,
      \mathcal{B}( V_0 )
    )
$.
Furthermore, observe that, e.g., Lemma~\ref{lem:Kuratowski} and the fact that 
\begin{equation}
  \forall \, A \in \mathcal{B}(V_0) \colon
  \phi^{-1}(A)
  =
  \phi^{-1}(A \cap V_1)
  =
  \begin{cases}
    A \cap V_1
  &
    \colon
    0 \notin A
  \\
    ( V_0 \setminus V_1) \cup ( A \cap V_1 )
  &
    \colon
    0 \in A
  \end{cases}  
\end{equation}
ensure that 
$
  \phi \in \mathcal{M}( \mathcal{B}(V_0), 
  \mathcal{B}(V_0) )
$. 
Combining this with the fact that 
$
  \mathcal{Y} \in 
    \mathcal{M}(
      \mathrm{Pred}( ( \mathcal{F}_t )_{ t \in [0,T] } )
      ,
$
$
      \mathcal{B}( V_0 )
    )
$
establishes that 
$
  Y \in 
    \mathcal{M}(
      \mathrm{Pred}( ( \mathcal{F}_t )_{ t \in [0,T] } )
      ,
      \mathcal{B}( V_0 )
    )
$
and 
$
  Y( (0,T] \times \Omega ) \subseteq V_1
$.
In the next step we note that the assumption that 
$
  \limsup_{ N \to \infty }
  \sup_{ n,m \in \N \cap [N,\infty) }
  | X^n - X^m |_\mathcal{L}
  =0
$ 
shows that for all $ t \in [0,T] $ it holds that 
$
  \limsup_{ N \to \infty }
  \sup_{ n, m \in \N \cap [N,\infty) }
  \| X^n_t - X^m_t \|_{ \mathcal{L}^p( \P; V_{\1_{(0,T]}(t)} ) }
  =0
$. 
Hence, we obtain for every $ t \in [0,T] $ that there exists a 
$
  \mathfrak{Y}_t \in
  \mathcal{L}^p( \P\vert_{\mathcal{F}_t} ; V_{\1_{(0,T]}(t)} )
$ 
such that 
$
  \limsup_{ N \to \infty }
  \| X^N_t - \mathfrak{Y}_t \|_{ \mathcal{L}^p( \P; V_{\1_{(0,T]}(t)} ) }
  =0
$.
The fact that 
$
  \forall \, k \in \N
  \colon
  | X^{N_{k+1}} - X^{N_k} |_\mathcal{L}
  < \frac{1}{2^k}  
$ 
therefore proves that for every $ t \in [0,T] $ there exists a 
$
  \mathfrak{Y}_t \in
  \mathcal{L}^p( \P\vert_{\mathcal{F}_t} ; V_{\1_{(0,T]}(t)} )
$ 
such that for all $ n \in \N $ it holds that 
\begin{equation}
\begin{split}
&
  \|
  \mathfrak{Y}_t - X^{N_n}_t  
  \|_{ \mathcal{L}^p( \P; V_{\1_{(0,T]}(t)} ) }
\\&\leq
  \limsup_{ m \to \infty }
  \big(
    \|
    \mathfrak{Y}_t - X^{N_m}_t  
    \|_{ \mathcal{L}^p( \P; V_{\1_{(0,T]}(t)} ) }
    +
    \|
    X^{N_m}_t - X^{N_n}_t  
    \|_{ \mathcal{L}^p( \P; V_{\1_{(0,T]}(t)} ) }
  \big)
\\&=
  \limsup_{ m \to \infty }
  \big\|
    \smallsum^{m-1}_{ k=n }
    \big( X^{N_{k+1}}_t - X^{N_k}_t \big) 
  \big\|_{ \mathcal{L}^p( \P; V_{\1_{(0,T]}(t)} ) }
\\&\leq
  \sum^\infty_{ k=n }
  \|
    X^{N_{k+1}}_t - X^{N_k}_t 
  \|_{ \mathcal{L}^p( \P; V_{\1_{(0,T]}(t)} ) }
\\&\leq
  \sum^\infty_{ k=n }
  t^{ -\1_{(0,T]}(t) \cdot \lambda } \,
  | X^{N_{k+1}}_t - X^{N_k}_t |_{\mathcal{L}}  
  \leq
  t^{ -\1_{(0,T]}(t) \cdot \lambda }
  \left(
  \sum^\infty_{ k=n }
  \frac{1}{2^k}
  \right)
  =
  t^{ -\1_{(0,T]}(t) \cdot \lambda } \,
  2^{ (1-n) }
  .    
\end{split}
\end{equation}
This and, e.g., item~(ii) of Theorem~6.12 in Klenke~\cite{Klenke2008} 
assure that for every $ t \in [0,T] $ there exists a 
$
  \mathfrak{Y}_t \in
  \mathcal{L}^p( \P\vert_{\mathcal{F}_t} ; V_{\1_{(0,T]}(t)} )
$ 
such that for all $ n \in \N $ it holds that 
$
  \|
    \mathfrak{Y}_t - X^{N_n}_t
  \|_{ \mathcal{L}^p( \P; V_{\1_{(0,T]}(t)} ) }
  \leq
  t^{ -\1_{(0,T]}(t) \cdot \lambda } \,
  2^{ (1-n) }  
$ 
and 
$
  \P(
    \cap_{ k \in \N } \cup_{ M \in \N } 
    \cap_{ m \in \N \cap [M,\infty) } 
    \{
    \| \mathfrak{Y}_t - X^{N_m}_t \|_{ V_{\1_{(0,T]}(t)} }
    < \frac{1}{k}
    \}
  )
  =
  \P(
    \limsup_{ m \to \infty }
    \| \mathfrak{Y}_t - X^{N_m}_t \|_{ V_{\1_{(0,T]}(t)} }
    = 0
  )
  = 1
  .
$
The assumption that 
$ V_1 \subseteq V_0 $ continuously 
hence ensures that for all $ t \in [0,T] $, $ n \in \N $ 
it holds that 
$
  \|
    Y_t - X^{N_n}_t
  \|_{ \mathcal{L}^p( \P; V_{\1_{(0,T]}(t)} ) }
  \leq
  t^{ -\1_{(0,T]}(t) \cdot \lambda } \,
  2^{ (1-n) }  
$. 
This shows that for all $ n \in \N $ it holds that 
$
  \|
    Y_0 - X^{N_n}_0
  \|_{ \mathcal{L}^p( \P; V_0 ) }
  +
  \sup_{ t \in (0,T] }
  [
  t^\lambda \,
  \|
    Y_t - X^{N_n}_t
  \|_{ \mathcal{L}^p( \P; V_1 ) } 
  ]
  \leq
  2^{(2-n)}
$.
Therefore, we get that for all $ n \in \N $ 
it holds that $ Y - X^{N_n} \in \mathcal{L} $ 
and 
$
  | Y - X^{N_n} |_{\mathcal{L}}
  \leq
  2^{(2-n)}
$. 
Hence, we obtain that $ Y \in \mathcal{L} $ 
and 
$
  \limsup_{n\to\infty}
  | Y - X^{N_n} |_{\mathcal{L}}  
  = 0
$.
This completes the proof of Lemma~\ref{lem:singular.space.completeness}.
\end{proof}

\subsection{A perturbation estimate for stochastic processes}
\label{sec:perturbation}

Lemma~\ref{lem:gronwall} is a consequence of the generalized Gronwall inequality from Lemma~7.1.1 in Chapter~7 in Henry~\cite{h81} 
(cf. also Exercise~4 in Chapter~7 in Henry~\cite{h81}).

\begin{lemma}
\label{lem:gronwall}
Let $ \alpha, \beta \in (-\infty,1) $,
$ a, b \in [0,\infty) $,
$ T \in (0,\infty) $,
$
  e \in
  \mathcal{M}\big(
    \mathcal{B}( [0,T] ) ,
    \mathcal{B}( [0,\infty] )
  \big)
$
satisfy for all
$ t \in (0,T] $	that  
$
  \int_0^T e(s) \, \diffns s < \infty
$
and 
$
  e(t) \leq
  \frac{ a }{ t^{ \alpha } }
  +
  \int_0^t
  \frac{ b \, e(s) }{ \left( t - s \right)^{ \beta } }
  \,
  \diffns s
$.
Then for all
$ t \in (0,T] $
it holds that
$
  e(t)
  \leq
  \frac{ a }{ t^{ \alpha } }
  \mathrm{E}_{ \alpha, \beta }\!\left[
    b \, t^{ (1-\beta) }
  \right]
$.
\end{lemma}

In the next result, Proposition~\ref{prop:perturbation_estimate}, 
we prove a strong perturbation result that
will be used several times throughout the paper. 
We refer to~\eqref{eq:BigTheta} in Subsection~\ref{sec:notation} above for the introduction of the real numbers 
$
  \Theta^{\alpha,\beta,\lambda}_{ A, \eta, p, T }(L_0,L_1)
$
appearing on the right hand side of inequality~\eqref{eq:pertubation1} 
in Proposition~\ref{prop:perturbation_estimate}.

\begin{prop}[Perturbation estimate]
\label{prop:perturbation_estimate}
Assume the setting in Section \ref{sec:SPDE_setting}, 
let $ \delta \in \R $,  
and let 
$ Y^1 , Y^2 \colon [0,T] \times \Omega \to H_{\delta} $ 
be $ ( \mathcal{F}_t )_{ t \in [0,T] } $-predictable
stochastic processes 
which satisfy 
$ 
  \cup_{ k \in \{ 1, 2 \} }
  Y^k ( (0,T] \times \Omega ) \subseteq  H 
$
 and
\begin{equation}
\label{eq:perturb.condition.II}
   \limsup_{ 
     \lambda \nearrow 
     \frac{
       1
     }{ 2 }
       [
         1 +
         \1_{ \{ 0 \} }( L_1 )
       ] 
   }
   \max_{ k \in \{ 1, 2 \} }
   \sup_{ t \in ( 0 , T ] }
   t^{ \lambda }
   \,
   \|
     Y_t^k
   \|_{
     L^p ( \P ; H )
   }
   < \infty
   .
\end{equation}
Then 
\begin{enumerate}[(i)]
\item
it holds for all 
$ t \in [0,T] $ 
that
\begin{equation}
  \P
  \big(
  \smallsum^2_{ k=1 }
  \int_0^t
  \|
    e^{  (t - s) A}
    \mathbf{F}( s , Y^k_s )
  \|_H
  +
  \|
    e^{  (t - s) A}
    \mathbf{B}( s , Y^k_s )
  \|_{ HS( U, H ) }^2
  \,
  \diffns s
  < \infty
  \big)
  =1
\end{equation}
and 
\item
it holds for all 
$
  \lambda \in 
  \big( 
    - \infty, \frac{1}{2}[ 1  + \mathbbm{1}_{\{0\}}( L_1 ) ] 
  \big)
$
that 
\begin{equation}
\begin{split}
\label{eq:pertubation1}
&
  \sup_{
    t \in (0,T]
  }
  \left[
  t^{ \lambda }
  \,
  \|
    Y_t^1 - Y_t^2
  \|_{
    L^p ( \P ; H )
  }
  \right]
\leq
  \Theta_{ A , \eta , p , T }^{ \alpha , \beta , \lambda }( L_0, L_1 )
\\ &
  \cdot
  \sup_{ t \in ( 0 , T ] }
  \bigg[
  t^{ \lambda }
  \,
  \bigg\|
    Y_t^1
    -
    \int_0^t
      e^{(t-s)A}
      \mathbf{F}(s,Y_s^1)
    \, \diffns s
    -
    \int_0^t
      e^{(t-s)A}
      \mathbf{B}(s,Y_s^1)
    \, \diffns W_s\\
&   +
    \int_0^t
      e^{(t-s)A}
      \mathbf{F}(s,Y_s^2)
    \, \diffns s
    +
    \int_0^t
      e^{(t-s)A}
      \mathbf{B}(s,Y_s^2)
    \, \diffns W_s
    -
    Y_t^2
  \bigg\|_{
    L^p ( \P ; H )
  }
  \bigg]
  .
\end{split}
\end{equation}
\end{enumerate}
\end{prop}

\begin{proof}
Throughout this proof let 
$
  r \in ( -\infty, \frac{1}{2}[ 1  + \mathbbm{1}_{\{0\}}( L_1 ) ] )
$
and let 
$ \perturb \in [0,\infty] $ 
be the extended real number given by 
\begin{equation}
\begin{split}
&
  \perturb=
  \sup_{ t \in ( 0 , T ] }
  \bigg[
  t^r
  \,
  \bigg\|
    Y_t^1
    -
    \int_0^t
      e^{(t-s)A}
      \mathbf{F}(s,Y_s^1)
    \, \diffns s
    -
    \int_0^t
      e^{(t-s)A}
      \mathbf{B}(s,Y_s^1)
    \, \diffns W_s\\
&   +
    \int_0^t
      e^{(t-s)A}
      \mathbf{F}(s,Y_s^2)
    \, \diffns s
    +
    \int_0^t
      e^{(t-s)A}
      \mathbf{B}(s,Y_s^2)
    \, \diffns W_s
    -
    Y_t^2
  \bigg\|_{
    L^p ( \P ; H )
  }
  \bigg]
  .
\end{split}
\end{equation}
We observe that item~\eqref{item:integral.defined} of 
Lemma~\ref{lem:integrals}
and item~\eqref{item:stoch.integral.defined} of Lemma~\ref{lem:integrals2}
establish that for all 
$ t \in [0,T] $ 
it holds that 
$
  \P
  \big(
  \sum^2_{ k=1 }
  \int_0^t
  \|
    e^{  (t - s) A}
    \mathbf{F}( s , Y^k_s )
  \|_H
  +
  \|
    e^{  (t - s) A}
    \mathbf{B}( s , Y^k_s )
  \|_{ HS( U, H ) }^2
  \,
  \diffns s
  < \infty
  \big)
  =1
$. 
It thus remains to prove~\eqref{eq:pertubation1}.
For this we assume without loss of generality in the following that 
$ \perturb < \infty $.
Next we note that the triangle inequality shows that
for all $ t \in ( 0 , T ] $ it holds that
\begin{equation}
\begin{split}
  \big\|
    Y_t^1 - Y_t^2
  \big\|_{
    L^p ( \P ; H )
  }
& \leq
  \bigg\|
    Y_t^1
    -
    \int_0^t
      e^{(t-s)A}
      \mathbf{F}(s,Y_s^1)
    \, \diffns s
    -
    \int_0^t
      e^{(t-s)A}
      \mathbf{B}(s,Y_s^1)
    \, \diffns W_s
\\
&
\quad
  +
    \int_0^t
      e^{ ( t - s ) A }
      \mathbf{ F } ( s , Y_s^2 )
    \, \diffns s
    +
    \int_0^t
      e^{ ( t - s ) A }
      \mathbf{B} ( s , Y_s^2 )
    \, \diffns W_s
    -
    Y_t^2
  \bigg\|_{
    L^p( \P ; H )
  }
\\ & \quad +
  \left\|
    \int_0^t
      e^{ ( t - s ) A }
      \big(
        \mathbf{F} ( s , Y_s^1 )
        -
        \mathbf{F} ( s , Y_s^2 )
      \big)
    \,
    \diffns s
  \right\|_{
    L^p(\P;H)
  }
\\ & \quad +
  \left\|
    \int_0^t
      e^{(t-s)A}
      \big(
        \mathbf{B} ( s , Y_s^1 )
        -
        \mathbf{B} ( s , Y_s^2 )
      \big)
    \,\diffns W_s
  \right\|_{
    L^p(\P;H)
  }.
\end{split}
\end{equation}
This and the Burkholder-Davis-Gundy type inequality in
Lemma~7.7 in Da Prato \& Zabczyk~\cite{dz92} 
imply that for all $ t \in (0,T] $ it holds that
\begin{equation}
\label{eq:perturbation.intermediate}
\begin{split}
& \big\|
    Y_t^1 - Y_t^2
  \big\|_{
    L^p ( \P ; H )
  }
\leq
  t^{ - r } \,
  \perturb
+
  \chi_{ A , \eta }^{ \alpha , T }
  \,
  L_0
  \int_0^t
    ( t - s )^{ - \alpha }
    \,
    \|
      Y_s^1
      -
      Y_s^2
    \|_{
      L^p( \P ; H )
    }
    \,\diffns s
\\ & +
  \chi_{ A , \eta }^{ \beta , T }\,
  L_1
  \bigg[
    \tfrac{
      p \, ( p - 1 )
    }{
      2
    }
    \int_0^t
      \,
      (t-s)^{ - 2 \beta }
      \,
      \|
        Y_s^1
        -
        Y_s^2
      \|_{
        L^p( \P ; H )
      }^2
    \,\diffns s
  \bigg]^{ 1/2 }
  .
\end{split}
\end{equation}
Combining this with Lemma~\ref{lem:gronwall} proves~\eqref{eq:pertubation1}
in the case $ L_1 = 0 $.
It thus remains to prove~\eqref{eq:pertubation1}
in the case $ L_1 > 0 $.
For this we observe that~\eqref{eq:perturbation.intermediate} together with H\"{o}lder's inequality ensures
that for all $ t \in (0,T] $ it holds that
\begin{equation}
\begin{split}
&
  \big\|
    Y_t^1 - Y_t^2
  \big\|_{
    L^p ( \P ; H )
  }
\leq
  t^{ - r } \,
  \perturb
\\&+
  \chi_{ A , \eta }^{ \alpha , T }
  \,
  L_0
  \bigg[
    T^{
      \max\{
        2 \beta - \alpha , 0
      \}
    }
    \int_0^t
      (t - s)^{ - \alpha }
    \, \diffns s
    \int_0^t
      (t - s)^{
        - \max\{ \alpha , 2 \beta \}
      }
      \,
      \|
        Y_s^1
        -
        Y_s^2
      \|_{
        L^p( \P ; H )
      }^2
      \, \diffns s
    \bigg]^{ 1/2 }
\\ & +
  \chi_{ A, \eta }^{ \beta, T }
  \,
  L_1
  \bigg[
    \tfrac{
      p \, ( p - 1 )
    }{
      2
    }
    \,
    T^{
      \max\{ \alpha - 2 \beta , 0 \}
    }
    \int_0^t
      (t - s)^{
        - \max\{ \alpha , 2 \beta \}
      }
      \,
      \|
        Y_s^1
        -
        Y_s^2
      \|_{
        L^p( \P ; H )
      }^2
    \,
    \diffns s
  \bigg]^{1/2}
  .
\end{split}
\end{equation}
The fact that
$
  \forall \, a, b \in \R
  \colon
  ( a + b )^2 \leq 2 a^2 + 2 b^2
$
hence yields that for all $ t \in (0,T] $ it holds
that
\begin{equation}
\begin{split}
&
  \|
    Y_t^1 - Y_t^2
  \|_{
    L^p ( \P ; H )
  }^2
\leq
  \frac{
    2
  }{
    t^{ 2 r }
  } \,
  |\perturb|^2
  +
  \int_0^t
    ( t - s )^{
      - \max\{ \alpha , 2 \beta \}
    }
    \,
    \|
      Y_s^1
      -
      Y_s^2
    \|_{
      L^p( \P ; H )
    }^2
    \,\diffns s
\\ & \cdot
  \Big[
    \chi_{ A , \eta }^{ \alpha , T }
    \,
    L_0
    \,
    \tfrac{
      \sqrt{2}
      \,
      T^{
        1 / 2
        -
        \alpha
        +
        \max\{ \beta , \alpha / 2 \}
      }
    }{
      \sqrt{ 1 - \alpha }
    }
    +
    \chi_{ A , \eta }^{ \beta , T }
    \,
    L_1
    \,
    \sqrt{ p ( p - 1 ) }
    T^{ \max \{ \alpha/2 , \beta \} - \beta }
  \Big]^2
  .
\end{split}
\end{equation}
Combining this with Lemma \ref{lem:gronwall}
and the fact that
\begin{equation}
\begin{split}
&
  E_{ 2 r , \max\{ \alpha , 2 \beta \} }\!\left[
    T^{
      (1 - \max\{ \alpha , 2 \beta \})
    }
    \left|
      \chi_{ A , \eta }^{ \alpha , T }
      \,
      L_0
      \,
      \tfrac{
        \sqrt{2}\,
        T^{ 1 / 2
        -
        \alpha
        +
        \max\{ \beta , \alpha / 2 \}}
      }{
        \sqrt{ 1 - \alpha }
      }
      +
      \chi_{A,\eta}^{\beta,T}
      \,
      L_1
      \sqrt{ p \, ( p - 1 ) }
      \,
      T^{
        \max\{ \alpha / 2 , \beta \} - \beta
      }
    \right|^2
  \right]
\\ & =
  E_{
    2 r , \max\{ \alpha , 2 \beta \}
  }\!\left[
    \left|
      \chi_{ A , \eta }^{ \alpha , T }
      \,
      L_0
      \,
      \tfrac{
        \sqrt{2}\,
        T^{ (1 - \alpha) }
      }{
        \sqrt{ 1 - \alpha }
      }
      +
      \chi_{A,\eta}^{\beta,T}
      \,
      L_1
      \sqrt{ p \, ( p - 1 ) \, T^{ (1 - 2\beta) } }
    \right|^2
  \right]
\end{split}
\end{equation}
ensures that for all $ t \in (0,T] $ it holds that 
\begin{equation}
\begin{split}
&
  \|
    Y_t^1 - Y_t^2
  \|_{
    L^p ( \P ; H )
  }^2
\\&\leq
  \frac{
    2
  }{
    t^{ 2 r }
  } \,
  |\perturb|^2 \,
  E_{
    2 r , \max\{ \alpha , 2 \beta \}
  }\!\left[
    \left|
      \chi_{ A , \eta }^{ \alpha , T }
      \,
      L_0
      \,
      \tfrac{
        \sqrt{2}\,
        T^{ (1 - \alpha) }
      }{
        \sqrt{ 1 - \alpha }
      }
      +
      \chi_{A,\eta}^{\beta,T}
      \,
      L_1
      \sqrt{ p \, ( p - 1 ) \, T^{ (1 - 2\beta) } }
    \right|^2
  \right]
  .
\end{split}
\end{equation}
Hence, we obtain that 
\begin{equation}
\begin{split}
&
  \sup_{
    t \in (0,T]
  }
  \left[
  t^r
  \,
  \|
    Y_t^1 - Y_t^2
  \|_{
    L^p ( \P ; H )
  }
  \right]
\\&\leq
  \sqrt{2} \, \perturb
  \left|
  E_{
    2 r , \max\{ \alpha , 2 \beta \}
  }\!\left[
    \left|
      \chi_{ A , \eta }^{ \alpha , T }
      \,
      L_0
      \,
      \tfrac{
        \sqrt{2}\,
        T^{ (1 - \alpha) }
      }{
        \sqrt{ 1 - \alpha }
      }
      +
      \chi_{A,\eta}^{\beta,T}
      \,
      L_1
      \sqrt{ p \, ( p - 1 ) \, T^{ (1 - 2\beta) } }
    \right|^2
  \right]
  \right|^{ 1/2 }
  .
\end{split}
\end{equation}
This finishes the proof of Proposition~\ref{prop:perturbation_estimate}.
\end{proof}

In the next result, Corollary~\ref{cor:initial_perturbation}, we illustrate Proposition~2.5 by a simple example. 
In particular, Corollary~\ref{cor:initial_perturbation} ensures uniqueness of solutions of SEEs with singularities at the initial time.
We refer, e.g., to item~(i) of Theorem~7.4 in Da Prato \& Zabczyk~\cite{dz92} for an existence and uniqueness result
for SEEs without singularities at the initial time.

\begin{corollary}[Initial conditions]
\label{cor:initial_perturbation}
Assume the setting in Section \ref{sec:SPDE_setting}, 
let
$
  \delta \in
  \big[
    0
    ,
    \frac{
      1
    }{ 2 }
    [
      1 +
      \1_{ \{ 0 \} }( L_1 )
    ]
  \big)
$, 
and let 
$ X^1 , X^2 \colon [0,T] \times \Omega \to H_{ -\delta } $ 
be $ (\mathcal{F}_t)_{ t \in [0,T] } $-predictable stochastic  processes which fulfill 
for all 
$
  k \in \{ 1, 2 \}
$,
$ t \in [ 0 , T ] $ 
that 
$ X^k( (0,T] \times \Omega ) \subseteq H $, 
that 
$ 
  \limsup_{ 
    \lambda \nearrow
    \frac{
      1
    }{ 2 }
    [
      1 +
      \1_{ \{ 0 \} }( L_1 )
    ]
  }
  \sup_{ s \in (0,T] } 
  s^\lambda \, 
  \| X^k_s \|_{ \lpn{p}{\P}{H} } 
$
$
  < \infty 
$, 
and $ \P $-a.s.\ that
\begin{equation}
  X_t^k
  =
  e^{ t A }
  X_0^k
  +
  \int_0^t
    e^{ ( t - s ) A }
    \mathbf{F}( s , X_s^k )
  \, \diffns s
  +
  \int_0^t
    e^{ ( t - s ) A }
    \mathbf{B}( s , X_s^k )
  \, \diffns W_s
  .
\end{equation}
Then it holds 
for all 
$
  \lambda \in
  \big[ 
    \delta, \frac{1}{2}[ 1 + \mathbbm{1}_{ \{0\} }(L_1) ] 
  \big)
$ 
that 
\begin{equation}
\begin{split}
&
  \sup_{ t \in (0,T] }
  \left[
    t^{ \lambda }
    \,
    \|
      X_t^1 - X_t^2
    \|_{
      L^p( \P ; H )
    }
  \right]
\leq
  \chi_{ A, \eta }^{ \delta, T }
  \,
  T^{ ( \lambda - \delta ) }
  \|
    X_0^1 - X_0^2
  \|_{
    L^p( \P ; H_{ - \delta } )
  }
  \,
  \Theta_{ A , \eta, p , T }^{ \alpha, \beta, \lambda }( L_0, L_1 )
  .
\end{split}
\end{equation}
\end{corollary}

\subsection{Existence, uniqueness, and regularity for SEEs with singularities at the initial time}
\label{sec:existence_uniqueness}

In Theorem~\ref{thm:existence_uniqueness} below we establish existence, uniqueness, and regularity for SEEs with singularities at the initial time. The following remark helps to access the formulation of Theorem~\ref{thm:existence_uniqueness}.

\begin{remark}
Assume the setting in Section~\ref{sec:SPDE_setting} and let 
$
  \delta
  \in
  \big(
    -\infty
    ,
    \frac{
      1
    }{ 2 }
    [
      1 +
      \1_{ \{ 0 \} }( L_1 )
    ]
  \big)
$. 
Observe that the assumptions that 
$ \alpha < 1 $, 
$ \hat{\alpha} < 1 $, 
$ \beta < \nicefrac{1}{2} $, 
$ \hat{\beta} < \nicefrac{1}{2} $, 
and 
$
  \1_{(0,\infty)}(L_1) \cdot
  [ \alpha + \hat{\alpha} ]
  < \nicefrac{3}{2}
$
ensure that 
\begin{equation}
      \max\{
        \delta ,
        \alpha + \hat{ \alpha } - 1 ,
        \beta + \hat{ \beta } - \nicefrac{ 1 }{ 2 }
      \}
    <
    \tfrac{
      1
    }{ 2 }\,
    [
      1 +
      \1_{ \{ 0 \} }( L_1 )
    ].
\end{equation}
\end{remark}

We now present the promised existence, uniqueness, and regularity results for SEEs with singularities at the initial time.

\begin{theorem}
\label{thm:existence_uniqueness}
Assume the setting in Section~\ref{sec:SPDE_setting} 
and let 
$
  \delta
  \in
  \big(
    -\infty
    ,
    \frac{
      1
    }{ 2 }
    [
      1 +
      \1_{ \{ 0 \} }( L_1 )
    ]
  \big)
$,
$
  \lambda \in
      \big[\!
      \max\{
        \delta ,
        \alpha + \hat{ \alpha } - 1 ,
        \beta + \hat{ \beta } - \nicefrac{ 1 }{ 2 }
      \}
    ,
    \frac{
      1
    }{ 2 }
    [
      1 +
      \1_{ \{ 0 \} }( L_1 )
    ]
  \big)
$, 
$
  \rho =
  \max\{ \lambda + ( \hat{\beta} - \lambda ) \1_{\{0\}}(L_1), \hat{\beta} \}
$, 
$ \xi \in \mathcal{L}^p( \P|_{\mathcal{F}_0}; H_{-\max\{\delta,0\}} ) $
satisfy  
$
  \sup_{ t\in (0,T] }
  t^{ \delta}\,
  \| e^{tA} \xi \|_{ \lpn{p}{\P}{H} }
$
$
  < \infty
$.
Then
\begin{enumerate}[(i)]
\item
\label{item:thm_existence:existence}
there exists an up-to-modifications unique 
$ (\mathcal{F}_t)_{ t \in [0,T] } $-predictable stochastic process 
$ X \colon [0,T] \times \Omega \to H_{-\max\{\delta,0\}} $
which satisfies for all $ t \in [0,T] $
that 
$ X( (0,T] \times \Omega ) \subseteq H $, 
that 
$
  \sup_{ s \in (0,T] }
  s^\lambda
$
$
  \|X_s\|_{ \lpn{p}{\P}{H} }
  < \infty
$, 
that 
$
  \P\big(
  \int_0^t
  \|
    e^{  (t - s) A}
    \mathbf{F}( s , X_s )
  \|_H
  +
  \|
    e^{  (t - s) A}
    \mathbf{B}( s , X_s )
  \|_{ HS( U, H ) }^2
  \,
  \diffns s
  < \infty
  \big)
  =1
$,
and $ \P $-a.s.\ that
\begin{equation}
\label{eq:SEE}
  X_t
  =
  e^{t A } \xi
  +
  \int_0^t
    e^{  (t - s) A}
    \mathbf{F}( s , X_s )
    \,
  \diffns s
  +
  \int_0^t
    e^{  (t - s) A}
    \mathbf{B}( s , X_s )
    \,
  \diffns W_s,
\end{equation}
\item
\label{item:thm_existence:a_priori}
it holds that
\begin{equation}
\label{eq:item_ii_existence}
\begin{split}
&
  \sup_{ t \in (0,T] }
  \left[
    t^{ \lambda }
    \left\|
      X_t
    \right\|_{
      L^p( \P; H )
    }
  \right]
  \leq
  T^\lambda\,
  \Theta_{A,\eta,p,T}^{\alpha,\beta,\lambda}(L_0,L_1)
\\&\quad\cdot
  \bigg[
  \tfrac{
      \sup_{ t \in (0,T] }
      (
      t^\delta
      \|
        e^{tA}
        \xi
      \|_{
        L^p( \P; H )
      }
      )
  }{
    T^\delta
  }
+
    \tfrac{
      \chi^{ \alpha, T }_{ A, \eta }
      \,
      \hat{L}_0
      \,
      \mathbbm{B}(
        1 - \alpha
        ,
        1 - \hat{\alpha}
      )
    }{
      T^{
        \left(
          \alpha + \hat{\alpha} - 1
        \right)
      }
    }
    +
    \tfrac{
      \chi_{ A, \eta }^{ \beta , T }
      \,
      \hat{L}_1
      \left|
        p \, ( p - 1 )
        \,
        \mathbbm{B}(
          1 - 2 \beta
          ,
          1 - 2 \hat{\beta}
        )
      \right|^{ \nicefrac{ 1 }{ 2 } }
    }{
      \sqrt{2}\,
      T^{
        (
          \beta + \hat{ \beta } - 1/2
        )
      }
    }
  \bigg]
  < \infty,
\end{split}
\end{equation}
\item
\label{item:SEE.temporal}
and for all 
$\varrho\in[0,\min\{1-\alpha,\nicefrac{1}{2}-\beta\})$, 
$s,t\in(0,T]$ 
with $s<t$ 
it holds that 
\begin{equation}
\begin{split}
&
    \left\|
      X_s - X_t
    \right\|_{
      L^p( \P; H )
    }
  \leq
  |s-t|^\varrho
\\&\cdot
  \Bigg[
    \tfrac{
    \kappa^{\varrho,T}_{A,\eta} \,
    \chi^{\varrho+\max\{\delta,0\},T}_{A,\eta} \,
    \|\xi\|_{\lpn{p}{\P}{H_{-\max\{\delta,0\}}}}   
    }{ s^{(\varrho+\max\{\delta,0\})} }
+
    |T\vee1|^{|\lambda-\hat{\alpha}|}
    \left(
    \hat{L}_0
    +
    L_0
    \sup\nolimits_{ u \in (0,T] }
      u^\lambda\,
      \|
        X_u
      \|_{
        \lpn{p}{\P}{H}
      }
    \right)
\\&\cdot
    \bigg[
    \tfrac{
    \chi^{ \alpha, T }_{A,\eta} \,
      \left| s - t \right|^{ (1 - \alpha-\varrho) }
    }{
      \left(
        1 - \alpha
      \right)
      \min\{
      s^{
        \max\{
          \lambda,
          \hat{ \alpha }
        \}
      }
      ,
     t^{
        \max\{
          \lambda,
          \hat{ \alpha }
        \}
      }
      \}
    }
  +
    \tfrac{
    \kappa^{\varrho,T}_{A,\eta}
    \,
    \chi^{ \varrho + \alpha, T }_{A,\eta} \,
      \mathbbm{B}(
        1 -
        \alpha
        - \varrho
        ,
        1 -
        \max\{ \lambda, \hat{ \alpha } \}
      )
    }{
      s^{(
        \varrho + \alpha
        +
        \max\{ \lambda , \hat{ \alpha } \}
        - 1
      )}
    }
    \bigg]
\\&+
  \sqrt{\tfrac{p\,(p-1)}{2}} \,
  |T\vee1|^{ |\lambda-\hat{\beta}| \1_{(0,\infty)}(L_1) }
  \left(
  \hat{L}_1
  +
  L_1
  \sup\nolimits_{ u \in (0,T] }
    u^\lambda\,
    \|
      X_u
    \|_{
      \lpn{p}{\P}{H}
    }
  \right)
\\&\cdot
  \bigg[
  \tfrac{
  \chi^{
    \beta, T
  }_{A,\eta}\,
    \left|
      s - t
    \right|^{
      ( 1/2 - \beta - \varrho )
    }
  }{
    \min\{
    s^\rho
    ,
    t^\rho
    \} \,
    \sqrt{1-2\beta}
  }
+
    \tfrac{
  \kappa^{
    \varrho, T
  }_{A,\eta}
  \,
  \chi^{
    \varrho + \beta, T
  }_{A,\eta} \,
      |
      \mathbb{B}( 1 - 2 \beta - 2 \varrho, 1 - 2\rho )
      |^{1/2}
    }{
      s^{ ( \rho + \varrho + \beta - 1/2 ) }
    }
  \bigg]
  \Bigg]
  .
\end{split}
\end{equation}
\end{enumerate}
\end{theorem}

\begin{proof}
Throughout this proof let 
$ \mathcal{L} $
and
$ \mathbb{L} $
be the sets given by
\begin{equation}
\begin{split}
&
  \mathcal{L}
  =
  \left\{
  \substack{
    X \in
    \mathcal{M}(
      \mathrm{Pred}( ( \mathcal{F}_t )_{ t \in [0,T] } )
      ,
      \mathcal{B}( H_{ - \max\{\delta,0\} } )
    )
    \colon
    X( (0,T] \times \Omega ) \subseteq H,
\\ 
    \| X_0 \|_{
      L^p( \P; H_{ -\max\{\delta,0\} } )
    }
    +
    \sup_{ t \in (0,T] }
    t^{ \lambda }
    \,
    \| X_t \|_{
      L^p( \P ; H )
    }
    < \infty
  }
  \right\}
\end{split}
\end{equation}
and 
$
  \mathbb{L}
=
  \big\{
    \big\{
      Y \in \mathcal{L}
      \colon
      \inf\nolimits_{ t \in [0,T] }
      \P\big(
        Y_t = X_t
      \big)
      = 1
    \big\}
    \subseteq \mathcal{L}
    \colon
    X \in \mathcal{L}
  \big\}
$ 
and let 
$
  \left| \cdot \right|_{
    \mathbb{L}
    , r
  }
  \colon
  \mathbb{L}
  \to [0,\infty)
$, $ r \in \R $, 
and
$
  \left\| \cdot \right\|_{
    \mathbb{L}
    , r
  }
  \colon
  \mathbb{L}
  \to [0,\infty)
$, $ r \in \R $, 
be the functions
which satisfy for all 
$ r \in \R $, 
$
  X \in
  \mathbb{L}
$
that
\begin{equation}
  \left|
    X
  \right|_{
    \mathbb{L}
    , r
  }
  =
  \sup_{ t \in (0,T] }
  \left[
    e^{ r t }
    \,
    t^{ \lambda }
    \,
    \|
      X_t
    \|_{ L^p( \P ; H ) }
  \right]
  \qquad
  \text{and}
  \qquad
  \left\|
    X
  \right\|_{
    \mathbb{L}
    , r
  }
  =
  \left\| X_0 \right\|_{ L^p( \P ; H_{ -\max\{\delta,0\} } ) }
  +
  \left|
    X
  \right|_{
    \mathbb{L}
    , r
  }.
\end{equation}
Here and below we do not distinguish between an element $ X \in \mathcal{L} $ 
and its equivalence class 
$
    \big\{
      Y \in \mathcal{L}
      \colon
      \inf\nolimits_{ t \in [0,T] }
      \P\big(
        Y_t = X_t
      \big)
      = 1
    \big\}
    \in \mathbb{L}
$.
We observe that for all $ t \in (0,T] $ it holds that
\begin{equation}
\label{eq:initial}
  t^\lambda
  \|
    e^{ t A } \xi
  \|_{
    L^p( \P ; H )
  }
\leq
  T^{ (\lambda-\delta) }
  \sup_{ s \in (0,T] }
  s^\delta
  \|
    e^{ sA } \xi
  \|_{
    L^p( \P ; H )
  }
< \infty
  .
\end{equation}
This ensures that
\begin{equation}
  \big(
  [0,T] \times \Omega \ni
  ( t , \omega ) \mapsto
  e^{ t A } \xi( \omega )
  \in H_{ - \max\{\delta,0\} }
  \big)
  \in
  \mathcal{L} 
  .
\end{equation}
Combining this with
Lemma~\ref{lem:integrals} and Lemma~\ref{lem:integrals2}
shows that there exists a unique mapping
$
  \Phi
  \colon
  \mathbb{L}
  \to
  \mathbb{L}
$
which satisfies that for all
$
  Y \in
  \mathbb{L}
$,
$ t \in [0,T] $
it holds $ \P $-a.s.\ that
\begin{equation}
\label{eq:fixed_point_mapping}
  \Phi(Y)_t
  =
  e^{tA} \xi
  +
  \int_0^t
    e^{ (t-s) A }
    \mathbf{F}(s,Y_s)
    \,
  \diffns s
  +
  \int_0^t
    e^{ (t-s) A }
    \mathbf{B}(s,Y_s)
    \,
  \diffns W_s
  .
\end{equation}
Our next aim is to prove that there exists a real number
$ r \in \R $ such that
$ \Phi $
is a contraction on the normed $ \R $-vector space
$
  (
    \mathbb{L} ,
    \left\| \cdot \right\|_{ \mathbb{L}, r }
  )
$.
Banach's fixed point theorem together with Lemma~\ref{lem:singular.space.completeness} 
will then allow us to prove \eqref{item:thm_existence:existence}.
Observe that the Burkholder-Davis-Gundy type inequality in
Lemma~7.7 in Da Prato \& Zabczyk~\cite{dz92}
proves that for all
$
  Y, Z \in \mathbb{L}
$, 
$ r \in \R $, 
$ t \in [0,T] $
it holds that
\begin{equation}
\begin{split}
&
  \left\|
    \Phi( Y )_t -
    \Phi( Z )_t
  \right\|_{
    L^p( \P; H )
  }
\leq
  \left\|
    \int_0^t
      e^{ ( t - s ) A }
      \left(
        \mathbf{F}( s, Y_s ) - \mathbf{F}( s, Z_s )
      \right)
    \diffns s
  \right\|_{
    L^p( \P; H )
  }
\\ & \quad
  +
  \left\|
    \int_0^t
      e^{ ( t - s ) A }
      \left(
        \mathbf{B}( s, Y_s ) -
        \mathbf{B}( s, Z_s )
      \right)
    \diffns W_s
  \right\|_{
    L^p( \P; H )
  }
\\& \leq
  \int_0^t
    \left\|
      e^{ ( t - s ) A }
    \right\|_{
      L( H_{ - \alpha }, H )
    }
    \left\|
      \mathbf{F}( Y_s ) -
      \mathbf{F}( Z_s )
    \right\|_{
      L^p( \P; H_{ - \alpha } )
    }
  \diffns s
\\ & \quad
  +
  \left[
    \tfrac{ p \, ( p - 1 ) }{ 2 }
    \int_0^t
      \left\|
        e^{ ( t - s ) A }
      \right\|_{
        L( H_{ - \beta }, H )
      }^2
      \left\|
        \mathbf{B}( Y_s ) -
        \mathbf{B}( Z_s )
      \right\|_{
        L^p( \P ; HS( U, H_{ - \beta } ) )
      }^2
    \diffns s
  \right]^{
    \nicefrac{ 1 }{ 2 }
  }
\\ &
  \leq
  \chi_{ A, \eta }^{ \alpha, T }
  \int_0^t
    L_0 
    \left( t - s \right)^{ - \alpha }
    \left\|
      Y_s - Z_s
    \right\|_{
      L^p( \P; H )
    }
  \diffns s
\\ & \quad
  +
  \chi_{ A, \eta }^{ \beta, T }\,
  \left[
    \tfrac{ p \, ( p - 1 ) }{ 2 }
    \int_0^t
      |L_1|^2
      \left(
        t - s
      \right)^{
        - 2 \beta
      }
      \left\|
        Y_s - Z_s
      \right\|_{
        L^p( \P ; H )
      }^2
    \diffns s
  \right]^{
    \frac{ 1 }{ 2 }
  }
\\& \leq
  \chi_{ A , \eta }^{ \alpha, T }
  \left|
    Y - Z
  \right|_{
    \mathbb{L}, r
  }
  \int_0^t
    L_0 
    \left(
      t - s
    \right)^{ - \alpha }
    s^{ - \lambda }
    \,
    e^{ - r s }
    \,
  \diffns s
\\ & \quad
  +
  \chi_{ A , T }^{ \beta , T }\,
    \left|
      Y - Z
    \right|_{
      \mathbb{L} , r
    }
  \left[
    \tfrac{ p \, ( p - 1 ) }{ 2 }
    \int_0^t
      |L_1|^2 
      \left( t - s \right)^{ - 2 \beta }
      s^{ - 2 \lambda }
      \,
      e^{ - 2 r s }
      \,
    \diffns s
  \right]^{
    \frac12
  }
\\ & \leq
\textstyle
  \left[
    \chi_{ A , \eta }^{ \alpha, T }
    \int_0^t
      L_0 \,
      e^{ -r s }
      \left(
        t - s
      \right)^{ - \alpha }
      s^{ - \lambda }
      \,
    \diffns s
    +
    \chi_{ A , T }^{ \beta , T }\,
    \left[
      \tfrac{ p \, ( p - 1 ) }{ 2 }
      \int_0^t
        |L_1|^2 \,
        e^{ - 2 r s }
        \left( t - s \right)^{ - 2 \beta }
        s^{ - 2 \lambda }
        \,
      \diffns s
    \right]^{
      \frac12
    }
  \right]
\\&\quad\cdot
  \left|
    Y - Z
  \right|_{
    \mathbb{L} , r
  }
  < \infty
  .
\end{split}
\end{equation}
Hence, we obtain that
for all $ Y, Z \in \mathbb{L} $, $ r \in (-\infty,0] $
it holds that
\begin{equation}
\label{eq:contraction}
\begin{split}
&
  \left\|
    \Phi( Y ) - \Phi( Z )
  \right\|_{
    \mathbb{L} , r
  }
\\&=
  \left\|
    \Phi( Y )_0
    -
    \Phi( Z )_0
  \right\|_{
    L^p( \P ; H_{ -\max\{\delta,0\} } )
  }
  +
  \sup_{ t \in (0,T] }
  \left[
    e^{ r t }
    \,
    t^{ \lambda }
    \left\|
      \Phi( Y )_t
      -
      \Phi( Z )_t
    \right\|_{
      L^p(\P;H)
    }
  \right]
\\ & \leq
\textstyle
  \sup\limits_{ t \in (0,T] }
  \left[
    \chi_{ A , \eta }^{ \alpha, T }
    \int_0^t
    \frac{
      L_0 \,
      e^{ r ( t - s ) }
      \,
      t^{ \lambda }
    }{
      \left(
        t - s
      \right)^{ \alpha }
      \,
      s^{ \lambda }
    }
      \,
    \diffns s
    +
    \chi_{ A , T }^{ \beta , T }\,
    \left[
      \tfrac{ p \, ( p - 1 ) }{ 2 }
      \int_0^t
      \frac{
        |L_1|^2 \,
        e^{ 2 r ( t - s ) }
        \,
        t^{ 2 \lambda }
      }{
        \left( t - s \right)^{ 2 \beta }
        \,
        s^{ 2 \lambda }
      }
        \,
      \diffns s
    \right]^{
      \frac12
    }
  \right]
  \left|
    Y - Z
  \right|_{
    \mathbb{L} , r
  }
.
\end{split}
\end{equation}
This and the integral transformation theorem with the diffeomorphisms 
$
    (0,1) \ni s \mapsto t (1-s) \in (0,t)
$ 
for 
$ t \in (0,T] $ 
show that for all 
$ Y, Z \in \mathbb{L} $, $ r \in (-\infty,0] $
it holds that
\begin{equation}
\begin{split}
&
  \left\|
    \Phi( Y ) - \Phi( Z )
  \right\|_{
    \mathbb{L} , r
  }
  \leq
  \left|
    Y - Z
  \right|_{
    \mathbb{L} , r
  }
\\ & \cdot
\textstyle
  \left(
    \chi_{ A , \eta }^{ \alpha, T }
    \sup_{ t \in (0,T] }
    \left[
    \int_0^1
    \frac{
      L_0 \,
      e^{ r t s }
      \,
      t^{ (1 - \alpha) }
    }{
      s^{ \alpha }
      \,
      \left( 1 - s \right)^{ \lambda }
    }
      \,
    \diffns s
    \right]
    +
    \chi_{ A , T }^{ \beta , T }\!
    \left[
      \tfrac{ p \, ( p - 1 ) }{ 2 }
      \sup_{ t \in (0,T] }
      \left[
      \int_0^1
      \frac{
       |L_1|^2 \,
        e^{ 2 r t s }
        \,
        t^{ (1 - 2\beta) }
      }{
        s^{ 2 \beta }
        \,
        \left( 1 - s \right)^{ 2 \lambda }
      }
        \,
      \diffns s
      \right]
    \right]^{
      \frac12
    }
  \right)
  .
\end{split}
\end{equation}
Next note that Lebesgue's theorem of dominated convergence 
ensures that 
for all $ r \in \R $
it holds that the functions
\begin{equation}
  [0,T] \ni t \mapsto 
    \int_0^1
    \frac{
      L_0 \,
      e^{ r t s }
      \,
      t^{ (1 - \alpha) }
    }{
      s^{ \alpha }
      \,
      \left( 1 - s \right)^{ \lambda }
    }
      \,
    \diffns s
    =
    L_0 
      \,
      t^{ (1 - \alpha) }
    \int_0^1
    \frac{
      e^{ r t s }
    }{
      s^{ \alpha }
      \,
      \left( 1 - s \right)^{ \lambda }
    }
      \,
    \diffns s
  \in 
  [0,\infty)
\end{equation}
and 
\begin{equation}
  [0,T] \ni t \mapsto
      \int_0^1
      \frac{
       |L_1|^2 \,
        e^{ 2 r t s }
        \,
        t^{ (1 - 2\beta) }
      }{
        s^{ 2 \beta }
        \,
        \left( 1 - s \right)^{ 2 \lambda }
      }
        \,
      \diffns s
   =
      t^{ (1 - 2\beta) }
      \int_0^1
      \frac{
        |L_1|^2 \,
        e^{ 2 r t s }
      }{
        s^{ 2 \beta }
        \,
        \left( 1 - s \right)^{ 2 \lambda }
      }
        \,
      \diffns s
  \in [0,\infty)
\end{equation}
are continuous.
This and the fact that for all $ t \in [0,T] $
it holds that
\begin{equation}
  \limsup_{ r \to - \infty } 
  \left[
    \int_0^1
    \frac{
      L_0 \,
      e^{ r t s }
      \,
      t^{ (1 - \alpha) }
    }{
      s^{ \alpha }
      \,
      \left( 1 - s \right)^{ \lambda }
    }
      \,
    \diffns s
  \right]
  =
  \limsup_{ r \to - \infty } 
    \left[
      \int_0^1
      \frac{
       |L_1|^2 \,
        e^{ 2 r t s }
        \,
        t^{ (1 - 2\beta) }
      }{
        s^{ 2 \beta }
        \,
        \left( 1 - s \right)^{ 2 \lambda }
      }
        \,
      \diffns s
    \right]
  = 0
\end{equation}
allows us to apply Dini's theorem 
(see, e.g., Theorem~7.13 in Rudin~\cite{Rudin1976}) 
to obtain that
\begin{equation}
\label{eq:uniform0}
\begin{split}
&
  \limsup_{ r \to - \infty }
  \left(
    \chi_{ A , \eta }^{ \alpha, T }
    \sup_{ t \in (0,T] }
    \left[
    \int_0^1
    \tfrac{
      L_0 \,
      e^{ r t s }
      \,
      t^{ (1 - \alpha) }
    }{
      s^{ \alpha }
      \,
      \left( 1 - s \right)^{ \lambda }
    }
      \,
    \diffns s
    \right]
    +
    \chi_{ A , T }^{ \beta , T }\,
    \left[
      \tfrac{ p \, ( p - 1 ) }{ 2 }
      \sup_{ t \in (0,T] }
      \left[
      \int_0^1
      \tfrac{
       |L_1|^2 \,
        e^{ 2 r t s }
        \,
        t^{ (1 - 2\beta) }
      }{
        s^{ 2 \beta }
        \,
        \left( 1 - s \right)^{ 2 \lambda }
      }
        \,
      \diffns s
      \right]
    \right]^{
      \frac12
    }
  \right)
\\ &   = 0
.
\end{split}
\end{equation}
The Banach fixed point theorem 
together with Lemma~\ref{lem:singular.space.completeness} and~\eqref{eq:contraction} hence establishes \eqref{item:thm_existence:existence}, that is,
there exists an up-to-modifications unique $ X \in \mathcal{L} $
which fulfills that for all 
$ t \in[0,T] $ 
it holds that 
$
  \P\big(
  \int_0^t
  \|
    e^{  (t - s) A}
    \mathbf{F}( s , X_s )
  \|_H
  +
  \|
    e^{  (t - s) A}
    \mathbf{B}( s , X_s )
  \|_{ HS( U, H ) }^2
  \,
  \diffns s
  < \infty
  \big)
  =1
$
and~\eqref{eq:SEE}.
In the next step we observe that~\eqref{item:SEE.temporal} follows directly 
from item~\eqref{item:drift.temporal} of Lemma~\ref{lem:integrals}, 
from item~\eqref{item:diff.temporal} of Lemma~\ref{lem:integrals2}, 
and from the fact that 
$
  \forall \, 
  \varrho \in [0,1], \,
  t \in (0,T], \,
  s \in (0,t)
  \colon
  \| e^{tA} \xi - e^{sA} \xi \|_{\lpn{p}{\P}{H}}
  \leq
  \frac{
    |t-s|^\varrho
  }{ s^{(\varrho+\max\{\delta,0\})} } \,
  \kappa^{ \varrho, T }_{A,\eta} \, 
  \chi^{ \varrho+\max\{\delta,0\}, T }_{A,\eta} \,
  \|\xi\|_{\lpn{p}{\P}{H_{-\max\{\delta,0\}}}}
$.
It thus remains to prove \eqref{item:thm_existence:a_priori}.
For this we apply Proposition~\ref{prop:perturbation_estimate} 
(with $ Y^1 = X $, $ Y^2 = 0 $, and $ r = \lambda $ in the notation of Proposition~\ref{prop:perturbation_estimate}) 
to obtain that 
\begin{equation}
\label{eq:a.priori.bound}
\begin{split}
&
  \sup_{
    t \in (0,T]
  }
  \left[
  t^\lambda
  \,
  \|
    X_t
  \|_{
    L^p ( \Omega ; H )
  }
  \right]
\leq
  \Theta_{ A , \eta , p , T }^{ \alpha , \beta , \lambda }( L_0, L_1 )
\\ &
  \cdot
  \sup_{t\in(0,T]}
  \bigg[
  t^\lambda\,
  \Big\|
    X_t
    -
    \int_0^t
      e^{(t-s)A}
      \mathbf{F}(s,X_s)
    \,\diffns s
    -
    \int_0^t
      e^{(t-s)A}
      \mathbf{B}(s,X_s)
    \,\diffns W_s\\
& \qquad\qquad\quad  +
    \int_0^t
      e^{(t-s)A}
      \mathbf{F}(s,0)
    \,\diffns s
    +
    \int_0^t
      e^{(t-s)A}
      \mathbf{B}(s,0)
    \,\diffns W_s
  \Big\|_{L^p(\P;H)}
  \bigg]\\
& =
  \Theta_{ A , \eta , p , T }^{ \alpha , \beta , \lambda }( L_0, L_1 )
\\&\quad\cdot
  \sup_{t\in(0,T]}
  \bigg[  
  t^\lambda
  \,
  \Big\|
    e^{tA}\xi
    +
    \int_0^t
      e^{(t-s)A}
      \mathbf{F}(s,0)
    \,\diffns s
    +
    \int_0^t
      e^{(t-s)A}
      \mathbf{B}(s,0)
    \,\diffns W_s
  \Big\|_{L^p(\P;H)}
  \bigg]
  .
\end{split}
\end{equation}
Next we note that the Burkholder-Davis-Gundy type inequality 
in Lemma~7.7 in Da Prato \& Zabczyk~\cite{dz92} 
implies that for all $t\in[0,T]$ it holds that
\begin{equation}
\begin{split}
& \Big\|
    e^{tA}\xi
    +
    \int_0^t
      e^{(t-s)A}
      \mathbf{F}(s,0)
    \,\diffns s
    +
    \int_0^t
      e^{(t-s)A}
      \mathbf{B}(s,0)
    \,\diffns W_s
  \Big\|_{L^p(\P;H)}\\
& \leq
  \big\|
    e^{tA}\xi
  \big\|_{L^p(\P;H)}
  +
  \int_0^t
    \big\|
      e^{(t-s)A}
      \mathbf{F}(s,0)
    \big\|_{L^p(\P;H)}
    \,\diffns s
    +
    \bigg[
      \tfrac{p(p-1)}2
      \int_0^t
        \big\|
          e^{(t-s)A}
          \mathbf{B}(s,0)
        \big\|_{L^p(\P;H)}^2
      \,\diffns s
    \bigg]^{ 1/2 }\\
& \leq
  \big\|
    e^{tA}\xi
  \big\|_{L^p(\P;H)}
  +
  \chi_{A,\eta}^{\alpha,T}\,
  \hat{L}_0
  \int_0^t
    (t-s)^{-\alpha}\,
    s^{-\hat{\alpha}}
    \,\diffns s
    +
    \chi_{A,\eta}^{\beta,T}\,
    \hat{L}_1
    \bigg[
      \tfrac{p\,(p-1)}2
      \int_0^t
        (t-s)^{-2\beta}\,
        s^{-2\hat{\beta}}
      \,\diffns s
    \bigg]^{ 1/2 }\\
& \leq
  \big\|
    e^{tA}\xi
  \big\|_{L^p(\P;H)}
  +
  \tfrac{
    \chi_{A,\eta}^{\alpha,T}\,
    \hat{L}_0\,
    \mathbb{ B } ( 1 - \alpha , 1 - \hat{ \alpha } )
  }
  {
    t^{(\alpha+\hat{\alpha}-1)}
  }
  +
  \tfrac
  {
    \chi_{ A , \eta }^{ \beta , T }\,
    \hat{ L }_1\,
    \sqrt{
      p \, ( p - 1 ) \,
      \mathbb{B} ( 1 - 2\beta , 1 - 2\hat{\beta} )
    }
  }
  {
    \sqrt{2}
    \,
    t^{(\beta+\hat{\beta}-1/2)}
  }.
\end{split}
\end{equation}
This shows that
\begin{equation}
\begin{split}
&
  \sup_{ t \in (0,T] }
  \left[
  t^\lambda
  \left\|
    e^{tA}\xi
    +
    \int_0^t
      e^{(t-s)A}
      \mathbf{F}(s,0)
    \,\diffns s
    +
    \int_0^t
      e^{(t-s)A}
      \mathbf{B}(s,0)
    \,\diffns W_s
  \right\|_{L^p(\P;H)}
  \right]
\\
& \leq
    T^{(\lambda-\delta)}
    \sup_{ t \in (0,T] }
    \big[
    t^\delta
    \|
      e^{tA}
      \xi
    \|_{L^p(\P;H)}
    \big]
  +
    \chi_{A,\eta}^{\alpha,T}
    \,
    \hat{L}_0
    \,
    T^{(\lambda+1-\alpha-\hat{\alpha})}
    \,
    \mathbb{ B } ( 1 - \alpha , 1 - \hat{ \alpha } )
\\&\quad+
  \tfrac{
    \chi_{ A , \eta }^{ \beta , T }
    \,
    \hat{ L }_1
    \,
    T^{(\lambda+\nicefrac12-\beta-\hat{\beta})}
    \sqrt{
      p \,
      ( p - 1 ) \,
      \mathbb{B}( 1 - 2 \beta , 1 - 2 \hat{\beta} )
    }
  }{
    \sqrt{2}
  }
  < \infty.
\end{split}
\end{equation}
Combining this with~\eqref{eq:a.priori.bound} proves~\eqref{item:thm_existence:a_priori}. 
The proof of Theorem~\ref{thm:existence_uniqueness} is thus completed.
\end{proof}

\begin{corollary} 
\label{cor:rough_initial}
Let
$
  \left(
    H,
    \left\| \cdot \right\|_H,
    \left< \cdot, \cdot \right>_H
  \right)
$ 
and 
$
  \left(
    U,
    \left\| \cdot \right\|_U ,
    \left< \cdot, \cdot \right>_U
  \right)
$
be separable $ \R $-Hilbert spaces
with $ \#_H(H) > 1 $,
let
$ T \in (0,\infty) $,
$ \eta \in \R $,
$ \alpha \in [ 0, 1 ) $, $ \beta \in [ 0, \nicefrac{1}{2} ) $, 
let
$
  ( \Omega , \mathcal{F}, \P, ( \mathcal{F}_t )_{ t \in [0,T] } )
$
be a stochastic basis,
let
$
  ( W_t )_{ t \in [0,T] }
$
be an $\operatorname{Id}_U$-cylindrical $ ( \mathcal{F}_t )_{ t \in [0,T] } $-Wiener process,
let
$
  A \colon D(A)
  \subseteq
  H \rightarrow H
$
be a generator of a strongly continuous analytic semigroup
with the property that
$
  \operatorname{spectrum}( A )
  \subseteq
  \{
    z \in \mathbb{C}
    \colon
    \operatorname{Re}( z ) < \eta
  \}
$,
let
$
  (
    H_r
    ,
    \left\| \cdot \right\|_{ H_r }
    ,
    \left< \cdot , \cdot \right>_{ H_r }
  )
$,
$ r \in \R $,
be a family of interpolation spaces associated to
$
  \eta - A
$,
and let 
$ F \in \operatorname{Lip}( H , H_{ - \alpha } ) $,
$ B \in \operatorname{Lip}( H , HS( U , H_{ - \beta } ) ) $, 
$
  \hat{ \delta }
  =
  \frac{
    1
  }{ 2 }
  \big[
    1 +
    \mathbbm{1}_{ 
      \{ 0 \} 
    }(
      | B |_{
        \operatorname{Lip}( H, HS( U, H_{ - \beta } ) )
      }
    )
  \big]
$.
Then
\begin{enumerate}[(i)]
\item
\label{item:cor_rough_initial:existence}
there exist up-to-modifications unique 
$ ( \mathcal{F}_t )_{ t \in [0,T] } $-predictable stochastic processes
$
  X^x
  \colon
  [ 0 , T ] \times \Omega
  \to
  H_{ - \delta }
$,
$
  x
  \in
  H_{ - \delta }
$,
$
  \delta \in [ 0 , \hat{ \delta } )
$,
which fulfill
for all
$ p \in [2,\infty) $,
$
  \delta
  \in [ 0 , \hat{\delta} )
$,
$
  x \in H_{ - \delta } 
$,
$ t \in [0,T] $
that
$
  X^x
  ( ( 0 , T ] \times \Omega ) \subseteq H
$,
that
$
  \sup_{ s \in (0,T] }
  s^{
    \delta
  }
  \,
  \|
    X_s^x
  \|_{ L^p( \P ; H ) }
  < \infty
$,
and $ \P $-a.s.\ that
\begin{equation}
\label{eq:SEE2_only_Lip}
\begin{split}
&
  X_t^x
  =
  e^{ t A } x
  +
  \int_0^t
    e^{ ( t - s ) A }
      F( X_s^x )
    \, \diffns s
  +
  \int_0^t
    e^{ ( t - s ) A }
      B( X_s^x )
  \, \diffns W_s
  ,
\end{split}
\end{equation}
\item
\label{item:cor_rough_initial:a_priori}
for all
$ p \in [ 2 , \infty ) $, $ \delta \in [ 0 , \hat{ \delta } ) $
it holds that
\begin{equation}
\label{eq:SPDE_smooth_only_Lip}
\begin{split}
&
  \sup_{ x \in H_{ - \delta } }
  \sup_{ t \in (0,T] }
  \left[
  \frac{
    t^{ \delta }
    \left\|
      X_t^x
    \right\|_{
      L^p( \P; H )
    }
  }{
    \max\{ 1, \| x \|_{ H_{ - \delta } } \}
  }
  \right]
  \leq
  \Theta_{A,\eta,p,T}^{ \alpha, \beta, \delta }\big( 
    | F |_{ \operatorname{Lip}( H, H_{ - \alpha } ) } , 
    | B |_{ \operatorname{Lip}( H, HS( U, H_{ - \beta } ) ) } 
  \big)
\\ & \cdot
  \left[
      \chi_{ A, \eta }^{ \delta, T }
    +
    \tfrac{
      \chi^{ \alpha, T }_{ A, \eta }
      \left\| F(0) \right\|_{
        H_{ - \alpha }
      }
      T^{
        (
          \delta + 1 - \alpha 
        )
      }
    }{
      ( 1 - \alpha ) 
    }
    +
    \tfrac{
      \sqrt{
        p \, ( p - 1 )
      }
      \,
      \chi_{ A, \eta }^{ \beta , T }
      \left\| B(0) \right\|_{
        HS( U, H_{ - \beta } )
      }
      T^{
        (
          \delta
          + \nicefrac{ 1 }{ 2 } - \beta 
        )
      }
    }{
      \sqrt{ 2 - 4 \beta } 
    }
  \right]
  < \infty,
\end{split}
\end{equation}
\item 
\label{item:cor_rough_initial:Lipschitz_estimate}
for all
$ p \in [ 2 , \infty ) $, $ \delta \in [ 0 , \hat{ \delta } ) $
it holds that
\begin{equation}
\label{eq:SPDE_smooth_only_Lip_III}
\begin{split}
&
  \sup_{
    \substack{
      x, y \in H_{ - \delta } ,
    \\
      x \neq y
    }
  }
  \sup_{ t \in (0,T] }
  \left[
    \frac{
      t^{ \delta }
      \,
      \|
        X_t^x - X_t^y
      \|_{
        L^p( \P ; H )
      }
    }{
      \|
      x - y
      \|_{
        H_{ - \delta } 
      }
    }
  \right]
\\&\leq
  \chi_{ A, \eta }^{ \delta, T }\,
  \Theta_{A,\eta,p,T}^{ \alpha, \beta, \delta }\big( 
    | F |_{ \operatorname{Lip}( H, H_{ - \alpha } ) } , 
    | B |_{ \operatorname{Lip}( H, HS( U, H_{ - \beta } ) ) } 
  \big)
  < \infty
  ,
\end{split}
\end{equation}
\item
\label{item:cor.SEE.temporal}
and for all 
$ p \in [ 2 , \infty ) $, $ \delta \in [ 0 , \hat{ \delta } ) $,   
$ \varrho \in [0,\min\{1-\alpha,\nicefrac{1}{2}-\beta\}) $   
it holds that 
\begin{equation}
\begin{split}
&
    \sup_{ t \in (0,T] } \sup_{ s \in (0,t) }
    \sup_{x \in H_{-\delta}}
    \left[
    \frac{
    s^{ ( \varrho + \delta ) }
    \left\|
      X^x_s - X^x_t
    \right\|_{
      L^p( \P; H )
    }
    }{
      \max\{1,\|x\|_{H_{-\delta}}\} \,
      | s - t |^\varrho
    }
    \right]
  <\infty.
\end{split}
\end{equation}
\end{enumerate}
\end{corollary}

\begin{proof}[Proof 
of Corollary~\ref{cor:rough_initial}]
Throughout this proof 
let
$ L_0, L_1, \hat{L}_0, \hat{L}_1 \in [0,\infty) $
be the real numbers given by 
$
  L_0 = 
  | F |_{ \operatorname{Lip}( H, H_{ - \alpha } ) }
$,
$
  L_1 =
  |B|_{
    \operatorname{Lip}( H , HS( U, H_{ - \beta } ) ) 
  }
$,
$
  \hat{L}_0 = \| F ( 0 ) \|_{ H_{ \alpha } }
$, 
and 
$
  \hat{L}_1 = \| B ( 0 ) \|_{ HS( U, H_{ \beta } ) }
$.
We note that
for all $ t \in (0,T] $, $ X, Y \in \mathcal{L}^p( \P; H ) $
it holds that
\begin{align}
  \|
    F(X) - F(Y)
  \|_{
    L^p( \P; H_{ - \alpha } )
  }
  \leq
  L_0 \,
  \| X - Y \|_{
    L^p( \P; H )
  }
  ,
\quad
&
  \|
    F( 0 )
  \|_{
    L^p( \P ; H_{ - \alpha } )
  }
  \leq
  \hat{L}_0
  ,
\\
  \|
    B(X) - B(Y)
  \|_{
    L^p( \P ; HS( U, H_{ - \beta } ) )
  }
  \leq
  L_1
  \| X - Y \|_{
    L^p( \P; H )
  } ,
\quad
&
  \| B(0) \|_{
    L^p( \P; HS( U, H_{ - \beta } ) )
  }
\leq
  \hat{L}_1
  .
\end{align}
We can hence apply Corollary~\ref{cor:initial_perturbation} and Theorem~\ref{thm:existence_uniqueness}.
More specifically,
an application of Theorem~\ref{thm:existence_uniqueness}
(with 
$ \delta = \delta $, 
$ \lambda = \delta $, 
$ \hat{\alpha} = \hat{\beta} = 0 $, 
$ L_0 = |F|_{ \operatorname{Lip}( H, H_{-\alpha} ) } $, 
$ \hat{L}_0 = \|F(0)\|_{ H_{-\alpha} } $, 
$ L_1 = |B|_{ \operatorname{Lip}( H, HS(U,H_{-\beta}) ) } $, 
and 
$ \hat{L}_1 = \|B(0)\|_{ HS(U,H_{-\beta}) } $ 
for $ \delta \in [0, \hat{\delta} ) $
in the notation of Theorem~\ref{thm:existence_uniqueness})
proves \eqref{item:cor_rough_initial:existence}, 
proves that for all 
$ p \in [2,\infty) $, $ \delta \in [0, \hat{\delta} ) $,
$ x \in H_{ -\delta } $
it holds that
\begin{equation}
\label{eq:apriori.strong}
\begin{split}
&
  \sup_{ t \in (0,T] }
  \left[
    t^\delta
    \left\|
      X_t^x
    \right\|_{
      L^p( \P; H )
    }
  \right]
  \leq
  \Theta_{A,\eta,p,T}^{ \alpha, \beta, \delta }\big( 
    | F |_{ \operatorname{Lip}( H, H_{ - \alpha } ) } , 
    | B |_{ \operatorname{Lip}( H, HS( U, H_{ - \beta } ) ) } 
  \big)
\\ & \cdot
  \bigg[
      \chi_{ A, \eta }^{ \delta, T }
      \left\|
        x
      \right\|_{
        H_{ - \delta } 
      }
    +
      \chi^{ \alpha, T }_{ A, \eta }
      \,
      \|F(0)\|_{ H_{-\alpha} }
      \,
     T^{
        \left(
          \delta + 1 - \alpha 
        \right)
      }
      \,
      \mathbbm{B}(
        1 - \alpha
        ,
        1 
      )
\\&\quad+
    \tfrac{
      \chi_{ A, \eta }^{ \beta , T }
      \,
      \|B(0)\|_{ HS(U,H_{-\beta}) }\,
      T^{
        (
          \delta
          + \nicefrac{ 1 }{ 2 } - \beta 
        )
      }
      \left|
        p \, ( p - 1 )
        \,
        \mathbbm{B}(
          1 - 2 \beta
          ,
          1 
        )
      \right|^{ \nicefrac{ 1 }{ 2 } }
    }{
    \sqrt{2}
    }
  \bigg]
  < \infty,
\end{split}
\end{equation}
and proves that for all 
$ p \in [2,\infty) $, $ \delta \in [0, \hat{\delta} ) $,
$ x \in H_{ -\delta } $, 
$\varrho\in[0,\min\{1-\alpha,\nicefrac{1}{2}-\beta\})$, 
$s,t\in(0,T]$ 
with $s<t$ 
it holds that 
\begin{equation}
\label{eq:temporal.regularity}
\begin{split}
&
    \left\|
      X^x_s - X^x_t
    \right\|_{
      L^p( \P; H )
    }
  \leq
  |s-t|^\varrho
\\&\cdot
  \Bigg[
    \tfrac{
    \kappa^{\varrho,T}_{A,\eta} \,
    \chi^{\varrho+\delta,T}_{A,\eta} \,
    \|x\|_{H_{-\delta}}   
    }{ s^{(\varrho+\delta)} }
+
    |T\vee1|^\delta
    \bigg[
    \tfrac{
    \chi^{ \alpha, T }_{A,\eta} \,
      \left| s - t \right|^{ (1 - \alpha-\varrho) }
    }{
      \left(
        1 - \alpha
      \right) \,
      s^\delta
    }
  +
    \tfrac{
    \kappa^{\varrho,T}_{A,\eta}
    \,
    \chi^{ \varrho + \alpha, T }_{A,\eta} \,
      \mathbbm{B}(
        1 -
        \alpha
        - \varrho
        ,
        1 - \delta
      )
    }{
      s^{(
        \varrho + \alpha
        +
        \delta
        - 1
      )}
    }
    \bigg]
\\&\cdot
    \left(
    \|F(0)\|_{ H_{-\alpha} }
    +
    | F |_{ \operatorname{Lip}( H, H_{ - \alpha } ) }
    \sup\nolimits_{ u \in (0,T] }
      u^\delta\,
      \|
        X^x_u
      \|_{
        \lpn{p}{\P}{H}
      }
    \right)
+
  |T\vee1|^{ \delta \1_{(0,\infty)}(| B |_{ \operatorname{Lip}( H, HS( U, H_{ - \beta } ) ) } ) }
\\&\cdot
  \sqrt{\tfrac{p\,(p-1)}{2}}
  \left(
  \|B(0)\|_{ HS(U,H_{-\beta}) }
  +
  | B |_{ \operatorname{Lip}( H, HS( U, H_{ - \beta } ) ) } 
  \sup\nolimits_{ u \in (0,T] }
    u^\delta\,
    \|
      X^x_u
    \|_{
      \lpn{p}{\P}{H}
    }
  \right)
\\&\cdot
  \bigg[
  \tfrac{
  \chi^{
    \beta, T
  }_{A,\eta}\,
    \left|
      s - t
    \right|^{
      ( 1/2 - \beta - \varrho )
    }
  }{
    s^{ \delta \1_{(0,\infty)}(| B |_{ \operatorname{Lip}( H, HS( U, H_{ - \beta } ) ) }) }
    \,
    \sqrt{1-2\beta}
  }
+
    \tfrac{
  \kappa^{
    \varrho, T
  }_{A,\eta}
  \,
  \chi^{
    \varrho + \beta, T
  }_{A,\eta} \,
      |
      \mathbb{B}( 1 - 2 \beta - 2 \varrho, 1 - 2 \delta \1_{(0,\infty)}(| B |_{ \operatorname{Lip}( H, HS( U, H_{ - \beta } ) ) }) )
      |^{1/2}
    }{
      s^{ ( \delta \1_{(0,\infty)}(| B |_{ \operatorname{Lip}( H, HS( U, H_{ - \beta } ) ) }) + \varrho + \beta - 1/2 ) }
    }
  \bigg]
  \Bigg]
  .
\end{split}
\end{equation}
Observe that~\eqref{eq:apriori.strong} establishes~\eqref{item:cor_rough_initial:a_priori} 
and note that~\eqref{item:cor_rough_initial:a_priori} and~\eqref{eq:temporal.regularity} establish~\eqref{item:cor.SEE.temporal}.
In addition, an application of Corollary~\ref{cor:initial_perturbation} 
(with 
$ X^1 = X^x $, 
$ X^2 = X^y $, 
$ \delta = \delta $, 
and $ \lambda = \delta $ 
for 
$ x, y \in H_{ -\delta } $, 
$ \delta \in [0, \hat{\delta} ) $
in the notation of Corollary~\ref{cor:initial_perturbation})
ensures that
for all $ p \in [2,\infty) $, $ \delta \in [ 0, \hat{\delta} ) $, $ x, y \in H_{ -\delta } $
it holds that
\begin{equation}
\begin{split}
&
  \sup_{ t \in (0,T] }
  \left[
    t^\delta
    \,
    \|
      X_t^x - X_t^y
    \|_{
      L^p( \P ; H )
    }
  \right]
\\&\leq
  \chi_{ A, \eta }^{ \delta, T }
  \,
  \|
    x - y
  \|_{
    H_{ - \delta } 
  }
  \Theta_{A,\eta,p,T}^{ \alpha, \beta, \delta }\big( 
    | F |_{ \operatorname{Lip}( H, H_{ - \alpha } ) } , 
    | B |_{ \operatorname{Lip}( H, HS( U, H_{ - \beta } ) ) } 
  \big)
< \infty
  .
\end{split}
\end{equation}
This establishes \eqref{item:cor_rough_initial:Lipschitz_estimate}.
The proof of Corollary~\ref{cor:rough_initial}
is thus completed.
\end{proof}

\section{Examples and counterexamples for SEEs with irregular initial values}

Corollary~\ref{cor:rough_initial} in Subsection~\ref{sec:existence_uniqueness} above establishes existence, uniqueness, and regularity properties for solutions of parabolic SEEs.
In this section we first illustrate the statement of Corollary~\ref{cor:rough_initial} in the case of semilinear stochastic heat equations with space-time white noise and periodic boundary conditions; see Corollary~\ref{cor:periodic_boundary} in Subsection~\ref{sec:counterexample} below. 
Roughly speaking, Corollary~\ref{cor:periodic_boundary} shows existence and uniqueness of solutions of the considered stochastic heat equation provided that the initial value lies in 
$
  \cup_{ \delta \in (-\nicefrac{1}{2},\infty) } H_\delta
$ 
where 
$ H_r $, $ r \in \R $, 
is a family of interpolation spaces associated to the Laplacian with 
periodic boundary conditions. 
Corollary~\ref{cor:periodic_boundary} applies, in particular, to the continuous version of the parabolic Anderson model. 
Thereafter, we illustrate in Proposition~\ref{prop:counterexampleI} in Subsection~\ref{sec:counterexample}, in Proposition~\ref{prop:counterexampleII} in Subsection~\ref{sec:counterexample.II}, and in Proposition~\ref{prop:counterexample2} in Subsection~\ref{sec:counterexample.III} 
by means of several example SEEs that the statement of Corollary~\ref{cor:rough_initial} can in general not be improved.
Moreover, we illustrate in Proposition~\ref{prop:positive.example} in Subsection~\ref{sec:counterexample} in the case of a specific linear example SEE with regular noise that the statement of Corollary~\ref{cor:rough_initial} can be improved.

\subsection{Setting}
\label{sec:setting.examples}
Throughout this section the following setting is sometimes used.
Let 
$
  \left(
    H,
    \left\| \cdot \right\|_H,
    \left< \cdot, \cdot \right>_H
  \right)
$ 
be a separable $\R$-Hilbert space, 
let 
$ T \in (0,\infty) $, 
$ \eta, \delta \in \R $, 
let
$
  ( \Omega , \mathcal{F}, \P, ( \mathcal{F}_t )_{ t \in [0,T] } )
$
be a stochastic basis, 
let $ A \colon D(A) \subseteq H \rightarrow H $ 
be a diagonal linear operator with 
$
  \operatorname{spectrum}( A )
  \subseteq
  (-\infty,\eta)
$, 
and let
$
  (
    H_r
    ,
    \left\| \cdot \right\|_{ H_r }
    ,
    \left< \cdot , \cdot \right>_{ H_r }
  )
$,
$ r \in \R $,
be a family of interpolation spaces associated to
$
  \eta - A
$.

\subsection{Stochastic heat equations with linear multiplicative noise}
\label{sec:counterexample}

\begin{corollary}
\label{cor:periodic_boundary}
Let 
$
  \left(
    H,
    \left\| \cdot \right\|_H,
    \left< \cdot, \cdot \right>_H
  \right)
  =
  \big(
    L^2( \mu_{(0, 1)}; \R ) ,
    \left\| \cdot \right\|_{
      L^2( \mu_{(0, 1)}; \R )
    },
    \left< \cdot, \cdot \right>_{
      L^2( \mu_{(0, 1)}; \R )
    }
  \big)
$, 
let 
$ T, \eta \in (0,\infty) $, 
$
  \beta \in ( \nicefrac{1}{4}, \nicefrac{1}{2} )
$, 
$ f, b \in \operatorname{Lip}( \R, \R ) $, 
let
$
  ( \Omega , \mathcal{F}, \P, ( \mathcal{F}_t )_{ t \in [0,T] } )
$
be a stochastic basis,
let
$
  ( W_t )_{ t \in [0,T] }
$
be an $\operatorname{Id}_H$-cylindrical $ ( \mathcal{F}_t )_{ t \in [0,T] } $-Wiener process,
let  
$
  A \colon D(A) \subseteq H \to H 
$
be the Laplacian with periodic boundary conditions 
on $ H $,
let
$
  (
    H_r
    ,
    \left\| \cdot \right\|_{ H_r }
    ,
    \left< \cdot , \cdot \right>_{ H_r }
  )
$,
$ r \in \R $,
be a family of interpolation spaces associated to
$
  \eta - A
$,
and let 
$ F \colon H \rightarrow H $ 
and  
$ B \colon H \rightarrow HS( H, H_{-\beta} ) $ 
satisfy for all 
$ v \in \mathcal{L}^2( \mu_{(0,1)} ; \R ) $, 
$ u \in \mathcal{C}\big( [0,1] , \R \big) $ 
that 
$
  F( [v]_{ \mu_{(0,1)}, \mathcal{B}(\R) } ) 
  = 
  [
  \{ f( v(x) ) \}_{ x \in (0,1) }
  ]_{ \mu_{(0,1)}, \mathcal{B}(\R) }
$ 
and 
$
  B\big( [v]_{ \mu_{(0,1)}, \mathcal{B}(\R) } \big) [ u|_{ (0,1) } ]_{ \mu_{(0,1)}, \mathcal{B}(\R) } 
  = 
  \big[
  \big\{ b( v(x) ) \cdot u(x) \big\}_{ x \in (0,1) }
  \big]_{ \mu_{(0,1)}, \mathcal{B}(\R) }
$.
Then 
\begin{enumerate}[(i)]
\item
there exist up-to-modifications 
unique 
$ ( \mathcal{F}_t )_{ t \in [0,T] } $-predictable
stochastic processes
$
  X^x \colon [0,T] \times \Omega \to H_{ - \delta }
$,
$ x \in H_{ - \delta } $, 
$ \delta \in [0, \nicefrac{ 1 }{ 2 } ) $,
which fulfill for all 
$ p \in [2,\infty) $,
$
  \delta
  \in [ 0 , \nicefrac{1}{2} )
$,
$
  x \in H_{ - \delta } 
$, 
$ t \in [0,T] $
that
$
  X^x
  ( ( 0 , T ] \times \Omega ) \subseteq H
$,
that
$
  \sup_{ s \in (0,T] }
  s^\delta
  \,
  \|
    X_s^x
  \|_{ L^p( \P ; H ) }
  < \infty
$, 
and $ \P $-a.s.\ that 
\begin{equation}
\begin{split}
  X^x_t
  =
  e^{ tA } x
  +
  \int^t_0 
  e^{ (t-s)A }
  F(X^x_s)
  \,\diffns{s}
  +
  \int^t_0 
  e^{ (t-s)A }
  B(X^x_s)
  \,\diffns{W_s}
\end{split}
\end{equation}
\item
and for all 
$ p \in [2,\infty) $,
$
  \delta
  \in [ 0 , \nicefrac{1}{2} )
$ 
it holds that 
\begin{equation}
  \sup_{
    \substack{
      x, y \in H_{ - \delta } ,
    \\
      x \neq y
    }
  }
  \sup_{ t \in (0,T] }
  \left[
  \frac{
    t^{ \delta }
    \left\|
      X_t^x
    \right\|_{
      L^p( \P; H )
    }
  }{
    \max\{ 1, \| x \|_{ H_{ - \delta } } \}
  }
  +
    \frac{
      t^{ \delta }
      \,
      \|
        X_t^x - X_t^y
      \|_{
        L^p( \P ; H )
      }
    }{
      \|
      x - y
      \|_{
        H_{ - \delta } 
      }
    }
  \right]
  < \infty
  .
\end{equation}
\end{enumerate}
\end{corollary}

\begin{prop}
\label{prop:counterexampleI}
Let 
$
  \left(
    H,
    \left\| \cdot \right\|_H,
    \left< \cdot, \cdot \right>_H
  \right)
  =
  \big(
    L^2( \mu_{(0, 1)}; \R ) ,
    \left\| \cdot \right\|_{
      L^2( \mu_{(0, 1)}; \R )
    },
    \left< \cdot, \cdot \right>_{
      L^2( \mu_{(0, 1)}; \R )
    }
  \big)
$, 
let 
$ T, \eta, $
$
\nu \in (0,\infty) $, 
$ r \in [ 0, \infty ) $,
$
  \delta \in \R
$, 
$
  \beta \in ( \frac{ 1 }{ 4 } , \frac{ 1 }{ 2 } ) 
$,
let
$
  ( \Omega , \mathcal{F}, \P, ( \mathcal{F}_t )_{ t \in [0,T] } )
$
be a stochastic basis,
let
$
  ( W_t )_{ t \in [0,T] }
$
be an $\operatorname{Id}_H$-cylindrical $ ( \mathcal{F}_t )_{ t \in [0,T] } $-Wiener process,
let
$ (b_n)_{ n \in \mathbb{Z} } \subseteq H $
satisfy 
for all $ n \in \N $
that
$
  b_0
  = 
  [
    \{1\}_{ x \in (0,1) }
  ]_{ \mu_{ (0,1) }, \mathcal{B}(\R) }
$,
$
  b_n
  =
  [
    \{
  \sqrt{2}
  \cos( 2 n \pi x )
    \}_{ x \in (0,1) }
  ]_{ \mu_{ (0,1) }, \mathcal{B}(\R) }
$,
and
$
  b_{ - n }
  =
  [
    \{
  \sqrt{2}
  \sin( 2 n \pi x )
    \}_{ x \in (0,1) }
  ]_{ \mu_{ (0,1) }, \mathcal{B}(\R) }
$,
let $ A \colon D(A) \subseteq H \rightarrow H $ 
be the linear operator 
which satisfies  
$
  D( A ) = \big\{
    v \in H \colon
    \sum_{ n \in \mathbb{Z} }
    |
      n^2 \langle b_n , v \rangle_H
    |^2
    < \infty
  \big\}
$
and which satisfies 
for all $ v \in D(A) $
that
$
  A v
  =
  -
  \nu
  \sum_{ n \in \mathbb{Z} }
  n^2
  \langle b_n, v \rangle_H
  b_n
$,
let
$
  (
    H_q
    ,
    \left\| \cdot \right\|_{ H_q }
    ,
    \left< \cdot , \cdot \right>_{ H_q }
  )
$,
$ q \in \R $,
be a family of interpolation spaces associated to
$
  \eta - A
$,
let 
$
  \xi \in \mathcal{M}\big( \mathcal{F}_0, \mathcal{B}(H_{ \delta }) \big)
$,
$
  B \in L( H , HS( H, H_{ - \beta } ) )
$
satisfy for all 
$ v \in \mathcal{L}^2( \mu_{(0,1)} ; \R ) $, 
$ u \in \mathcal{C}\big( [0,1], \R \big) $ 
that
$
  B\big( [v]_{ \mu_{(0,1)}, \mathcal{B}(\R) } \big) [ u|_{ (0,1) } ]_{ \mu_{(0,1)}, \mathcal{B}(\R) }
  = 
  [
  \{ v(x) \cdot u(x) \}_{ x \in (0,1) }
  ]_{ \mu_{(0,1)}, \mathcal{B}(\R) }
$, 
and let 
$
  X \colon [0,T] \times \Omega \to H_{ \delta }
$
be an $ ( \mathcal{F}_t )_{ t \in [0,T] } $-predictable stochastic 
process 
which satisfies for all $ t \in (0,T] $
that
$
  X( (0,T] \times \Omega ) \subseteq H
$,
that
\begin{equation}
\label{eq:counterexampleI.integrability}
  \E\big[ 
    \| e^{ t A } \xi \|^2_{ H_{-r} } 
  \big] 
  +
  \int_0^t 
  \E\big[ 
    \| e^{ ( t - s ) A } B( X_s ) \|^2_{ HS( H, H_{-r} ) } 
  \big] 
  \, ds
  < \infty,
\end{equation}
and $ \P $-a.s.\ that
$
  X_t = e^{ t A } \xi + \int_0^t e^{ ( t - s ) A } B( X_s ) \, dW_s
$.
Then 
\begin{enumerate}[(i)]
\item 
it holds that
$
  \P\!\left(
  \xi \in H_{ - 1 / 2 }
  \right)
  =1
$ 
and 
\item
it holds for all $ t \in (0,T] $ that 
\begin{equation}
    2^{ -1/2 }\,
    \eta^{-r}\,
    (
      1 -
      e^{
        - 2 \eta t
      }
    )^{ 1 / 2 }
  \,
  \|
    \xi
  \|_{ \lpn{2}{\P}{H_{ - 1 / 2 }} }
\leq 
      \| X_t \|_{ \lpn{2}{\P}{H_{-r}} }
< 
  \infty.
\end{equation}
\end{enumerate}
\end{prop}

\begin{proof}
Throughout this proof let 
$ \kappa_k \in [0,\infty] $, 
$ k \in \mathbb{Z} $, 
be the extended real numbers which satisfy for all 
$ k \in \mathbb{Z} $ 
that 
\begin{equation}
  \kappa_k
  =
  \sup_{ x \in (\cap_{ s \in \R } H_s) \setminus \{0\} }
  \frac{
    \| B(x) b_k \|^2_{ H_{-r-1} }
  }{ \| x \|^2_{H_{-r}} }
  .
\end{equation}
Observe that the product rule for differentiation 
and the fact that the mapping
$
  \mathcal{C}\big( [0,1], \R \big)
  \ni v 
  \mapsto 
  [
    v|_{ (0,1) }
  ]_{
    \mu_{ (0,1) }, \mathcal{B}( \R ) 
  }
  \in 
  H_{ 1/2 }
$ 
is continuous
ensures that for all $ n \in \N $ 
it holds that 
$
  \forall \, 
  u, v \in \cap_{s\in\R} H_s
  \colon
  u \cdot v \in \cap_{s\in\R} H_s
$ 
and 
$
  \sup_{ u, v \in (\cap_{s\in\R} H_s) \setminus \{0\} }
  \frac{
    \| u \cdot v \|_{ H_n }
  }{
    \| u \|_{ H_n } \, \| v \|_{ H_n }
  }
  < \infty
$. 
This implies that for all 
$ k \in \mathbb{Z} $, 
$ n \in \N_0 $ 
it holds that 
\begin{equation}
\begin{split}
&
  \sup_{ x \in ( \cap_{s\in\R} H_s ) \setminus \{0\} }
  \frac{
    \| B(x) b_k \|_{ H_{-n} }
  }{
    \|x\|_{ H_{-n} }
  }
  =
  \sup_{ x, u \in ( \cap_{s\in\R} H_s ) \setminus \{0\} }
  \frac{
    \left|\left< u, B(x) b_k \right>_H\right|
  }{
    \|x\|_{ H_{-n} } \, \|u\|_{ H_n }
  }
\\&=
  \sup_{ x, u \in ( \cap_{s\in\R} H_s ) \setminus \{0\} }
  \frac{
    \left|\left< u \cdot b_k, x \right>_H\right|
  }{
    \|x\|_{ H_{-n} } \, \|u\|_{ H_n }
  }
=
  \sup_{ x, u \in ( \cap_{s\in\R} H_s ) \setminus \{0\} }
  \frac{
    \left|\left< (\eta-A)^n (u \cdot b_k), (\eta-A)^{-n} x \right>_H\right|
  }{
    \|x\|_{ H_{-n} } \, \|u\|_{ H_n }
  }
\\&\leq
  \sup_{ u \in ( \cap_{s\in\R} H_s ) \setminus \{0\} }
  \frac{
    \| u \cdot b_k \|_{ H_n }
  }{
    \|u\|_{ H_n }
  }
  < \infty
  .
\end{split}
\end{equation}
Hence, we obtain that for all $ k \in \mathbb{Z} $ 
it holds that 
\begin{equation}
\begin{split}
&
  \kappa_k
  =
  \sup_{ x \in (\cap_{ s \in \R } H_s) \setminus \{0\} }
  \frac{
    \| B(x) b_k \|^2_{ H_{-r-1} }
  }{ \| x \|^2_{H_{-r}} }
\\&\leq
  \|
    ( \eta - A )^{ -r-1 - \lceil -r -1 \rceil_1 }
  \|^2_{ L(H) }
  \left[
    \sup_{ x \in (\cap_{ s \in \R } H_s) \setminus \{0\} }
    \frac{
      \| B(x) b_k \|_{ H_{\lceil -r -1 \rceil_1} }
    }{ \| x \|_{H_{-r}} }
  \right]^2
\\&\leq
  \|
    ( \eta - A )^{ -r-1 - \lceil -r -1 \rceil_1 }
  \|^2_{ L(H) }
  \left[
    \sup_{ x \in (\cap_{ s \in \R } H_s) \setminus \{0\} }
    \frac{
      \| x \|_{ H_{\lceil -r -1 \rceil_1} }
    }{ \| x \|_{H_{-r}} }
  \right]^2
\\&\quad\cdot
  \left[
    \sup_{ x \in (\cap_{ s \in \R } H_s) \setminus \{0\} }
    \frac{
      \| B(x) b_k \|_{ H_{\lceil -r -1 \rceil_1} }
    }{ \| x \|_{H_{\lceil -r -1 \rceil_1}} }
  \right]^2
\\&=
  \|
    ( \eta - A )^{ -r-1 - \lceil -r -1 \rceil_1 }
  \|^2_{ L(H) }
  \,
  \|
    ( \eta - A )^{ r + \lceil -r -1 \rceil_1 }
  \|^2_{ L(H) }
\\&\quad\cdot
  \left[
    \sup_{ x \in (\cap_{ s \in \R } H_s) \setminus \{0\} }
    \frac{
      \| B(x) b_k \|_{ H_{\lceil -r -1 \rceil_1} }
    }{ \| x \|_{H_{\lceil -r -1 \rceil_1}} }
  \right]^2
  < \infty
  .
\end{split}
\end{equation}
In the next step we observe that
for all 
$ t \in (0,T] $
it holds $ \P $-a.s.\ that
\begin{equation}
\label{eq:squared.solution}
\begin{split}
&
  \left\| X_t \right\|_{ H_{-r} }^2
=
  \left\| 
    e^{ t A } \xi 
    +
    \int_0^t
    e^{ ( t - s ) A } B( X_s ) \, dW_s
  \right\|_{ H_{-r} }^2
\\ & =
  \left\| e^{ t A } \xi \right\|^2_{ H_{-r} }
  +
  2
  \left<
    e^{ t A } \xi
    ,
    \int_0^t
    e^{ ( t - s ) A } B( X_s ) \,
    dW_s
  \right>_{ \! H_{-r} }
  +
  \left\|
    \int_0^t
    e^{ ( t - s ) A } B( X_s ) \, dW_s
  \right\|^2_{ H_{-r} }
  .
\end{split}
\end{equation}
Combining~\eqref{eq:squared.solution}
with It\^{o}'s isometry
and the assumption that 
$ 
  \forall \, t \in (0,T] \colon
  \E\big[ 
    \| e^{ t A } \xi \|^2_{ H_{-r} } 
  \big] 
  +
  \int_0^t 
    \E\big[ 
      \| e^{ ( t - s ) A } B( X_s ) \|^2_{ HS( H, H_{-r} ) } 
    \big] 
  \, ds
  < \infty
$
proves that
for all $ t \in (0,T] $
it holds that
\begin{equation}
\label{eq:counterexample_last_step}
\begin{split}
&
  \E\!\left[
    \left\| X_t \right\|_{ H_{-r} }^2
  \right]
=
  \E\!\left[ \| e^{ t A } \xi \|^2_{H_{-r}} \right]
  +
  2\,
  \E\!\left[
  \left<
    e^{ t A } \xi
    ,
      \int_0^t
      e^{ ( t - s ) A } B( X_s ) \,
      dW_s
  \right>_{ \! H_{-r} }
  \right]
\\&\quad+
  \E\!\left[
  \left\|
    \int_0^t
    e^{ ( t - s ) A } B( X_s ) \, dW_s
  \right\|^2_{ H_{-r} }
  \right]
\\ & =
  \E\!\left[ \| e^{ t A } \xi \|^2_{H_{-r}} \right]
  +
  2\,
  \E\!\left[\left.
  \left<
    e^{ t A } \xi
    ,
    \E\!\left[
      \int_0^t
      e^{ ( t - s ) A } B( X_s ) \,
      dW_s
    \right|\mathcal{F}_0
    \right]
  \right>_{ \! H_{-r} }
  \right]
\\&\quad+
  \int_0^t
  \E\!\left[
    \left\|
      e^{ ( t - s ) A } B( X_s ) 
    \right\|^2_{ HS( H, H_{-r} ) }
  \right]
  ds
\\ & =
  \E\!\left[ \| e^{ t A } \xi \|^2_{H_{-r}} \right]
  +
  \int_0^t
  \E\!\left[
    \left\|
      e^{ ( t - s ) A } B( X_s ) 
    \right\|^2_{ HS( H, H_{-r} ) }
  \right]
  ds
\\ & =
  \E\!\left[ \| e^{ t A } \xi \|^2_{H_{-r}} \right]
  +
  {\smallsum\limits_{ k \in \mathbb{Z} }}
  \int_0^t
  \E\!\left[
    \left\|
      e^{ ( t - s ) A } B( X_s ) b_k 
    \right\|^2_{ H_{-r} }
  \right]
  ds
  < \infty
  .
\end{split}
\end{equation}
Moreover, we note that for all 
$ k \in \mathbb{Z} $, $ t \in (0,T] $, $ s \in (0,t) $
it holds $ \P $-a.s.\ that
\begin{equation}
\label{eq:counterexample_2}
\begin{split}
&
    \left\|
      e^{ ( t - s ) A } B\!\left( X_s \right) b_k
    \right\|^2_{ H_{-r} }
= 
    \left\|
      e^{ ( t - s ) A } 
      B\!\left( 
        e^{ s A } \xi + \int_0^s e^{ ( s - u ) A } B( X_u ) \, dW_u 
      \right) \! b_k
    \right\|^2_{ H_{-r } }
\\ & =
    \left\|
      e^{ ( t - s ) A } 
      B\!\left( 
        e^{ s A } \xi 
      \right) \! b_k
    \right\|^2_{ H_{-r } }
  +
    \left\|
      e^{ ( t - s ) A } 
      B\!\left( 
        \int_0^s e^{ ( s - u ) A } B( X_u ) \, dW_u 
      \right) \! b_k
    \right\|^2_{ H_{-r } }
\\ & \quad +
    2
    \left<
      e^{ ( t - s ) A } 
      B\!\left( 
        e^{ s A } \xi 
      \right) \! b_k
      ,
      e^{ ( t - s ) A } 
      B\!\left( 
        \int_0^s e^{ ( s - u ) A } B( X_u ) \, dW_u 
      \right) \! b_k
    \right>_{ \! H_{-r} }
\\ & \geq
    \left\|
      e^{ ( t - s ) A } 
      B\!\left( 
        e^{ s A } \xi 
      \right) \! b_k
    \right\|^2_{ H_{-r} }
\\ & \quad +
    2
    \left<
      e^{ ( t - s ) A } 
      B\!\left( 
        e^{ s A } \xi 
      \right) \! b_k
      ,
      \int_0^s 
      e^{ ( t - s ) A } 
      B\!\left( 
        e^{ ( s - u ) A } B( X_u ) \, dW_u 
      \right) \! b_k
    \right>_{ \! H_{-r} }
    .
\end{split}
\end{equation}
This and assumption~\eqref{eq:counterexampleI.integrability}
imply that for all 
$ k \in \mathbb{Z} $, $ t \in (0,T] $, $ s \in (0,t) $
it holds that
\begin{equation}
\label{eq:integrability.stoch.integrand}
\begin{split}
&
  \E\!\left[
      \textstyle
      \sum\limits_{ n \in \mathbb{Z} }
      \displaystyle
      \int_0^s 
      \left\|
      e^{ ( t - s ) A } 
      B\!\left( 
        e^{ ( s - u ) A } B( X_u ) \, b_n 
      \right) \!
      b_k
      \right\|^2_{
        H_{-r}
      } 
      du
  \right]
\\ & \leq
  \left\|
    e^{ ( t - s ) A } 
  \right\|_{
    L( H_{ -r-1 }, H_{-r} )
  }^2
  \E\!\left[
      \textstyle
      \sum\limits_{ n \in \mathbb{Z} }
      \displaystyle
      \int_0^s 
      \left\|
      B\!\left( 
        e^{ ( s - u ) A } B( X_u ) \, b_n 
      \right)\! 
      b_k
      \right\|^2_{
        H_{-r-1}
      } 
      du
  \right]
\\ & \leq
    \kappa_k
    \left\|
      e^{ ( t - s ) A } 
    \right\|_{
      L( H_{ - r - 1 }, H_{-r} )
    }^2
  \E\!\left[
      \textstyle
      \sum\limits_{ n \in \mathbb{Z} }
      \displaystyle
      \int_0^s 
      \left\|
        e^{ ( s - u ) A } B( X_u ) \, b_n 
      \right\|^2_{ H_{ - r } }
      du
  \right]
\\ & =
    \kappa_k
    \left\|
      e^{ ( t - s ) A } 
    \right\|_{
      L( H_{ - r - 1 }, H_{-r} )
    }^2
  \E\!\left[
      \int_0^s 
      \left\|
        e^{ ( s - u ) A } B( X_u ) 
      \right\|^2_{ HS( H, H_{ -r } ) }
      du
  \right]
  < \infty
\end{split}
\end{equation}
and 
\begin{equation}
\begin{split}
\label{eq:integrability.initial}
&
  \E\!\left[
  \left\|
  e^{ (t-s)A } B( e^{sA} \xi ) b_k
  \right\|^2_{ H_{-r} }
  \right]
\leq
    \kappa_k
    \left\|
      e^{ ( t - s ) A } 
    \right\|^2_{
      L( H_{ -r-1 }, H_{-r} )
    }
  \,
  \E\!\left[
  \left\|
  e^{ sA } \xi
  \right\|^2_{ H_{ -r } }
  \right] 
  < \infty
  .
\end{split}
\end{equation}
Combining~\eqref{eq:counterexample_2} with~\eqref{eq:integrability.stoch.integrand}--\eqref{eq:integrability.initial} proves that
for all 
$ k \in \mathbb{Z} $, $ t \in (0,T] $, $ s \in (0,t) $
it holds that
\begin{equation}
\begin{split}
& 
  \E\!\left[
    \left\|
      e^{ ( t - s ) A } B( X_s ) b_k
    \right\|^2_{ H_{-r} }
  \right]
\geq
    \E\!\left[
    \left\|
      e^{ ( t - s ) A } 
      B\!\left( 
        e^{ s A } \xi 
      \right) \! b_k
    \right\|^2_{ H_{-r} }
    \right]
\\ & \quad +
    2\,
    \E\!\left[
    \left<
      e^{ ( t - s ) A } 
      B\!\left( 
        e^{ s A } \xi 
      \right)
      b_k
      ,
      \int_0^s 
      e^{ ( t - s ) A } 
      B\!\left( 
        e^{ ( s - u ) A } B( X_u ) \, dW_u 
      \right) 
      b_k
    \right>_{ \! H_{-r} }
    \right]
\\&=
    \E\!\left[
    \left\|
      e^{ ( t - s ) A } 
      B\!\left( 
        e^{ s A } \xi 
      \right) b_k
    \right\|^2_{ H_{-r} }
    \right]
\\ & \quad +
    2\,
    \E\!\left[
    \left<
      e^{ ( t - s ) A } 
      B\!\left( 
        e^{ s A } \xi 
      \right)
      b_k
      ,
      \E\!\left[\left.
      \int_0^s 
      e^{ ( t - s ) A } 
      B\!\left( 
        e^{ ( s - u ) A } B( X_u ) \, dW_u 
      \right)
      b_k
      \right|
      \mathcal{F}_0
      \right]
    \right>_{ \! H_{-r} }
    \right]
\\ & 
=
    \E\!\left[
    \left\|
      e^{ ( t - s ) A } 
      B\!\left( 
        e^{ s A } \xi 
      \right) b_k
    \right\|^2_{ H_{-r} }
    \right]
    .
\end{split}
\end{equation}
Combining this with~\eqref{eq:counterexample_last_step}
ensures that
for all $ t \in (0,T] $
it holds that
\begin{equation}
\label{eq:counterexample_xi_estimate}
\begin{split}
&
  \infty >
  \E\!\left[
    \left\| X_t \right\|_{ H_{-r} }^2
  \right]
\geq
  \E\!\left[
  \| e^{ t A } \xi \|^2_{ H_{-r} }
  \right]
  +
  {\smallsum\limits_{ k \in \mathbb{Z} }}
  \int_0^t
    \E\!\left[
    \left\|
      e^{ ( t - s ) A } 
      B\!\left( 
        e^{ s A } \xi 
      \right) b_k
    \right\|^2_{ H_{-r} }
    \right]
  ds
\\ & \geq
  \int_0^t
    \E\!\left[
    \left\|
      e^{ ( t - s ) A } 
      B\!\left( 
        e^{ s A } \xi 
      \right) 
    \right\|^2_{ HS( H, H_{-r} ) }
    \right]
  ds
=
  \int_0^t
    \E\!\left[
    \left\|
      ( \eta - A )^{-r}
      e^{ ( t - s ) A } 
      B\!\left( 
        e^{ s A } \xi 
      \right) 
    \right\|^2_{ HS( H ) }
    \right]
  ds
\\&=
  \int_0^t
    \E\!\left[
    \left\|
      B\!\left( 
        e^{ s A } \xi 
      \right)
      e^{ ( t - s ) A } 
      ( \eta - A )^{-r}
    \right\|^2_{ HS( H ) }
    \right]
  ds
\\ & =
  \sum_{ n \in \mathbb{Z} }
  \int_0^t
    \E\!\left[
    \left\|
      B( 
        e^{ s A } \xi 
      )
      \,
      e^{ ( t - s ) A } 
      \,
      ( \eta - A )^{-r} b_n
    \right\|^2_H 
    \right]
  ds
\\&\geq
  \int_0^t
    \E\!\left[
    \left\|
      B( 
        e^{ s A } \xi 
      )
      \,
      e^{ ( t - s ) A } 
      \,
      ( \eta - A )^{-r} b_0
    \right\|^2_H 
    \right]
  ds
\\ & =
  \eta^{ -2 r }
  \int_0^t
    \E\!\left[
    \left\|
      B( 
        e^{ s A } \xi 
      )
      \,
      b_0
    \right\|^2_H 
    \right]
  ds
  =
  \eta^{ -2 r }
  \int_0^t
    \E\!\left[
    \|
      e^{ s A } \xi 
    \|^2_H 
    \right]
  \, ds
\\ & =
  \eta^{ -2 r }
  \sum_{ n \in \mathbb{Z} }
  \int_0^t
    \E\!\left[
    \left|
      \left< 
        e^{ s A }
        b_n ,
        \xi 
      \right>_H
    \right|^2
    \right]
  ds
=
  \eta^{ -2 r }
  \sum_{ n \in \mathbb{Z} }
  \int_0^t
    e^{
      - 2 ( \nu n^2  + \eta ) s
    }
    \,
    e^{ 2 \eta s }
    \,
    \E\!\left[
    \left|
      \left< 
        b_n ,
        \xi 
      \right>_H
    \right|^2
    \right]
  ds
\\&\geq
  \eta^{ -2 r }
  \sum_{ n \in \mathbb{Z} }
  \int_0^t
    e^{
      - 2 ( \nu n^2  + \eta ) s
    }
    \,
    \E\!\left[
    \left|
      \left< 
        b_n ,
        \xi 
      \right>_H
    \right|^2
    \right]
  ds
=
  \eta^{ -2 r }
  \sum_{ n \in \mathbb{Z} }
  \frac{
    (
      1 - e^{ - 2 ( \nu n^2  + \eta ) t }
    )
    \,
    \E\!\left[
    \left|
      \left< 
        b_n ,
        \xi 
      \right>_H
    \right|^2
    \right]
  }{
    2 \, ( \nu n^2  + \eta )
  }
\\&=
  \frac{
    1
  }{ 2 \eta^{ 2r } }
  \sum_{ n \in \mathbb{Z} }
    \big(
      1 - e^{ - 2 ( \nu n^2  + \eta ) t }
    \big)
    \,
    \E\!\left[
    \left|
      \left< 
        ( \eta - A )^{ - 1 / 2 } b_n ,
        \xi 
      \right>_H
    \right|^2
    \right]
\\&\geq
  \frac{ 
    \big(
      1 - e^{ - 2 \eta t }
    \big)
  }{ 2 \eta^{ 2r } }
  \sum_{ n \in \mathbb{Z} }
    \E\!\left[
    \left|
      \left< 
        ( \eta - A )^{ -1 / 2 } b_n ,
        \xi 
      \right>_H
    \right|^2
    \right]
  .
\end{split}
\end{equation}
In particular, we obtain that 
$
    \E\big[
  \sum_{ n \in \mathbb{Z} }
    \big|
      \left< 
        ( \eta - A )^{ -1 / 2 } b_n ,
        \xi 
      \right>_H
    \big|^2
    \big]
  < \infty
$.
Therefore, it holds that
$
  \P\!\left(
  \xi \in H_{ - 1 / 2 }
  \right)
  =1
$.
This and \eqref{eq:counterexample_xi_estimate} complete 
the proof of Proposition~\ref{prop:counterexampleI}.
\end{proof}

\begin{prop}[Positive example]
\label{prop:positive.example}
Assume the setting in Subsection~\ref{sec:setting.examples}, 
let 
$ k \in \N $, 
$
  \xi \in \mathcal{M}\big( 
    \mathcal{F}_0, 
    \mathcal{B}( H_\delta )
  \big)
$,
$
  (L_i)_{ i \in \{ 1,2,\ldots, k \} } \subseteq L(H)
$, 
$
  B \in L( H , HS( \R^k, H ) )
$
satisfy for all 
$ i,j \in \{ 1,2,\ldots, k \} $, 
$
  v \in H
$, 
$ u \in D(A) $, 
$ \mathbf{y} = ( y_1,y_2,\ldots,y_k ) \in \R^k $ 
that 
$ L_i( D(A) ) \subseteq D(A) $, 
$
  L_i L_j u - L_j L_i u
  =
  L_i A u - A L_i u
  =
  0
$, 
and 
$
  B(v)\mathbf{y}
  =
  \sum^k_{ l=1 }
  y_l
  L_l v
$, 
let
$
  W
  =
  ( W^{(1)}, W^{(2)}, 
  \ldots, W^{(k)} )
  \colon 
$
$
  [0,T] \times \Omega \rightarrow \R^k
$ 
be a $ k $-dimensional standard $ ( \mathcal{F}_t )_{ t \in [0,T] } $-Brownian motion 
with continuous sample paths,
and let 
$
  X \colon [0,T] \times \Omega \to H_\delta
$
satisfy for all $ t \in [0,T] $
that
\begin{equation}
  X_t = 
  \exp\!\big(
    tA +
    {\smallsum^k_{ i=1 }}
    \big[
    W^{(i)}_t L_i
    -\tfrac{1}{2} t (L_i)^2
  \big]\big)
  \,
  \xi
  .
\end{equation}
Then 
$ X $ has continuous sample paths and 
for all 
$ r \in \R $, 
$ t \in [0,T] $ 
it holds $\P$-a.s.\ that 
$
  \int^t_0
    \| e^{ ( t - s ) A } B( X_s ) \|^2_{ HS( \R^k, H_{ r } ) } 
  \, ds
  < \infty
$
and 
$
  X_t = 
  e^{ t A } \xi 
  + 
  \int_0^t e^{ ( t - s ) A } B( X_s ) \, dW_s
$.
\end{prop}
\begin{proof}
Throughout this proof 
let $ r \in [ 0, \infty ) $ 
and 
let 
$
  \varphi
  \in \mathcal{C}\big( 
  [ 0, T ] \times \R^k \times H_{ r }
  ,
  H_{ r }
  \big)
$
be the mapping with the property that for all 
$ t \in [0,T] $, 
$ \mathbf{y} = ( y_1, y_2, \ldots, y_k ) \in \R^k $, 
$ v \in H_{ r } $ 
it holds that 
$
  \varphi( t, \mathbf{y}, v )
  =
  \exp\!\big(
    {\sum^k_{ i=1 }}
    [
    y_i L_i
    -\tfrac{1}{2} t (L_i)^2
  ]\big)
  v
$. 
Note that the assumption that $ W $ has continuous sample paths ensures that $ X $ also has continuous sample paths.
Next observe that 
$
  \varphi \in \mathcal{C}^2\big( 
  [0,T] \times \R^k \times H_{ r }, H_{ r } \big)
$. 
It\^{o}'s formula 
(cf., e.g., Theorem~2.4 in Brze\'{z}niak, Van Neerven, Veraar \& Weis~\cite{bvvw08}) 
therefore implies that for all $ t \in (0,T] $ 
it holds $\P$-a.s.\ that 
\begin{equation}
\label{eq:positive.integrability}
\begin{split}
&
  \int^t_0
  \big\|
    \big(
    \tfrac{\partial}{ \partial \mathbf{y} }
    \varphi
    \big)( s, W_s, e^{ tA } \xi )
  \big\|^2_{ HS( \R^k, H_{ r } ) }
  \,\diffns{s}
=
  \int^t_0
  \smallsum^k_{ i=1 }
  \big\|
    \big(
    \tfrac{\partial}{\partial y_i}
    \varphi
    \big)
    ( s, W_s, e^{ t A } \xi )
  \big\|^2_{ H_r }
  \,\diffns{s}
\\&=
  \int^t_0
  \smallsum^k_{ i=1 }
  \big\|
  e^{ (t-s)A }
  \big(
    \tfrac{\partial}{\partial y_i}
    \varphi
  \big)
  ( s, W_s, e^{ s A } \xi )
  \big\|^2_{ H_r }
  \,\diffns{s}
=
  \int^t_0
  \big\|
    e^{ (t-s)A }
    B(X_s)
  \|^2_{ HS( \R^k, H_{ r } ) }
  \,\diffns{s}
  < \infty
\end{split}
\end{equation}
and 
\begin{equation}
\label{eq:positive.solution}
\begin{split}
  X_t&=
  \varphi( t, W_t, e^{ tA } \xi )
  =
  \varphi( 0, 0, e^{ tA } \xi )
  +
  \int^t_0
    \big(
    \tfrac{\partial}{ \partial s }
    \varphi
    \big)( u, W_u, e^{ tA } \xi )
  \,\diffns{u}
\\&\quad+
  \int^t_0
    \big(
    \tfrac{\partial}{ \partial \mathbf{y} }
    \varphi
    \big)( s, W_s, e^{ tA } \xi )
  \,\diffns{ W_s }
+
  \frac{1}{2}
  \sum^k_{ i=1 }
  \int^t_0
    \big(
    \tfrac{\partial^2}{ \partial y^2_i }
    \varphi
    \big)( s, W_s, e^{ tA } \xi )
  \,\diffns{ s }
\\&=
  e^{ tA } \xi
  -
  \int^t_0
  \frac{1}{2}
  \sum^k_{ i=1 }
  (L_i)^2
  \varphi( s, W_s, e^{ tA } \xi )
  \,\diffns{s}
  +
  \int^t_0
  e^{ (t-s)A }
  B( X_s )
  \,\diffns{W_s}
\\&\quad+
  \frac{1}{2}
  \sum^k_{ i=1 }
  \int^t_0
  (L_i)^2
  \varphi( s, W_s, e^{ tA } \xi )
  \,\diffns{s}
  =
  e^{ tA } \xi
  +
  \int^t_0
  e^{ (t-s)A }
  B( X_s )
  \,\diffns{W_s}
  .
\end{split}
\end{equation}
Combining this and~\eqref{eq:positive.integrability} completes the proof of Proposition~\ref{prop:positive.example}.
\end{proof}

\subsection{Stochastic heat equations with nonlinear multiplicative noise}
\label{sec:counterexample.II}

\begin{prop}
\label{prop:counterexampleII}
Assume the setting in Subsection~\ref{sec:setting.examples}, 
let 
$ r, \beta \in [0,\infty) $,   
$ w \in H_{-\beta} \setminus \{0\} $, 
$
  \xi \in \mathcal{M}\big( \mathcal{F}_0, \mathcal{B}( H_\delta ) \big)
$,
$
  B \in \mathcal{C}( H , HS( \R, H_{-\beta} ) )
$
satisfy for all 
$
  v \in H
$,
$ u \in \R $
that
$
  B( v ) u = u \left\| v \right\|_H w
$,
let
$
  W \colon [ 0, T ] \times \Omega \to \R
$ 
be a standard $ ( \mathcal{F}_t )_{ t \in [0,T] } $-Brownian motion,
and let 
$
  X \colon [0,T] \times \Omega \to H_{ \delta }
$
be an $ ( \mathcal{F}_t )_{ t \in [0,T] } $-predictable stochastic 
process 
which fulfills for all $ t \in (0,T] $
that
$
  X( (0,T] \times \Omega ) \subseteq H
$,
that
$
  \E\big[ 
    \| e^{ t A } \xi \|^2_{ H_{-r} } 
  \big] 
  +
  \int_0^t 
  \E\big[ 
    \| e^{ ( t - s ) A } B( X_s ) \|^2_{ HS( \R, H_{-r} ) } 
  \big] 
  \, ds
  < \infty
$,
and $ \P $-a.s.\ that
$
  X_t = e^{ t A } \xi + \int_0^t e^{ ( t - s ) A } B( X_s ) \, dW_s
$.
Then for all $ t \in (0,T] $ it holds that 
$
  \P\!\left(
  \xi \in H_{ - 1 / 2 }
  \right)
  =
  1
$ 
and 
\begin{equation}
  2^{-1/2}\,
  e^{ -|\eta| t }\, 
    \big(
      1 - e^{ - 2 [ \eta - \sup( \sigma_p(A) ) ] t }
    \big)^{ 1/2 }
    \,
    \|
      e^{ t A } w
    \|_{ H_{-r} }\,
  \|
    \xi
  \|_{ \lpn{2}{\P}{H_{ - 1 / 2 }} }
\leq 
  \|
    X_t
  \|_{ \lpn{2}{\P}{H_{ - r }} }
< 
  \infty
  .
\end{equation}
\end{prop}

\begin{proof}
Throughout this proof 
let $ \mathbb{B} \subseteq H $ 
be an orthonormal basis of $ H $ 
and let 
$ \lambda \colon \mathbb{B} \to \R $ 
be a mapping 
which satisfies 
$ \sup_{ b \in \mathbb{B} } -\lambda_b < \eta $, 
which satisfies 
$
  D(A)
  =
  \{
    v \in H
    \colon 
    \sum_{ b \in \mathbb{B} }
    |\lambda_b\left< b, v \right>_H|^2
    < \infty
  \}
$, 
and which satisfies for all 
$ v \in D(A) $ 
that 
$
  Av
  =
    -
    \sum_{ b \in \mathbb{B} }
    \lambda_b
    \left< b, v \right>_H
    b  
$.
Note that 
for all 
$ t \in (0,T] $
it holds $ \P $-a.s.\ that
\begin{equation}
\label{eq:squared.solutionII}
\begin{split}
  \left\| X_t \right\|_{ H_{-r} }^2
=&
  \left\| e^{ t A } \xi \right\|^2_{ H_{-r} }
  +
  2
  \left<
    e^{ t A } \xi
    ,
    \int_0^t
    e^{ ( t - s ) A } B( X_s ) \,
    dW_s
  \right>_{ H_{-r} }
\\&+
  \left\|
    \int_0^t
    e^{ ( t - s ) A } B( X_s ) \, dW_s
  \right\|^2_{ H_{-r} }
  .
\end{split}
\end{equation}
Equation~\eqref{eq:squared.solutionII}
together with It\^{o}'s isometry and 
the assumption that 
$ 
  \forall \, t \in [0,T] \colon
  \E\big[ 
    \| e^{ t A } \xi \|^2_{ H_{-r} } 
  \big] 
  +
  \int_0^t 
    \E\big[ 
      \| e^{ ( t - s ) A } B( X_s ) \|^2_{ HS( \R, H_{-r} ) } 
    \big] 
  \, ds
$
$
  < \infty
$
hence prove that
for all $ t \in (0,T] $
it holds that
\begin{equation}
\label{eq:counterexampleII_last_step}
\begin{split}
&
  \E\!\left[
    \left\| X_t \right\|_{ H_{-r} }^2
  \right]
=
  \E\!\left[ \| e^{ t A } \xi \|^2_{H_{-r}} \right]
  +
  2\,
  \E\!\left[
  \left<
    e^{ t A } \xi
    ,
      \int_0^t
      e^{ ( t - s ) A } B( X_s ) \,
      dW_s
  \right>_{ \! H_{-r} }
  \right]
\\&\quad+
  \E\!\left[
  \left\|
    \int_0^t
    e^{ ( t - s ) A } B( X_s ) \, dW_s
  \right\|^2_{ H_{-r} }
  \right]
\\ & =
  \E\!\left[ \| e^{ t A } \xi \|^2_{H_{-r}} \right]
  +
  2\,
  \E\!\left[\left.
  \left<
    e^{ t A } \xi
    ,
    \E\!\left[
      \int_0^t
      e^{ ( t - s ) A } B( X_s ) \,
      dW_s
    \right|\mathcal{F}_0
    \right]
  \right>_{ \! H_{-r} }
  \right]
\\&\quad+
  \int_0^t
  \E\!\left[
    \left\|
      e^{ ( t - s ) A } B( X_s ) 
    \right\|^2_{ HS( \R, H_{-r} ) }
  \right]
  ds
\\ & =
  \E\!\left[ \| e^{ t A } \xi \|^2_{H_{-r}} \right]
  +
  \int_0^t
  \E\!\left[
    \left\|
      e^{ ( t - s ) A } B( X_s ) 
    \right\|^2_{ HS( \R, H_{-r} ) }
  \right]
  ds
  < \infty
  .
\end{split}
\end{equation}
Next we note that for all $ t \in (0,T] $, $ s \in (0,t) $
it holds $ \P $-a.s.\ that
\begin{equation}
\label{eq:counterexample_2_B}
\begin{split}
&
    \left\|
      e^{ ( t - s ) A } B( X_s ) 
    \right\|^2_{ HS( \R, H_{-r} ) }
= 
    \left\|
      e^{ ( t - s ) A } 
      B\!\left( 
        e^{ s A } \xi + \int_0^s e^{ ( s - u ) A } B( X_u ) \, dW_u 
      \right) 
    \right\|^2_{ HS( \R, H_{-r} ) }
\\&= 
      \left\|
        e^{ s A } \xi + \int_0^s e^{ ( s - u ) A } B( X_u ) \, dW_u 
      \right\|^2_H
    \left\| 
     e^{ ( t - s ) A } w 
    \right\|^2_{ H_{-r} }
\\ & =
    \left\|
      e^{ ( t - s ) A } 
      B\!\left( 
        e^{ s A } \xi 
      \right) 
    \right\|^2_{ HS( \R, H_{-r} ) }
  +
    \left\|
      e^{ ( t - s ) A } 
      B\!\left( 
        \int_0^s e^{ ( s - u ) A } B( X_u ) \, dW_u 
      \right) 
    \right\|^2_{ HS( \R, H_{-r} ) }
\\ &
  \quad+
    2
    \left\|
      e^{ ( t - s ) A } w
    \right\|^2_{ H_{-r} }    
    \left<
        e^{ s A } \xi 
      , 
        \int_0^s e^{ ( s - u ) A } B( X_u ) \, dW_u 
    \right>_{ \!H }
\\ & 
  \geq
    \left\|
      e^{ ( t - s ) A } 
      B\!\left( 
        e^{ s A } \xi 
      \right) 
    \right\|^2_{ HS( \R, H_{-r} ) }
+
    2
    \left\|
      e^{ ( t - s ) A } w
    \right\|^2_{ H_{-r} }    
    \left<
        e^{ s A } \xi 
      , 
        \int_0^s e^{ ( s - u ) A } B( X_u ) \, dW_u 
    \right>_{ \!H }
\\ & 
  =
    \left\|
      e^{ ( t - s ) A } 
      B\!\left( 
        e^{ s A } \xi 
      \right) 
    \right\|^2_{ HS( \R, H_{-r} ) }
\\&\quad+
    2
    \left\|
      e^{ ( t - s ) A } w
    \right\|^2_{ H_{-r} }    
    \left<
        (\eta-A)^r e^{ s A } \xi 
      , 
        \int_0^s (\eta-A)^{-r} e^{ ( s - u ) A } B( X_u ) \, dW_u 
    \right>_{ \!H }
    .
\end{split}
\end{equation}
In addition, the assumption that 
$
  \forall\, t \in (0,T]
  \colon
  \E\big[
  \| e^{ tA } \xi \|^2_{ H_{-r} }
  \big]
  < \infty
$ 
implies that for all $ t \in (0,T] $ 
it holds that 
\begin{equation}
  \E\big[
  \| e^{ tA } \xi \|^2_{ H_r }
  \big]
  \leq
  \| e^{ \frac{t}{2} A } \|^2_{ L( H_{-r}, H_r ) }
  \,
  \E\big[
  \| e^{ \frac{t}{2} A } \xi \|^2_{ H_{-r} }
  \big]  
  < \infty
  .
\end{equation}
It\^{o}'s isometry and the assumption that
$ 
  \forall \, t \in (0,T] \colon
  \int_0^t 
    \E\big[ 
      \| e^{ ( t - s ) A } B( X_s ) \|^2_{ HS( \R, H_{-r} ) } 
    \big] 
  \, ds
  < \infty
$
hence prove that for all $ t \in (0,T] $, $ s \in (0,t) $
it holds that
\begin{equation}
\label{eq:nonlinear.diffusion.lb}
\begin{split}
& 
  \E\!\left[
    \left\|
      e^{ ( t - s ) A } B( X_s ) 
    \right\|^2_{ HS( \R, H_{-r} ) }
  \right]
\geq
    \E\!\left[
    \left\|
      e^{ ( t - s ) A } 
      B\!\left( 
        e^{ s A } \xi 
      \right) 
    \right\|^2_{ HS( \R, H_{-r} ) }
    \right]
\\ & \quad +
    2
    \left\|
      e^{ ( t - s ) A } w
    \right\|^2_{ H_{-r} }    
    \E\!\left[
    \left<
      (\eta-A)^r e^{ s A } \xi 
      ,
      \int_0^s 
        (\eta-A)^{-r} e^{ ( s - u ) A } B( X_u ) \, dW_u 
    \right>_{ \!H }
    \right]
\\&=
    \E\!\left[
    \left\|
      e^{ ( t - s ) A } 
      B\!\left( 
        e^{ s A } \xi 
      \right) 
    \right\|^2_{ HS( \R, H_{-r} ) }
    \right]
\\ & \quad +
    2
    \left\|
      e^{ ( t - s ) A } w
    \right\|^2_{ H_{-r} }    
    \E\!\left[
    \left<
      (\eta-A)^r e^{ s A } \xi 
      ,
      \E\!\left[\left.
      \int_0^s  
        (\eta-A)^{-r} e^{ ( s - u ) A } B( X_u ) \, dW_u 
      \right|
      \mathcal{F}_0
      \right]
    \right>_{ \!H }
    \right]
\\ & 
=
    \E\!\left[
    \left\|
      e^{ ( t - s ) A } 
      B\!\left( 
        e^{ s A } \xi 
      \right) 
    \right\|^2_{ HS( \R, H_{-r} ) }
    \right]
    .
\end{split}
\end{equation}
Furthermore, we observe that for all 
$ t \in (0,T] $, 
$ s \in (0,t) $
it holds that 
\begin{equation}
\label{eq:smooth.lb}
  \| e^{tA} w \|_{ H_{-r} }
  \leq
  \| e^{sA} \|_{ L(H) }\,
  \| e^{ (t-s)A } w \|_{ H_{-r} }
  \leq
  e^{ \eta s }\,
  \| e^{ (t-s)A } w \|_{ H_{-r} }
  \leq
  e^{ \max\{\eta,0\} t }\,
  \| e^{ (t-s)A } w \|_{ H_{-r} }
  .
\end{equation}
Combining~\eqref{eq:nonlinear.diffusion.lb} with~\eqref{eq:counterexampleII_last_step} and~\eqref{eq:smooth.lb}
ensures that
for all $ t \in (0,T] $
it holds that
\begin{equation}
\label{eq:counterexampleII_xi_estimate}
\begin{split}
&
  \infty
  >
  \E\!\left[
    \left\| X_t \right\|_{ H_{-r} }^2
  \right]
\geq
  \E\!\left[
  \| e^{ t A } \xi \|^2_{ H_{-r} }
  \right]
  +
  \int_0^t
    \E\!\left[
    \left\|
      e^{ ( t - s ) A } 
      B\!\left( 
        e^{ s A } \xi 
      \right) 
    \right\|^2_{ HS( \R, H_{-r} ) }
    \right]
  ds
\\ & \geq
  \int_0^t
    \E\!\left[
    \left\|
      e^{ ( t - s ) A } 
      B\!\left( 
        e^{ s A } \xi 
      \right) 
    \right\|^2_{ HS( \R, H_{-r} ) }
    \right]
  ds
=
  \int_0^t
    \left\|
      e^{ (t-s) A } w
    \right\|^2_{ H_{-r} }\,
    \E\!\left[
    \|
      e^{ s A } \xi 
    \|^2_H 
    \right]
  \, ds
\\&\geq
    e^{ -2 \max\{\eta,0\} t }
    \left\|
      e^{ t A } w
    \right\|^2_{ H_{-r} }
  \int_0^t
    \E\!\left[
    \|
      e^{ s A } \xi 
    \|^2_H 
    \right]
  \, ds
\\&=
    e^{ -2 \max\{\eta,0\} t }
    \left\|
      e^{ t A } w
    \right\|^2_{ H_{-r} }
  \sum_{ b \in \mathbb{B} }
  \int_0^t
    \E\!\left[
    \left|
      \left< 
        e^{ s A }
        b ,
        \xi 
      \right>_H
    \right|^2
    \right]
  ds
\\&=
    e^{ -2 \max\{\eta,0\} t }
    \left\|
      e^{ t A } w
    \right\|^2_{ H_{-r} }
  \sum_{ b \in \mathbb{B} }
  \int_0^t
    e^{
      - 2 ( \lambda_b + \eta ) s
    }
    \,
    e^{ 2 \eta s }
    \,
    \E\!\left[
    \left|
      \left< 
        b ,
        \xi 
      \right>_H
    \right|^2
    \right]
  ds
\\&\geq
    e^{ -2 |\eta| t }
    \left\|
      e^{ t A } w
    \right\|^2_{ H_{-r} }
  \sum_{ b \in \mathbb{B} }
  \int_0^t
    e^{
      - 2 ( \lambda_b + \eta ) s
    }
    \,
    \E\!\left[
    \left|
      \left< 
        b ,
        \xi 
      \right>_H
    \right|^2
    \right]
  ds
\\&=
    e^{ -2 |\eta| t }
    \left\|
      e^{ t A } w
    \right\|^2_{ H_{-r} }
  \sum_{ b \in \mathbb{B} }
  \frac{
    (
      1 - e^{ - 2 ( \lambda_b + \eta ) t }
    )
    \E\!\left[
    \left|
      \left< 
        b ,
        \xi 
      \right>_H
    \right|^2
    \right]
  }{
    2 \, ( \lambda_b + \eta )
  }
\\&=
  \frac{
   \|
      e^{ t A } w
    \|^2_{ H_{-r} }
  }{ 2 e^{ 2 |\eta| t } }
  \sum_{ b \in \mathbb{B} }
    \big(
      1 - e^{ - 2 ( \lambda_b + \eta ) t }
    \big)\,
    \E\!\left[
    \left|
      \left< 
        ( \eta - A )^{ - 1 / 2 } b ,
        \xi 
      \right>_H
    \right|^2
    \right]
\\&\geq
  \frac{ 
    \big(
      1 - e^{ - 2 ( \inf_{ b \in \mathbb{B} } \lambda_b + \eta ) t }
    \big)
    \|
      e^{ t A } w
    \|^2_{ H_{-r} }
  }{ 2 e^{ 2 |\eta| t } }
  \sum_{ b \in \mathbb{B} }
    \E\!\left[
    \left|
      \left< 
        ( \eta - A )^{ -1 / 2 } b ,
        \xi 
      \right>_H
    \right|^2
    \right]
  .
\end{split}
\end{equation}
This and the assumption that 
$ w \neq 0 $, 
in particular, 
assure that 
$
    \E\big[
  \sum_{ b \in \mathbb{B} }
    \left|
      \left< 
        ( \eta - A )^{ -1 / 2 } b ,
        \xi 
      \right>_H
    \right|^2
    \big]
  < \infty
$.
Hence, we obtain that
$
  \P\!\left(
  \xi \in H_{ - 1 / 2 }
  \right)
  =1
$.
This and \eqref{eq:counterexampleII_xi_estimate} complete 
the proof of Proposition~\ref{prop:counterexampleII}.
\end{proof}

\subsection{Nonlinear heat equations}
\label{sec:counterexample.III}

\begin{prop}
\label{prop:counterexample2}
Let 
$
  \left(
    H,
    \left\| \cdot \right\|_H,
    \left< \cdot, \cdot \right>_H
  \right)
$ 
be a separable $\R$-Hilbert space 
with $ \#_H(H) > 1 $, 
let $ \mathbb{B} \subseteq H $ be an orthonormal basis of $H$, 
let 
$
  ( \Omega, \mathcal{F}, \P )
$ 
be a probability space, 
let 
$ T \in (0,\infty) $, 
$ \eta, \delta \in \R $,  
let 
$ \lambda \colon \mathbb{B} \to \R $ 
be a function which satisfies 
$ \sup_{ b \in \mathbb{B} } (-\lambda_b) < \eta $, 
let $ A \colon D(A) \subseteq H \rightarrow H $ 
be a linear operator  
which satisfies 
$
  D( A ) = \left\{
    v \in H \colon
    \sum_{ b \in \mathbb{B} }
    \left|
      \lambda_b\left< b , v \right>_H
    \right|^2
    < \infty
  \right\}
$ 
and which satisfies 
for all $ v \in D(A) $
that
$
  A v
  =
  -
  \sum_{ b\in \mathbb{B} }
  \lambda_b
  \left< b, v \right>_H
  b
$,
let
$
  (
    H_r
    ,
    \left\| \cdot \right\|_{ H_r }
    ,
    \left< \cdot , \cdot \right>_{ H_r }
  )
$,
$ r \in \R $,
be a family of interpolation spaces associated to
$
  \eta - A
$,
let
$ w \in H $, 
$ b_0 \in \mathbb{B} $, 
$ \xi \in \mathcal{M}\big( \mathcal{F}, \mathcal{B}(H_\delta) \big) $, 
$ F \in \mathcal{C}(H,H) $  
satisfy for all 
$ v \in H $ 
that 
$ \left< b_0, w \right>_H > 0 $, 
$ w = \left< b_0, w \right>_H b_0 $, 
and  
$
  F( v ) = \left\| v \right\|_H w
$, 
and let
$ X \in \mathcal{M}\big( \mathcal{B}([0,T]) \otimes \mathcal{F}, \mathcal{B}(H_\delta) \big) $
satisfy for all $ t \in (0,T] $ that 
$ X( (0,T] \times \Omega ) \subseteq H $,   
$\P$-a.s.\ that 
$
  \int_0^t
  \|
    e^{ ( t - s ) A }
    F( X_s )
  \|_{ H_{ \delta } }
  \,
  ds
  < \infty
$,  
and $\P$-a.s.\ that 
$
  X_t
  =
  e^{ t A } \xi
  +
  \int_0^t
  e^{  ( t - s ) A }
  F( X_s ) \, ds
$.
Then for all $ t \in (0,T] $ it holds that
$
  \P\!\left(
  \xi \in H_{ - 1 }
  \right)
  =
  1
$
and $\P$-a.s.\ that  
\begin{equation}
\begin{split}
&
  \left< b_0, w \right>_H
  e^{ -(\lambda_{b_0}+|\eta|) t }
        \left[
          1
          -
          e^{ -( \inf_{ b \in \mathbb{B} } \lambda_b +\eta) t }
        \right]
  \left\|
    \xi - \left< b_0, \xi \right>_H b_0
  \right\|_{ H_{ - 1 } }
\\&\leq
  \left< b_0, X_t - e^{  t A } \xi \right>_H 
\leq
  \left\| X_t - e^{  t A } \xi \right\|_H
  < \infty
  .
\end{split}
\end{equation}
\end{prop}

\begin{proof}
Throughout this proof 
let $ P \in L( H_{ \min\{\delta,0\} } ) $
be the linear operator with the property that for all
$ v \in H $ it holds that
$
  P( v ) = v - \left< b_0, v \right>_H b_0
$.
We observe that the assumption that
$ X( (0,T] \times \Omega ) \subseteq H $ 
implies that for all $ t \in (0,T] $
it holds $\P$-a.s.\ that
\begin{equation}
\label{eq:project.lb}
\begin{split}
&
  \infty > \left\| X_t - e^{  t A } \xi \right\|_H
  =
  \left< b_0, X_t - e^{  t A } \xi \right>_H
  =
  \int_0^t
  \left< b_0 ,
  e^{  ( t - s ) A }
  F( X_s )
  \right>_H
  ds
 \\ &
  =
  \int_0^t
  \left< b_0 ,
  e^{  ( t - s ) A }
  w
  \right>_H
  \left\| X_s \right\|_H
  ds
  =
  \int_0^t
  \left< b_0, w \right>_H
  e^{ -(\lambda_{b_0}+\eta) \left( t - s \right) }\,
  e^{ \eta (t-s) }
  \left\| X_s \right\|_H
  ds
\\&\geq
  \left< b_0, w \right>_H
  e^{ \min\{\eta,0\} t }
  \int_0^t
  e^{ -(\lambda_{b_0}+\eta) \left( t - s \right) }
  \left\| P X_s \right\|_H
  ds
\\&=
  \left< b_0, w \right>_H
  e^{ \min\{\eta,0\} t }
  \int_0^t
  e^{ -(\lambda_{b_0}+\eta) \left( t - s \right) }
  \,
  \| e^{ s A } P \xi \|_H
  \, ds
\\&\geq
  \left< b_0, w \right>_H
  e^{ -(\lambda_{b_0}+\max\{\eta,0\}) t }
  \int_0^t
  \| e^{ s A } P \xi \|_H
  \, ds
\\&=
  \left< b_0, w \right>_H
  e^{ -(\lambda_{b_0}+\max\{\eta,0\}) t }
  \int_0^t
  \left[
    \textstyle
    \sum_{ b \in \mathbb{B} \setminus \{b_0\} }
    \displaystyle
    \left|
      e^{ -(\lambda_b+\eta) s }\,
      e^{ \eta s }
      \left< b, \xi \right>_H
    \right|^2
  \right]^{ \nicefrac{ 1 }{ 2 } }
  ds
  .
\end{split}
\end{equation}
This and the Minkowski integral inequality imply that for all $ t \in (0,T] $
it holds $\P$-a.s.\ that
\begin{equation}
\label{eq:minkowski}
\begin{split}
&
  \infty > \left\| X_t - e^{  t A } \xi \right\|_H
  \geq \left< b_0, X_t - e^{  t A } \xi \right>_H
\\&\geq
  \left< b_0, w \right>_H
  e^{ -(\lambda_{b_0}+|\eta|) t }
  \left[
    \textstyle
    \sum_{ b \in \mathbb{B} \setminus \{b_0\} }
    \displaystyle
    \left|
      \textstyle
      \int_0^t
      \displaystyle
      \left|
      e^{ -(\lambda_b+\eta) s }
      \left< b, \xi \right>_H
      \right|
      ds
    \right|^2
  \right]^{ \nicefrac{ 1 }{ 2 } }
\\&=
  \left< b_0, w \right>_H
  e^{ -(\lambda_{b_0}+|\eta|) t }
  \left[
    \textstyle
    \sum_{ b \in \mathbb{B} \setminus \{b_0\} }
    \displaystyle
      \tfrac{
        \left[
          1
          -
          e^{ -(\lambda_b+\eta) t }
        \right]^2
        \left|
          \left< b, \xi \right>_H
        \right|^2
      }{
        \left|
          \lambda_b + \eta
        \right|^2
      }
  \right]^{ \nicefrac{ 1 }{ 2 } }
\\ &
  \geq
  \left< b_0, w \right>_H
  e^{ -(\lambda_{b_0}+|\eta|) t }
        \left[
          1
          -
          e^{ -( \inf_{ b \in \mathbb{B} } \lambda_b +\eta) t }
        \right]
  \left[
    \smallsum_{ b \in \mathbb{B} \setminus \{b_0\} }
        \left|
          \left< (\eta-A)^{-1} b, \xi \right>_H
        \right|^2
  \right]^{ \nicefrac{ 1 }{ 2 } }
  .
\end{split}
\end{equation}
The assumption that 
$ \left< b_0, w \right>_H > 0 $ 
hence implies that it holds $\P$-a.s.\ that 
\begin{equation}
    \smallsum_{ b \in \mathbb{B} }
        \left|
          \left< (\eta-A)^{-1} b, \xi \right>_H
        \right|^2
        < \infty
        .
\end{equation}
This ensures that
$
  \P\!\left(
  \xi \in H_{ - 1 }
  \right)
$
$
  =
  1
$.
This together with~\eqref{eq:minkowski} completes the proof of Proposition~\ref{prop:counterexample2}.
\end{proof}

\section*{Acknowledgments}

Stig Larsson and Christoph Schwab are gratefully acknowledged for some useful comments.
This project has been supported through the SNSF-Research project 200021\_156603 
"Numerical approximations of nonlinear stochastic ordinary and partial differential equations".

\bibliographystyle{acm}
\bibliography{Bib/bibfile}

\def\cprime{$'$} \def\cprime{$'$}
  \def\polhk#1{\setbox0=\hbox{#1}{\ooalign{\hidewidth
  \lower1.5ex\hbox{`}\hidewidth\crcr\unhbox0}}}
\begin{thebibliography}{10}

\bibitem{AnderssonLarsson2015}
{\sc Andersson, A., and Larsson, S.}
\newblock Weak convergence for a spatial approximation of the nonlinear
  stochastic heat equation.
\newblock {\em Math. Comp. 85}, 299 (2016), 1335--1358.

\bibitem{Brehier2014}
{\sc Br{\'e}hier, C.-E.}
\newblock Approximation of the invariant measure with an {E}uler scheme for
  stochastic {PDE}s driven by space-time white noise.
\newblock {\em Potential Anal. 40}, 1 (2014), 1--40.

\bibitem{BrehierKopec2016}
{\sc Br\'{e}hier, C.-E., and Kopec, M.}
\newblock Approximation of the invariant law of {SPDE}s: error analysis using a
  {P}oisson equation for a full-discretization scheme.
\newblock \mbox{doi}:\url{10.1093/imanum/drw030}.

\bibitem{b97b}
{\sc Brze{\'z}niak, Z.}
\newblock On stochastic convolution in {B}anach spaces and applications.
\newblock {\em Stochastics Stochastics Rep. 61}, 3-4 (1997), 245--295.

\bibitem{bvvw08}
{\sc Brze{\'z}niak, Z., van Neerven, J. M. A.~M., Veraar, M.~C., and Weis, L.}
\newblock It\^o's formula in {UMD} {B}anach spaces and regularity of solutions
  of the {Z}akai equation.
\newblock {\em J. Differential Equations 245}, 1 (2008), 30--58.

\bibitem{CarmonaMolchanov1994}
{\sc Carmona, R.~A., and Molchanov, S.~A.}
\newblock Parabolic {A}nderson problem and intermittency.
\newblock {\em Mem. Amer. Math. Soc. 108}, 518 (1994), viii+125.

\bibitem{Chen2014}
{\sc Chen, L., and Dalang, R.~C.}
\newblock H\"older-continuity for the nonlinear stochastic heat equation with
  rough initial conditions.
\newblock {\em Stoch. Partial Differ. Equ. Anal. Comput. 2}, 3 (2014),
  316--352.

\bibitem{chen2015}
{\sc Chen, L., and Dalang, R.~C.}
\newblock Moments and growth indices for the nonlinear stochastic heat equation
  with rough initial conditions.
\newblock {\em Ann. Probab. 43}, 6 (2015), 3006--3051.

\bibitem{ConusJentzenKurniawan2014arXiv}
{\sc Conus, D., Jentzen, A., and Kurniawan, R.}
\newblock Weak convergence rates of spectral {G}alerkin approximations for
  {SPDE}s with nonlinear diffusion coefficients.
\newblock {\em arXiv:1408.1108\/} (2014), 29 pages.

\bibitem{dz92}
{\sc Da~Prato, G., and Zabczyk, J.}
\newblock {\em Stochastic equations in infinite dimensions}, vol.~44 of {\em
  Encyclopedia of Mathematics and its Applications}.
\newblock Cambridge University Press, Cambridge, 1992.

\bibitem{DaPratoZabczyk1996}
{\sc Da~Prato, G., and Zabczyk, J.}
\newblock {\em Ergodicity for infinite-dimensional systems}, vol.~229 of {\em
  London Mathematical Society Lecture Note Series}.
\newblock Cambridge University Press, Cambridge, 1996.

\bibitem{Debussche2011}
{\sc Debussche, A.}
\newblock Weak approximation of stochastic partial differential equations: the
  nonlinear case.
\newblock {\em Math. Comp. 80}, 273 (2011), 89--117.

\bibitem{Gorenfloetal2014}
{\sc Gorenflo, R., Kilbas, A.~A., Mainardi, F., and Rogosin, S.~V.}
\newblock {\em Mittag-{L}effler functions, related topics and applications}.
\newblock Springer Monographs in Mathematics. Springer, Heidelberg, 2014.

\bibitem{h81}
{\sc Henry, D.}
\newblock {\em Geometric theory of semilinear parabolic equations}, vol.~840 of
  {\em Lecture Notes in Mathematics}.
\newblock Springer-Verlag, Berlin-New York, 1981.

\bibitem{Hoffmann1994}
{\sc Hoffmann-J{\o}rgensen, J.}
\newblock {\em Probability with a view toward statistics. {V}ol. {I}}.
\newblock Chapman \& Hall Probability Series. Chapman \& Hall, New York, 1994.

\bibitem{JentzenKloeden2011}
{\sc Jentzen, A., and Kloeden, P.~E.}
\newblock {\em Taylor approximations for stochastic partial differential
  equations}, vol.~83 of {\em CBMS-NSF Regional Conference Series in Applied
  Mathematics}.
\newblock Society for Industrial and Applied Mathematics (SIAM), Philadelphia,
  PA, 2011.

\bibitem{JentzenKurniawan2015arXiv}
{\sc Jentzen, A., and Kurniawan, R.}
\newblock {W}eak convergence rates for {E}uler-type approximations of
  semilinear stochastic evolution equations with nonlinear diffusion
  coefficients.
\newblock {\em arXiv:1501.03539\/} (2015), 51 pages.

\bibitem{jr12}
{\sc Jentzen, A., and R{\"o}ckner, M.}
\newblock Regularity analysis for stochastic partial differential equations
  with nonlinear multiplicative trace class noise.
\newblock {\em J. Differential Equations 252}, 1 (2012), 114--136.

\bibitem{Klenke2008}
{\sc Klenke, A.}
\newblock {\em Probability theory}.
\newblock Universitext. Springer-Verlag London Ltd., London, 2008.
\newblock A comprehensive course, Translated from the 2006 German original.

\bibitem{Kopec2014_PhD_Thesis}
{\sc Kopec, M.}
\newblock Quelques contributions \`{a} l'analyse num\'{e}rique d'\'{e}quations
  stochastiques.
\newblock {\em Ph{D} thesis, ENS Rennes\/} (2014), viii+189.

\bibitem{KruseLarsson2012}
{\sc Kruse, R., and Larsson, S.}
\newblock Optimal regularity for semilinear stochastic partial differential
  equations with multiplicative noise.
\newblock {\em Electron. J. Probab. 17\/} (2012), no. 65, 19.

\bibitem{Parthasarathy67}
{\sc Parthasarathy, K.~R.}
\newblock {\em Probability measures on metric spaces}.
\newblock Probability and Mathematical Statistics, No. 3. Academic Press, Inc.,
  New York-London, 1967.

\bibitem{PrevotRoeckner2007}
{\sc Pr{\'e}v{\^o}t, C., and R{\"o}ckner, M.}
\newblock {\em A concise course on stochastic partial differential equations},
  vol.~1905 of {\em Lecture Notes in Mathematics}.
\newblock Springer, Berlin, 2007.

\bibitem{rr93}
{\sc Renardy, M., and Rogers, R.~C.}
\newblock {\em An introduction to partial differential equations}, vol.~13 of
  {\em Texts in Applied Mathematics}.
\newblock Springer-Verlag, New York, 1993.

\bibitem{Rudin1976}
{\sc Rudin, W.}
\newblock {\em Principles of mathematical analysis}, third~ed.
\newblock McGraw-Hill Book Co., New York-Auckland-D\"usseldorf, 1976.
\newblock International Series in Pure and Applied Mathematics.

\bibitem{sy02}
{\sc Sell, G.~R., and You, Y.}
\newblock {\em Dynamics of evolutionary equations}, vol.~143 of {\em Applied
  Mathematical Sciences}.
\newblock Springer-Verlag, New York, 2002.

\bibitem{VanNeervenVeraarWeis2012}
{\sc van Neerven, J., Veraar, M., and Weis, L.}
\newblock Maximal {$L^p$}-regularity for stochastic evolution equations.
\newblock {\em SIAM J. Math. Anal. 44}, 3 (2012), 1372--1414.

\bibitem{VanNeervenVeraarWeis2008}
{\sc van Neerven, J. M. A.~M., Veraar, M.~C., and Weis, L.}
\newblock Stochastic evolution equations in {UMD} {B}anach spaces.
\newblock {\em J. Funct. Anal. 255}, 4 (2008), 940--993.

\bibitem{Wang2016481}
{\sc Wang, X.}
\newblock Weak error estimates of the exponential {E}uler scheme for
  semi-linear {SPDE}s without {M}alliavin calculus.
\newblock {\em Discrete Contin. Dyn. Syst. 36}, 1 (2016), 481--497.

\bibitem{WangGan2013_Weak_convergence}
{\sc Wang, X., and Gan, S.}
\newblock Weak convergence analysis of the linear implicit {E}uler method for
  semilinear stochastic partial differential equations with additive noise.
\newblock {\em J. Math. Anal. Appl. 398}, 1 (2013), 151--169.

\end{thebibliography}
\end{document}